\documentclass{amsart}
\usepackage[latin1]{inputenc}
\usepackage[T1]{fontenc}
\RequirePackage{calrsfs}
\DeclareSymbolFont{rsfscript}{OMS}{rsfs}{m}{b}
\DeclareSymbolFontAlphabet{\mathrsfs}{rsfscript}
\usepackage{amsfonts}
\usepackage{amsmath}
\usepackage{amssymb}
\usepackage{graphicx}
\usepackage{pstricks}
\usepackage{pst-grad}
\usepackage{multido}
\usepackage{amsthm}
\usepackage[absolute]{textpos}
\definecolor{purple}{rgb}{0.8,0.12,0.8}
\definecolor{orange}{rgb}{1.0,0.7,0.0}
\definecolor{pink}{rgb}{1,0.5,0.8}
\definecolor{blackg}{rgb}{0.1,0.25,0.1}
\definecolor{ForestGreen}{cmyk}{0.91,0,0.88,0.42}
\definecolor{Turquoise}{cmyk}{0.85,0,0.20,0}

\definecolor{GreenYellow}{cmyk}{0.15,0,0.69,0} 
\definecolor{Yellow}{cmyk}{0,0,1.,0} 
\definecolor{Goldenrod}{cmyk}{0,0.10,0.84,0} 
\definecolor{Dandelion}{cmyk}{0,0.29,0.84,0} 
\definecolor{Apricot}{cmyk}{0,0.32,0.52,0} 
\definecolor{Peach}{cmyk}{0,0.50,0.70,0} 
\definecolor{Melon}{cmyk}{0,0.46,0.50,0} 
\definecolor{YellowOrange}{cmyk}{0,0.42,1.,0} 
\definecolor{Orange}{cmyk}{0,0.61,0.87,0} 
\definecolor{BurntOrange}{cmyk}{0,0.51,1.,0} 
\definecolor{Bittersweet}{cmyk}{0,0.75,1.,0.24} 
\definecolor{RedOrange}{cmyk}{0,0.77,0.87,0} 
\definecolor{Mahogany}{cmyk}{0,0.85,0.87,0.35} 
\definecolor{Maroon}{cmyk}{0,0.87,0.68,0.32} 
\definecolor{BrickRed}{cmyk}{0,0.89,0.94,0.28} 
\definecolor{Red}{cmyk}{0,1.,1.,0} 
\definecolor{OrangeRed}{cmyk}{0,1.,0.50,0} 
\definecolor{RubineRed}{cmyk}{0,1.,0.13,0} 
\definecolor{WildStrawberry}{cmyk}{0,0.96,0.39,0} 
\definecolor{Salmon}{cmyk}{0,0.53,0.38,0} 
\definecolor{CarnationPink}{cmyk}{0,0.63,0,0} 
\definecolor{Magenta}{cmyk}{0,1.,0,0} 
\definecolor{VioletRed}{cmyk}{0,0.81,0,0} 
\definecolor{Rhodamine}{cmyk}{0,0.82,0,0} 
\definecolor{Mulberry}{cmyk}{0.34,0.90,0,0.02} 
\definecolor{RedViolet}{cmyk}{0.07,0.90,0,0.34} 
\definecolor{Fuchsia}{cmyk}{0.47,0.91,0,0.08} 
\definecolor{Lavender}{cmyk}{0,0.48,0,0} 
\definecolor{Thistle}{cmyk}{0.12,0.59,0,0} 
\definecolor{Orchid}{cmyk}{0.32,0.64,0,0} 
\definecolor{DarkOrchid}{cmyk}{0.40,0.80,0.20,0} 
\definecolor{Purple}{cmyk}{0.45,0.86,0,0} 
\definecolor{Plum}{cmyk}{0.50,1.,0,0} 
\definecolor{Violet}{cmyk}{0.79,0.88,0,0} 
\definecolor{RoyalPurple}{cmyk}{0.75,0.90,0,0} 
\definecolor{BlueViolet}{cmyk}{0.86,0.91,0,0.04} 
\definecolor{Periwinkle}{cmyk}{0.57,0.55,0,0} 
\definecolor{CadetBlue}{cmyk}{0.62,0.57,0.23,0} 
\definecolor{CornflowerBlue}{cmyk}{0.65,0.13,0,0} 
\definecolor{MidnightBlue}{cmyk}{0.98,0.13,0,0.43} 
\definecolor{NavyBlue}{cmyk}{0.94,0.54,0,0} 
\definecolor{RoyalBlue}{cmyk}{1.,0.50,0,0} 
\definecolor{Blue}{cmyk}{1.,1.,0,0} 
\definecolor{Cerulean}{cmyk}{0.94,0.11,0,0} 
\definecolor{Cyan}{cmyk}{1.,0,0,0} 
\definecolor{ProcessBlue}{cmyk}{0.96,0,0,0} 
\definecolor{SkyBlue}{cmyk}{0.62,0,0.12,0} 
\definecolor{Turquoise}{cmyk}{0.85,0,0.20,0} 
\definecolor{TealBlue}{cmyk}{0.86,0,0.34,0.02} 
\definecolor{Aquamarine}{cmyk}{0.82,0,0.30,0} 
\definecolor{BlueGreen}{cmyk}{0.85,0,0.33,0} 
\definecolor{Emerald}{cmyk}{1.,0,0.50,0} 
\definecolor{JungleGreen}{cmyk}{0.99,0,0.52,0} 
\definecolor{SeaGreen}{cmyk}{0.69,0,0.50,0} 
\definecolor{Green}{cmyk}{1.,0,1.,0} 
\definecolor{ForestGreen}{cmyk}{0.91,0,0.88,0.12} 
\definecolor{PineGreen}{cmyk}{0.92,0,0.59,0.25} 
\definecolor{LimeGreen}{cmyk}{0.50,0,1.,0} 
\definecolor{YellowGreen}{cmyk}{0.44,0,0.74,0} 
\definecolor{SpringGreen}{cmyk}{0.26,0,0.76,0} 
\definecolor{OliveGreen}{cmyk}{0.64,0,0.95,0.40} 
\definecolor{RawSienna }{cmyk}{0,0.72,1.,0.45} 
\definecolor{Sepia}{cmyk}{0,0.83,1.,0.70} 
\definecolor{Brown}{cmyk}{0,0.81,1.,0.60} 
\definecolor{Tan}{cmyk}{0.14,0.42,0.56,0} 
\definecolor{Gray}{cmyk}{0,0,0,0.50} 
\definecolor{Black}{cmyk}{0,0,0,1.} 
\definecolor{White}{cmyk}{0,0,0,0} 

\newcommand{\hs}{\hat{s}}
\newcommand{\tPhi}{\tilde{\Phi}}
\newcommand{\tDe}{\tilde{\De}}

\newcommand{\tsi}{\tilde{\si}}
\newcommand{\tOm}{\tilde{\Om}}

\DeclareMathOperator{\Pos}{{\mathrm{Pos}}}
\newcommand\bGa{{\bf \Gamma}}
\def\fini{{\mathrm{fin}}}

\newcommand{\cA}{\mathcal{A}}

\newcommand{\cC}{\mathcal{C}}

\newcommand{\cE}{\mathcal{E}}
\newcommand{\cF}{\mathcal{F}}

\newcommand{\cH}{\mathcal{H}}

\newcommand{\cL}{\mathcal{L}}
\newcommand{\cM}{\mathcal{M}}

\newcommand{\cP}{\mathcal{P}}

\newcommand{\cR}{\mathcal{R}}

\newcommand{\cW}{\mathcal{W}}

\newcommand{\cLR}{\cL\cR}

\newcommand{\sg}{\langle}
\newcommand{\sd}{\rangle}

\newcommand{\bA}{\mathbf{A}}

\newcommand{\bC}{\mathbf{C}}

\newcommand{\bH}{\mathbf{H}}

\newcommand{\bL}{\mathbf{L}}
\newcommand{\bM}{\mathbf{M}}

\newcommand{\bP}{\mathbf{P}}

\newcommand{\bT}{\mathbf{T}}

\newcommand{\bc}{\mathbf{c}}

\newcommand{\bff}{\mathbf{f}}

\newcommand{\bm}{\mathbf{m}}

\newcommand{\bs}{\mathbf{s}}
\newcommand{\bt}{\mathbf{t}}

\newcommand{\fB}{\mathfrak{B}}

\newcommand{\fH}{\mathfrak{H}}
\newcommand{\fI}{\mathfrak{I}}

\newcommand{\nZ}{\mathbb{Z}}
\newcommand{\nR}{\mathbb{R}}
\newcommand{\nN}{\mathbb{N}}
\newcommand{\nQ}{\mathbb{Q}}

\newcommand{\al}{\alpha}

\newcommand{\si}{\sigma}
\newcommand{\la}{\lambda}
\newcommand{\ga}{\gamma}

\newcommand{\va}{\varphi}

\newcommand{\Ga}{\Gamma}
\newcommand{\De}{\Delta}
\newcommand{\Om}{\Omega}

\newcommand{\tA}{\tilde{A}}

\newcommand{\tW}{\tilde{W}}
\newcommand{\tw}{\tilde{w}}

\newcommand{\tL}{\tilde{L}}

\newcommand{\tcH}{\tilde{\cH}}

\newcommand{\tS}{\tilde{S}}
\newcommand{\tit}{\tilde{t}}

\newcommand{\sqs}{\sqsubset}

\newcommand{\ov}{\overline}

\newcommand{\hw}{w^{\circ}}

\newtheorem{Th}{Theorem}[section]
\newtheorem{Lem}[Th]{Lemma}
\newtheorem{Cor}[Th]{Corollary}
\newtheorem{Prop}[Th]{Proposition}
\newtheorem{Def-Prop}[Th]{Definition-Proposition}
\newtheorem{conjecture}[Th]{Conjecture}
\newtheorem{convention}[Th]{Convention}
\theoremstyle{definition}

\newtheorem{Exa}[Th]{Example}

\newtheorem{Claim}[Th]{Claim}
\theoremstyle{remark}
\newtheorem{Rem}[Th]{Remark}

\def\le{\hspace{0.1em}\mathop{\leqslant}\nolimits\hspace{0.1em}}
\def\ge{\hspace{0.1em}\mathop{\geqslant}\nolimits\hspace{0.1em}}
\def\notle{\hspace{0.1em}\mathop{\not\leqslant}\nolimits\hspace{0.1em}}

\def\NM{{\mathbb{N}}}

\def\RM{{\mathbb{R}}}

\def\ZM{{\mathbb{Z}}}

\def\Arm{{\mathrm{A}}}
\def\Brm{{\mathrm{B}}}
\def\Crm{{\mathrm{C}}}
\def\Drm{{\mathrm{D}}}

\def\Frm{{\mathrm{F}}}
\def\Grm{{\mathrm{G}}}

\def\a{\alpha}
\def\b{\beta}
\def\ga{\gamma}
\def\Ga{\Gamma}
\def\g{\gamma}
\def\G{\Gamma}
\def\d{\delta}

\def\e{\varepsilon}

\def\l{\lambda}

\def\o{\omega}

\def\s{\sigma}

\def\FC{{\mathcal{F}}}

\def\LC{{\mathcal{L}}}

\def\PC{{\mathcal{P}}}

\def\UC{{\mathcal{U}}}

\def\WC{{\mathcal{W}}}

\def\FCt{{\tilde{\mathcal{F}}}}

\def\cb{{\mathbf c}}

\def\Iti{{\tilde{I}}}

\def\Lti{{\tilde{L}}}

\def\Sti{{\tilde{S}}}

\def\Wti{{\tilde{W}}}

\def\cbt{{\tilde{\cb}}}

\def\deg{\mathop{\mathrm{deg}}\nolimits}

\def\ker{\mathop{\mathrm{ker}}\nolimits}

\def\mod{\mathop{\mathrm{mod}}\nolimits}

\def\longto{\longrightarrow}

\def\SS{\scriptstyle}
\def\SSS{\scriptscriptstyle}
\def\finl{~$\SS{\blacksquare}$}

\def\eqna{\begin{eqnarray*}}
\def\endeqna{\end{eqnarray*}}

\def\alc{{\mathrm{Alc}}}
\def\bfit{\bfseries\itshape}
\def\add{\hskip0.3mm{\SSS{\stackrel{\bullet}{}}}\hskip0.3mm}
\def\spe{{\mathrm{Spe}}}
\def\dotcup{\hskip0.5mm\dot{\cup}\hskip0.5mm}

\definecolor{purple}{rgb}{0.8,0.12,0.8}
\definecolor{orange}{rgb}{1.0,0.7,0.0}
\definecolor{pink}{rgb}{1,0.5,0.8}
\definecolor{blackg}{rgb}{0.1,0.25,0.1}
\definecolor{ForestGreen}{cmyk}{0.91,0,0.88,0.42}
\definecolor{Turquoise}{cmyk}{0.85,0,0.20,0}

\definecolor{GreenYellow}{cmyk}{0.15,0,0.69,0} 
\definecolor{Yellow}{cmyk}{0,0,1.,0} 
\definecolor{Goldenrod}{cmyk}{0,0.10,0.84,0} 
\definecolor{Dandelion}{cmyk}{0,0.29,0.84,0} 
\definecolor{Apricot}{cmyk}{0,0.32,0.52,0} 
\definecolor{Peach}{cmyk}{0,0.50,0.70,0} 
\definecolor{Melon}{cmyk}{0,0.46,0.50,0} 
\definecolor{YellowOrange}{cmyk}{0,0.42,1.,0} 
\definecolor{Orange}{cmyk}{0,0.61,0.87,0} 
\definecolor{BurntOrange}{cmyk}{0,0.51,1.,0} 
\definecolor{Bittersweet}{cmyk}{0,0.75,1.,0.24} 
\definecolor{RedOrange}{cmyk}{0,0.77,0.87,0} 
\definecolor{Mahogany}{cmyk}{0,0.85,0.87,0.35} 
\definecolor{Maroon}{cmyk}{0,0.87,0.68,0.32} 
\definecolor{BrickRed}{cmyk}{0,0.89,0.94,0.28} 
\definecolor{Red}{cmyk}{0,1.,1.,0} 
\definecolor{OrangeRed}{cmyk}{0,1.,0.50,0} 
\definecolor{RubineRed}{cmyk}{0,1.,0.13,0} 
\definecolor{WildStrawberry}{cmyk}{0,0.96,0.39,0} 
\definecolor{Salmon}{cmyk}{0,0.53,0.38,0} 
\definecolor{CarnationPink}{cmyk}{0,0.63,0,0} 
\definecolor{Magenta}{cmyk}{0,1.,0,0} 
\definecolor{VioletRed}{cmyk}{0,0.81,0,0} 
\definecolor{Rhodamine}{cmyk}{0,0.82,0,0} 
\definecolor{Mulberry}{cmyk}{0.34,0.90,0,0.02} 
\definecolor{RedViolet}{cmyk}{0.07,0.90,0,0.34} 
\definecolor{Fuchsia}{cmyk}{0.47,0.91,0,0.08} 
\definecolor{Lavender}{cmyk}{0,0.48,0,0} 
\definecolor{Thistle}{cmyk}{0.12,0.59,0,0} 
\definecolor{Orchid}{cmyk}{0.32,0.64,0,0} 
\definecolor{DarkOrchid}{cmyk}{0.40,0.80,0.20,0} 
\definecolor{Purple}{cmyk}{0.45,0.86,0,0} 
\definecolor{Plum}{cmyk}{0.50,1.,0,0} 
\definecolor{Violet}{cmyk}{0.79,0.88,0,0} 
\definecolor{RoyalPurple}{cmyk}{0.75,0.90,0,0} 
\definecolor{BlueViolet}{cmyk}{0.86,0.91,0,0.04} 
\definecolor{Periwinkle}{cmyk}{0.57,0.55,0,0} 
\definecolor{CadetBlue}{cmyk}{0.62,0.57,0.23,0} 
\definecolor{CornflowerBlue}{cmyk}{0.65,0.13,0,0} 
\definecolor{MidnightBlue}{cmyk}{0.98,0.13,0,0.43} 
\definecolor{NavyBlue}{cmyk}{0.94,0.54,0,0} 
\definecolor{RoyalBlue}{cmyk}{1.,0.50,0,0} 
\definecolor{Blue}{cmyk}{1.,1.,0,0} 
\definecolor{Cerulean}{cmyk}{0.94,0.11,0,0} 
\definecolor{Cyan}{cmyk}{1.,0,0,0} 
\definecolor{ProcessBlue}{cmyk}{0.96,0,0,0} 
\definecolor{SkyBlue}{cmyk}{0.62,0,0.12,0} 
\definecolor{Turquoise}{cmyk}{0.85,0,0.20,0} 
\definecolor{TealBlue}{cmyk}{0.86,0,0.34,0.02} 
\definecolor{Aquamarine}{cmyk}{0.82,0,0.30,0} 
\definecolor{BlueGreen}{cmyk}{0.85,0,0.33,0} 
\definecolor{Emerald}{cmyk}{1.,0,0.50,0} 
\definecolor{JungleGreen}{cmyk}{0.99,0,0.52,0} 
\definecolor{SeaGreen}{cmyk}{0.69,0,0.50,0} 
\definecolor{Green}{cmyk}{1.,0,1.,0} 
\definecolor{ForestGreen}{cmyk}{0.91,0,0.88,0.12} 
\definecolor{PineGreen}{cmyk}{0.92,0,0.59,0.25} 
\definecolor{LimeGreen}{cmyk}{0.50,0,1.,0} 
\definecolor{YellowGreen}{cmyk}{0.44,0,0.74,0} 
\definecolor{SpringGreen}{cmyk}{0.26,0,0.76,0} 
\definecolor{OliveGreen}{cmyk}{0.64,0,0.95,0.40} 
\definecolor{RawSienna }{cmyk}{0,0.72,1.,0.45} 
\definecolor{Sepia}{cmyk}{0,0.83,1.,0.70} 
\definecolor{Brown}{cmyk}{0,0.81,1.,0.60} 
\definecolor{Tan}{cmyk}{0.14,0.42,0.56,0} 
\definecolor{Gray}{cmyk}{0,0,0,0.50} 
\definecolor{Black}{cmyk}{0,0,0,1.} 
\definecolor{White}{cmyk}{0,0,0,0} 

\begin{document}

\title{Asymptotic lowest two-sided cell}

\author{{\sc C\'edric Bonnaf\'e}}
\address{C. Bonnaf\'e, 
Institut de Math\'ematiques et de Mod\'elisation de Montpellier (CNRS: UMR 5149), 
Universit\'e Montpellier 2,
Case Courrier 051,
Place Eug\`ene Bataillon,
34095 MONTPELLIER Cedex,
FRANCE} 

\makeatletter
\email{cedric.bonnafe@math.univ-montp2.fr}
\makeatother

\author{{\sc J\'er\'emie Guilhot}}

\address{J. Guilhot, School of Mathematics, University of East Anglia, Norwich NR4 7TJ, UK }
\email{j.guilhot@uea.ac.uk}


\subjclass{According to the 2000 classification: Primary 20C08; Secondary 20C15}

\date{\today}

\thanks{The first author is partly 
supported by the ANR (Project No JC07-192339)}


\maketitle

\pagestyle{myheadings}

\markboth{\sc C. Bonnaf\'e \& J. Guilhot}{\sc Asymptotic lowest two-sided cell}

\maketitle

\begin{abstract}
To a Coxeter system $(W,S)$ (with $S$ finite) 
and a weight function $L : W \to \NM$ is associated 
a partition of $W$ into Kazhdan-Lusztig (left, right or two-sided) $L$-cells. 
Let $S^\circ = \{s \in S~|~L(s)=0\}$, $S^+=\{s \in S~|~L(s) > 0\}$ and 
let $\cb$ be a Kazhdan-Lusztig (left, right or two-sided) $L$-cell. 
According to the semicontinuity conjecture of the first author, there 
should exist a positive natural number $m$ such that, for any weight function 
$L' : W \to \NM$ such that $L(s^+)=L'(s^+) > m L'(s^\circ)$ for all 
$s^+ \in S^+$ and $s^\circ \in S^\circ$, $\cb$ is a union of 
Kazhdan-Lusztig (left, right or two-sided) $L'$-cells. 

The aim of this paper is to prove this conjecture whenever 
$(W,S)$ is an affine Weyl group and $\cb$ is contained in the lowest 
two-sided $L$-cell.
\end{abstract}


\section{Introduction}


Let $(W,S)$ be a Coxeter system (with $S$ finite) and let $\G$ be a totally ordered 
abelian group. Let $L : W \to \G$ be a weight function in the sense of Lusztig 
\cite[\S 3.1]{bible}. To such a datum is associated a partition of $W$ into 
Kazhdan-Lusztig left, right or two-sided $L$-cells \cite[Chapter 8]{bible}. 
By virtue of~\cite[Corollary 2.5]{semicontinuity}, the computation of these partitions 
can be reduced to the case where $L$ has only non-negative values, 
which we assume here in this introduction. 
We then set
$$S^\circ = \{s \in S~|~L(s) = 0\}\qquad \text{and}\qquad 
S^+=\{s \in S~|~L(s) > 0\}.$$
A particular case of the semicontinuity conjecture of the first 
author~\cite[Conjecture~A(a)]{semicontinuity} 
can be stated as follows:

\bigskip

\noindent{\bf Semicontinuity Conjecture (asymptotic case).} 
{\it There exists a positive integer $m$ such that, 
for any Kazhdan-Lusztig (left, right or two-sided) $L$-cell $\cb$ 
and for any weight function $L' : W \to \G$ such that 
$L(s^+)=L'(s^+) > m L'(s^\circ)$ for all $s^+ \in S^+$ and 
$s^\circ \in S^\circ$, the subset $\cb$ is a union of 
Kazhdan-Lusztig (left, right or two-sided) $L'$-cells. }

\bigskip

The computation of the partition into Kazhdan-Lusztig cells is in general 
a very tough problem and a general proof of the semicontinuity conjecture 
would be very helpful. Even whenever it is not proved, 
it gives a lot of speculative ``upper bounds'' for the cells 
(for the inclusion order): at least, it can be seen as a guide 
along the computations. 

Note that the full semicontinuity conjecture~\cite[Conjecture A]{semicontinuity} 
(not only the asymptotic case) 
has been verified in different situations (see for instance the discussion 
in \cite[\S 5]{semicontinuity}). Note also that it has been established by the second author 
whenever $(W,S)$ is an affine Weyl group with $|S|=3$ (see~\cite{jeju4}). 
Our aim here is to prove a result slightly different in spirit than the previous ones. 
Indeed, it works for all affine Weyl groups and non-negative weight function $L$ 
but it focuses only on one particular two-sided cell, namely the 
lowest one (which we denote by $\cb_{\min}^L$).

\bigskip

\noindent{\bf Theorem.} 
{\it Assume that $(W,S)$ is an affine Weyl group. There exists a positive integer $m$ such that, for any Kazhdan-Lusztig (left, right or two-sided) 
$L$-cell $\cb$ contained in $\cb_{\min}^L$ and for any weight function 
$L' : W \to \G$ such that $L(s^+)=L'(s^+) > m L'(s^\circ)$ for all 
$s^+ \in S^+$ and $s^\circ \in S^\circ$, the subset $\cb$ is a union of 
Kazhdan-Lusztig (left, right or two-sided) $L'$-cells.}

\bigskip

The main ingredient of the proof of this result is the generalized induction 
of the second author~\cite{jeju3} together with the particular geometric description 
of the lowest two-sided cell and its left subcells. 

\medskip

The paper is organized as follows. In the literature, the lowest two-sided 
cell is defined whenever $L$ takes only positive values on $S$ (i.e. $S=S^+$). 
The aim of the first four sections is to extend this description of the case 
where $L$ is allowed to vanish on some elements of $S$ 
and to relate it to the semidirect product 
decomposition of $W$ associated to the partition $S = S^\circ ~\dot{\cup}~ S^+$ 
as in~\cite{Bonnafe-Dyer} (see also~\cite[\S 2.E]{semicontinuity}). 
It must be noticed that the proof of a key lemma (see 
Lemma~\ref{dot-property}) requires a case-by-case analysis: this lemma is of geometric 
nature and does not involve Kazhdan-Lusztig theory. 

In Section 5, we introduce Kazhdan-Lusztig theory and, in Section 6, we 
recall a more sophisticated version of the semicontinuity conjecture 
and we state our main results. The proof of these results is 
then done in the last two sections.

\bigskip

\section{Affine Weyl groups and Geometric realization}

\medskip

In this paper, we fix an euclidean $\RM$-vector space $V$ of dimension $r \ge 1$ and we denote 
by $\Phi$ an {\it irreducible} root system in $V$ of rank $r$: the scalar product will be denoted 
by $(,) : V \times V \longto \RM$. The dual of $V$ will be denoted by $V^*$ 
and $\langle , \rangle : V \times V^* \longto \RM$ will denote the canonical pairing. 
If $\a \in \Phi$, we denote by $\check{\a} \in V^*$ the associated coroot (if $x \in V$, then 
$\langle x,\check{\a} \rangle = 2 (x,\a)/(\a,\a)$) and by $\check{\Phi}$ the dual root system. We fix a positive system $\Phi^{+}$ and  for $\a \in \Phi^+$ we set
$$H_{\a,0}=\{x \in V~|~\langle x,\check{\a} \rangle = 0\}.$$
Then the Weyl group $\Om_{0}$ of $\Phi$ is generated by the orthogonal reflection with respect to the hyperplanes $H_{\al,0}$. It acts on the root lattice $\sg \Phi\sd$ and the semidirect product $\Om_{0}\ltimes \sg \Phi\sd$ is an affine Weyl group of type $\check{\Phi}$.

\subsection{Geometric realizations}
\label{geometric}
For $\a \in \Phi^+$ and $n \in \ZM$, we set
$$H_{\a,n}=\{x \in V~|~\langle x,\check{\a} \rangle = n\}$$
Then $H_{\a,n}$ is an {\it affine} hyperplane in $V$. Let
$$\FC=\{H_{\a,n}~|~\a \in \Phi^+\text{ and }n \in \ZM\}.$$
If $H \in \FC$, we denote by $\s_H$ the orthogonal reflection with respect to $H$. Let
$\Om= \langle \s_H ~|~H \in \FC\rangle$. Then $\Om$ is isomorphic to $W_{0}\ltimes \sg \Phi\sd$. We shall regard $\Om$ as acting on the right of $V$. 

\bigskip

An {\it alcove} is a connected component of the set
$$V-\underset{H\in\mathcal{F}}{\bigcup}H.$$
It is well-known that $\Om$ acts simply transitively on the set of alcoves $\alc(\cF)$. 
Recall also that, if $A$ is an alcove, then its closure $\overline{A}$ is 
a fundamental domain for the action of $\Om$ on $V$. 

\bigskip

The group $\Omega$ acts on the set of faces (the codimension 1 facets) of alcoves. We denote by $S$ the set of $\Om$-orbits in the set of faces. Note that if $A\in\alc(\cF)$, then the faces of $A$ is a set of representatives of $S$ since $\ov{A}$ is a fundamental domain for the action of $\Om$. If a face $f$ is contained in the orbit $s\in S$, we say that $f$ is of type $s$. To each $s\in S$ we can associate an involution $A\rightarrow sA$ of $\alc(\cF)$: the alcove $sA$ is the unique alcove which shares with $A$ a face of type $s$. Let $W$ be the group generated by all such involutions. Then $(W,S)$ is a Coxeter system and it is  isomomorphic to the affine Weyl group $W_{0}\ltimes \sg\Phi\sd$ (hence we also have $\Omega\simeq W$). We shall regard $W$ as acting on the left of $\alc(\cF)$. The action of $\Omega$ commutes with the action of $W$.

\bigskip 

We denote by $A_{0}$ the fundamental alcove associated to $\Phi$:
$$A_{0}=\{x\in V\mid 0<\sg x,\check{\alpha}\sd<1 \text{ for all $\alpha\in\Phi^{+}$}\}.$$
Let $A\in\alc(\cF)$. Then there exists a unique $w\in W$ such that $wA_{0}=A$.   We will freely identify $W$ with the set of alcoves $\alc(\cF)$
\bigskip

\subsection{Associated Coxeter system}
Let $\ell : W \to \NM$ denote the length function (with respect to the Coxeter system $(W,S)$). 
We denote by $\LC(W)$ the set of finite sequences $(w_1,\dots,w_n)$ of elements 
of $W$ such that $\ell(w_1 \cdots w_n)=\ell(w_1) + \cdots + \ell(w_n)$. 
If $(w_1,\dots,w_n)$ is a finite sequence of elements of $W$ then, 
in order to simplify notation, we shall write $w_1 \add w_2 \add \cdots \add w_n$ 
if $(w_1,\dots,w_n) \in \LC(W)$. 
If $I$ is a subset of $S$, we denote by $W_I$ the subgroup of $W$ generated 
by $I$. We denote by $X_I$ the set of elements $w \in W$ which are of minimal 
length in $wW_I$: it is a set of representatives of $W/W_I$. It follows from 
the irreducibility of $\Phi$ that $W_I$ is finite whenever $I$ is a {\it proper} 
subset of $S$: in this case, the longest element of $W_I$ will be denoted by $w_I$. 

\bigskip

\begin{Exa}
\label{standard-conjugate}
Let $\l \in V$ be a $0$-dimensional facet of an alcove. We denote by $W_{\la}$ the stabilizer in $W$ of the set of alcoves containing $\la$. It can be shown that $W_{\la}$ is the standard parabolic subgroup of $W$ generated by $S_\l=S\cap W_{\la}$
(in other words, with the previous notation, $W_\l=W_{S_\l}$). Note that $W_\l$ is finite: the longest element of $W_\l$ will be 
denoted by $w_\l$ and we set $X_\l=X_{S_\l}$.\finl
\end{Exa}

Let $H=H_{\alpha,n}\in \cF$ with $\a \in \Phi^+$ and $n \in \ZM$. Then $H$ divides $V-H$ into two half-spaces
\begin{align*}
V_{H}^{+}&=\{\mu\in V\mid \sg x,\check{\alpha}\sd>n\},\\
V_{H}^{-}&=\{\mu\in V\mid \sg x,\check{\alpha}\sd<n\}.
\end{align*}
We say that an hyperplane $H$ separates the alcoves $A$ and $B$ if $A\subset V_{H}^{+}$ and $B\subset V_{H}^{-}$ or $A\subset V_{H}^{-}$ and $B\subset V_{H}^{+}$. For $A,B\in\alc(\cF)$, we set
$$H(A,B)=\{H\in \cF| H\text{ separates $A$ and $B$}\}.$$
\begin{Prop}
\label{length-hyp}
Let $x,y\in W$ and $A\in\alc(\cF)$. We have
\begin{enumerate}
\item $\ell(x)=|H(A,xA)|$;
\item $x\add y$ if and only if  $H(A,yA)\cap H(yA,xyA)=\emptyset$.
\end{enumerate}
\end{Prop}

\def\fini{{\mathrm{fin}}}

\bigskip

\subsection{Weight functions} 
Let $(\Gamma,+)$ be a totally ordered abelian group: the order on $\G$ will be denoted by $\le$. 
Let $L : W \to \Ga$ be a {\it weight function} on $W$, that is a function satisfying 
$L(ww')=L(w)+L(w')$ whenever $\ell(ww')=\ell(w)+\ell(w')$. Recall that 
this implies the following property:
\begin{equation}\label{st}
\text{\it If $s$, $t \in S$ are conjugate in $W$, then $L(s)=L(t)$.}
\end{equation}
We denote by $\text{Weight}(W,\Ga)$ the set of weight function from $W$ to $\Ga$. Throughout this paper, 
we will always assume that $L$ is {\it non-negative} that is  $L(s) \ge 0$ for all $s \in S$. 
The weight function $L$ is called {\it positive} if $L(s) > 0$ for all $s \in S$ 
(in other words, $L$ is positive if and only if $L(w) > 0$ if $w \neq 1$). 
Note that a weight function on $W$ is completely determined by its values on the generators $s\in S$: the element of the set $\{L(s)\mid s\in S\}$ 
are called the parameters.

\bigskip

\begin{Exa}
The map $W \to \G$, $w \mapsto 0$ is a weight function (and will be denoted by $0$): 
it is not positive (if $W \neq 1$). On the other hand, 
$\ell : W \to \ZM$ is a positive weight function.\finl
\end{Exa}

\bigskip

Here is a first consequence of the non-negativeness assumption:

\bigskip

\begin{Prop}\label{addition L}
Let $x$, $y \in W$. If $L(x)=0$, then $L(xy)=L(yx)=L(y)$.
\end{Prop}

\bigskip

\begin{proof}
Let $l = \ell(x)$ and let $s_1$,\dots, $s_l$ 
be elements of $S$ such that $x=s_1\cdots s_l$.  
Then $L(x)=L(s_1)+\cdots+L(s_l)$, so $L(s_i)=0$ for all $i \in \{1,2,\dots,l\}$, 
because $L$ is non-negative. So, arguing by induction on the length of $x$, 
we may (and we will) assume that $\ell(x)=1$, i.e. that $x=s_1$. Two cases 
may occur:

\medskip

$\bullet$ If $xy > y$, then $\ell(xy)=\ell(x)+\ell(y)$, so $L(xy)=L(x)+L(y)=L(y)$, as desired. 

\medskip

$\bullet$ If $xy < y$, then $\ell(y)=\ell(x) + \ell(xy)$, so $L(y)=L(x)+L(xy)=L(xy)$, 
as desired.
\end{proof}

\bigskip

\subsection{$L$-special points} 
\label{L-special}

Let $H \in \FC$ and assume that $H$ supports a face of type $s\in S$: we then set $L_H=L(s)$. Note that this is well defined since if $H$ supports  faces of type $s,t\in S$ then $s$ and $t$ are conjugate in W \cite[Lemma 2.1]{Bremke}.  Then $L_H=L_{H\si}$ for all $\si \in \Om$. 
If $\la$ is a $0$-dimensional facet  of an alcove, we set
$$L_\l=\sum_{\substack{H \in \FC \\ \l \in H}} L_H=L(w_{\la}).$$
Note that $L_{\l\si}=L_\l$ for all $\si \in \Om$. We set
$$\nu_{L}:=\max_{\la\in V} L_{\l}.$$
We then say that $\la$ is an {\it $L$-special point} if $L_\l=\nu_{L}$. We denote by $\spe_{L}(W)$ the set of $L$-special points: it is stable 
under the action of $\Om$. Since $\ov{A_{0}}$ is a fundamental domain for the action of $\Om$, the set $\spe_{L}(W)\cap \ov{A_{0}}$ is a set of representative of orbits of $\spe_{L}(W)$ under the action of $\Om$. 

\begin{Exa}
If $L=\ell$ is the usual length function then $L_{\la}$ is just the number of hyperplanes which go through $\la$ hence $\nu_{\ell}=|\Phi^{+}|$  and the set of $\ell$-special points is equal to the weight lattice
$$P(\Phi)=\{v \in V~|~\forall~\a \in \Phi,~\langle v, \check{\a} \rangle \in \ZM\}.$$
Hence we recover the original definition of special points by Lusztig in 
\cite{Lusztig-1980}.\finl
\end{Exa}

\begin{quotation}
\begin{convention}\label{convention}
{\it If $W$ is not of type $\tilde{\Crm}_{r}$ ($r\ge 1$) then any two parallel 
hyperplanes have same weight \cite[Lemma 2.2]{Bremke}.  In the case where $W$ is of type
$\tilde{\Crm}_{r}$ with generators $t$, $s_{1}$,..., $s_{r-1}$, $t'$ such that 
$\sg t,s_1,\dots,s_{r-1}\sd= W_{0}$ and $\sg s_1,\dots,s_{r-1}\sd$ is of type 
$\Arm_{r-1}$, 
by symmetry of the Dynkin diagram, we may (and we will) assume that 
$L(t)\ge L(t')$.}
\end{convention}
\end{quotation}

\bigskip

Recall that the type $\tilde{\Crm}_1$ is also the type $\tilde{\Arm}_1$.

\bigskip

\begin{Rem}
Note that with our Convention \ref{convention} for $\tilde{\Crm}_{r}$ ($r \ge 1$), 
the point $0 \in V$ is always an $L$-special point.\finl
\end{Rem}

\bigskip
For $\al\in\Phi$ we set
$$L_\a:=\max_{n \in \ZM}~ L_{H_{\a,n}}.$$
\begin{Rem}
Note that if $W$ is not of type $\tilde{\Crm}$ then since any two parallel hyperplanes have same weight we have $L_{\al}=L_{H_{\al,n}}$ for all $n$. In general we will say that $H=H_{\al,n}\in\cF$ is of {\it maximal weight} if $L_{H}=L_{\al}$.\finl 
\end{Rem}
\medskip
We denote by $\Phi^{L}$ the subset of $\Phi$ which consists of all roots of positive weight. Note that $\Phi^{L}$ is a root system of rank $r$, not necessarily irreducible,  and that $\Phi^{L}\cap\Phi^{+}$ is a positive system in $\Phi^{L}$: see the proof of Lemma \ref{dot-property} where we classify the root systems $\Phi^{L}$. We denote by $\De^{L}$ the unique simple system contained  in $\Phi^{L}\cap \Phi^{+}$. We have
\begin{enumerate}
\item[$\bullet$] If $W$ is not of type $\tilde{\Crm}$ or if $L(t)=L(t')$ in type $\tilde{\Crm}$ then
$$\spe_{L}(W)=\{v \in V~|~\forall~\a \in \Phi^{L},~\langle v, \check{\a} \rangle \in \ZM\}.$$
\item[$\bullet$] If $W$ is of type $\tilde{\Crm}$ and $L(t)>L(t')$ then 
$$\spe_{L}(W)=\sg \Phi^{L}\sd\varsubsetneq \{v \in V~|~\forall~\a \in \Phi^{L},~\langle v, \check{\a} \rangle \in \ZM\}.$$
\end{enumerate}
In other words, the $L$-special points are those points of $V$ which lies in the intersection of $|\Phi^{L}\cap \Phi^{+}|$ hyperplanes of maximal weight.

\bigskip

Let $\cF^{L}=\{H\in\cF\mid L_{H}>0\}$. Let $\la$  be a $0$-dimensional facet  of an alcove which is contained in an hyperplane of positive weight.  An {\it $L$-quarter} with vertex $\la$ is a connected component of
$$V-\bigcup_{\substack{H \in \FC^{L} \\ \l \in H}} H.$$
It is an open simplicial cone: it has $r$ walls. 

\bigskip

Let $\al\in\Phi^{L}$. A {\it maximal $L$-strip orthogonal to $\a$} is a connected component of 
$$V-\bigcup_{\substack{n \in \ZM \\ L_{H_{\a,n}}=L_\a}} H_{\a,n}.$$
If $A$ is an alcove, we denove by $\UC_\a^L(A)$ the unique maximal $L$-strip 
orthogonal to $\a$ containing $A$. Finally, we set 
$$\UC^L(A_0) = \bigcup_{\substack{\a \in \Phi^+\\ L_\a > 0}} \UC_\a^L(A_0).$$

\section{On the lowest two-sided cell} 
\label{lowest}
We keep the notation of the previous section. We fix a non-negative weight function $L\in \text{Weight}(W,\Ga)$.


\subsection{Definition and examples}
\label{def-lowest}
Recall that we have set $\nu_{L}:=\max_{\la\in V} L_{\l}$. Since $W_{\la}$ is a standard parabolic subgroup of $W$, one can easily see that 
$$\nu_{L}=\max_{I\in \PC_\fini(S)} L(w_{I})$$
where $\PC_\fini(S)$ denotes the set of subsets $I$ of $S$ such that $W_I$ is finite. We set
$$\WC=\bigcup_{I \in \PC_\fini(S)} W_I.$$
Then we define the  {\it lowest two-sided cell}  of $W$ by
$$\cb^L_{\min}(W)=\{xwy~|~w \in \WC,~x \add w \add y \text{ and } L(w)=\nu_L\}.$$
We shall see later (see Section \ref{KL-lowest}) the reason for this terminology. When the group $W$ is clear from the context, we will write $\cb^{L}_{\min}$ instead of $\cb^{L}_{\min}(W)$. Note the following immediate property of $\cb^L_{\min}$:

\bigskip

\begin{Lem}\label{cmin addition}
Let $w$, $x$, $y \in W$ be such that $w \in \cb^L_{\min}$ and $x \add w \add y$. Then 
$xwy \in \cb^L_{\min}$.
\end{Lem}

\bigskip
The set $\cb^{L}_{\min}$ can change quite dramatically when the parameters are varying as shown in the following example.

\begin{Exa}
\label{B2}
Let $W$ be of type $\tilde{\Crm}_2$ with diagram and weight function given by
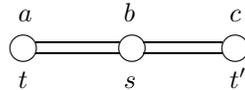
\begin{center}
\begin{picture}(150,32)
\put( 40, 10){\circle{10}}
\put( 44.5,  8){\line(1,0){31.75}}
\put( 44.5,  12){\line(1,0){31.75}}
\put( 81, 10){\circle{10}}
\put( 85.5,  8){\line(1,0){29.75}}
\put( 85.5,  12){\line(1,0){29.75}}
\put(120, 10){\circle{10}}
\put( 38, 20){$a$}
\put( 78, 20){$b$}
\put(118, 20){$c$}
\put( 38, -5){$t$}
\put( 78, -5){$s$}
\put(118,-5){$t'$}
\end{picture}
\end{center}
where $a,b,c\in\Ga$ and, by convention, we assume that $a\geq c$. We start by describing the set
$$\WC^{\max}=\{w\in\WC\mid L(w)=\nu_{L}\}.$$
We have
$$\WC^{\max}=
\begin{cases}
\{tsts\}&\mbox{ if $a > c$ and $b>0$,}\\
 \{tsts,tst\}&\mbox{ if $a>c$ and $b=0$,}\\
\{tsts,t'st's\}&\mbox{ if $a=c > 0$ and $b>0$,}\\
 \{tsts,tst,t'st's,t'st',tt'\}&\mbox{ if $a=c>0$ and $b=0$,}\\
\{sts,stst,st's,st'st'\}&\mbox{ if $a=c=0$ and $b>0$,}\\
\WC&\mbox{ if $a,b,c=0$.}\\
\end{cases}$$
If $\WC^{\max}=\WC$ then we get $\cb_{\min}^{L}=W$. Otherwise the corresponding sets $\cb_{\min}^{L}$ are described in the following figures: the black alcove is the fundamental alcove $A_{0}$, the alcoves with a star correspond to the set $\cW^{\max}$ and the set $\cb^{L}_{\min}$ consists of all the alcoves lying in the gray area (via the identification $w\leftrightarrow wA_{0}$).

\psset{linewidth=.13mm}
\psset{unit=.5cm}


\begin{textblock}{10}(2,10.5)

\begin{center}
\begin{pspicture}(-6,-6)(6,6)

\pspolygon[fillstyle=solid,fillcolor=black](0,0)(0,1)(.5,.5)

\pspolygon[fillstyle=solid,fillcolor=lightgray](0,0)(0,-6)(-6,-6)
\pspolygon[fillstyle=solid,fillcolor=lightgray](1,-1)(1,-6)(6,-6)
\pspolygon[fillstyle=solid,fillcolor=lightgray](2,0)(6,0)(6,-4)
\pspolygon[fillstyle=solid,fillcolor=lightgray](1,1)(6,6)(6,1)
\pspolygon[fillstyle=solid,fillcolor=lightgray](1,3)(1,6)(4,6)
\pspolygon[fillstyle=solid,fillcolor=lightgray](0,2)(0,6)(-4,6)
\pspolygon[fillstyle=solid,fillcolor=lightgray](-1,1)(-6,1)(-6,6)
\pspolygon[fillstyle=solid,fillcolor=lightgray](-2,0)(-6,0)(-6,-4)

\rput(-.2,-.5){{\footnotesize $\ast$}}

\psline(-6,-6)(-6,6)
\psline(-5,-6)(-5,6)
\psline(-4,-6)(-4,6)
\psline(-3,-6)(-3,6)
\psline(-2,-6)(-2,6)
\psline(-1,-6)(-1,6)
\psline(0,-6)(0,6)
\psline(1,-6)(1,6)
\psline(2,-6)(2,6)
\psline(3,-6)(3,6)
\psline(4,-6)(4,6)
\psline(5,-6)(5,6)
\psline(6,-6)(6,6)

\psline(-6,-6)(6,-6)
\psline(-6,-5)(6,-5)
\psline(-6,-4)(6,-4)
\psline(-6,-3)(6,-3)
\psline(-6,-2)(6,-2)
\psline(-6,-1)(6,-1)
\psline(-6,0)(6,0)
\psline(-6,1)(6,1)
\psline(-6,2)(6,2)
\psline(-6,3)(6,3)
\psline(-6,4)(6,4)
\psline(-6,5)(6,5)
\psline(-6,6)(6,6)

\psline(0,0)(1,1)
\psline(0,1)(1,0)

\psline(1,0)(2,1)
\psline(1,1)(2,0)

\psline(2,0)(3,1)
\psline(2,1)(3,0)

\psline(3,0)(4,1)
\psline(3,1)(4,0)

\psline(4,0)(5,1)
\psline(4,1)(5,0)

\psline(5,0)(6,1)
\psline(5,1)(6,0)

\psline(-1,0)(0,1)
\psline(-1,1)(0,0)

\psline(-2,0)(-1,1)
\psline(-2,1)(-1,0)

\psline(-3,0)(-2,1)
\psline(-3,1)(-2,0)

\psline(-4,0)(-3,1)
\psline(-4,1)(-3,0)

\psline(-5,0)(-4,1)
\psline(-5,1)(-4,0)

\psline(-6,0)(-5,1)
\psline(-6,1)(-5,0)

\psline(0,1)(1,2)
\psline(0,2)(1,1)

\psline(1,1)(2,2)
\psline(1,2)(2,1)

\psline(2,1)(3,2)
\psline(2,2)(3,1)

\psline(3,1)(4,2)
\psline(3,2)(4,1)

\psline(4,1)(5,2)
\psline(4,2)(5,1)

\psline(5,1)(6,2)
\psline(5,2)(6,1)

\psline(-1,1)(0,2)
\psline(-1,2)(0,1)

\psline(-2,1)(-1,2)
\psline(-2,2)(-1,1)

\psline(-3,1)(-2,2)
\psline(-3,2)(-2,1)

\psline(-4,1)(-3,2)
\psline(-4,2)(-3,1)

\psline(-5,1)(-4,2)
\psline(-5,2)(-4,1)

\psline(-6,1)(-5,2)
\psline(-6,2)(-5,1)

\psline(0,2)(1,3) 
\psline(0,3)(1,2) 

\psline(1,2)(2,3) 
\psline(1,3)(2,2) 

\psline(2,2)(3,3) 
\psline(2,3)(3,2) 
 
\psline(3,2)(4,3) 
\psline(3,3)(4,2) 

\psline(4,2)(5,3) 
\psline(4,3)(5,2) 

\psline(5,2)(6,3) 
\psline(5,3)(6,2) 

\psline(-1,2)(0,3) 
\psline(-1,3)(0,2) 

\psline(-2,2)(-1,3) 
\psline(-2,3)(-1,2) 

\psline(-3,2)(-2,3) 
\psline(-3,3)(-2,2) 

\psline(-4,2)(-3,3) 
\psline(-4,3)(-3,2) 

\psline(-5,2)(-4,3) 
\psline(-5,3)(-4,2) 

\psline(-6,2)(-5,3) 
\psline(-6,3)(-5,2)

\psline(0,3)( 1,4) 
\psline(0,4)( 1,3) 

\psline(1,3)( 2,4) 
\psline(1,4)( 2,3) 

\psline(2,3)( 3,4) 
\psline(2,4)( 3,3) 
 
\psline(3,3)( 4,4) 
\psline(3,4)( 4,3) 

\psline(4,3)( 5,4) 
\psline(4,4)( 5,3) 

\psline(5,3)( 6,4) 
\psline(5,4)( 6,3) 

\psline(-1,3)( 0,4) 
\psline(-1,4)( 0,3) 

\psline(-2,3)( -1,4) 
\psline(-2,4)( -1,3) 

\psline(-3,3)( -2,4) 
\psline(-3,4)( -2,3) 

\psline(-4,3)( -3,4) 
\psline(-4,4)( -3,3) 

\psline(-5,3)( -4,4) 
\psline(-5,4)( -4,3) 

\psline(-6,3)( -5,4) 
\psline(-6,4)( -5,3)

\psline(0,4)(1,5)
\psline(0,5)(1,4)

\psline(1,4)(2,5)
\psline(1,5)(2,4)

\psline(2,4)(3,5)
\psline(2,5)(3,4)

\psline(3,4)(4,5)
\psline(3,5)(4,4)

\psline(4,4)(5,5)
\psline(4,5)(5,4)

\psline(5,4)(6,5)
\psline(5,5)(6,4)

\psline(-1,4)(0,5)
\psline(-1,5)(0,4)

\psline(-2,4)(-1,5)
\psline(-2,5)(-1,4)

\psline(-3,4)(-2,5)
\psline(-3,5)(-2,4)

\psline(-4,4)(-3,5)
\psline(-4,5)(-3,4)

\psline(-5,4)(-4,5)
\psline(-5,5)(-4,4)

\psline(-6,4)(-5,5)
\psline(-6,5)(-5,4)

\psline(0,5)(1,6)
\psline(0,6)(1,5)

\psline(1,5)(2,6)
\psline(1,6)(2,5)

\psline(2,5)(3,6)
\psline(2,6)(3,5)

\psline(3,5)(4,6)
\psline(3,6)(4,5)

\psline(4,5)(5,6)
\psline(4,6)(5,5)

\psline(5,5)(6,6)
\psline(5,6)(6,5)

\psline(-1,5)(0,6)
\psline(-1,6)(0,5)

\psline(-2,5)(-1,6)
\psline(-2,6)(-1,5)

\psline(-3,5)(-2,6)
\psline(-3,6)(-2,5)

\psline(-4,5)(-3,6)
\psline(-4,6)(-3,5)

\psline(-5,5)(-4,6)
\psline(-5,6)(-4,5)

\psline(-6,5)(-5,6)
\psline(-6,6)(-5,5)

\psline(0,0)(1,-1)
\psline(0,-1)(1,0)

\psline(1,0)(2,-1)
\psline(1,-1)(2,0)

\psline(2,0)(3,-1)
\psline(2,-1)(3,0)

\psline(3,0)(4,-1)
\psline(3,-1)(4,0)

\psline(4,0)(5,-1)
\psline(4,-1)(5,0)

\psline(5,0)(6,-1)
\psline(5,-1)(6,0)

\psline(-1,0)(0,-1)
\psline(-1,-1)(0,0)

\psline(-2,0)(-1,-1)
\psline(-2,-1)(-1,0)

\psline(-3,0)(-2,-1)
\psline(-3,-1)(-2,0)

\psline(-4,0)(-3,-1)
\psline(-4,-1)(-3,0)

\psline(-5,0)(-4,-1)
\psline(-5,-1)(-4,0)

\psline(-6,0)(-5,-1)
\psline(-6,-1)(-5,0)

\psline(0,-1)(1,-2)
\psline(0,-2)(1,-1)

\psline(1,-1)(2,-2)
\psline(1,-2)(2,-1)

\psline(2,-1)(3,-2)
\psline(2,-2)(3,-1)

\psline(3,-1)(4,-2)
\psline(3,-2)(4,-1)

\psline(4,-1)(5,-2)
\psline(4,-2)(5,-1)

\psline(5,-1)(6,-2)
\psline(5,-2)(6,-1)

\psline(-1,-1)(0,-2)
\psline(-1,-2)(0,-1)

\psline(-2,-1)(-1,-2)
\psline(-2,-2)(-1,-1)

\psline(-3,-1)(-2,-2)
\psline(-3,-2)(-2,-1)

\psline(-4,-1)(-3,-2)
\psline(-4,-2)(-3,-1)

\psline(-5,-1)(-4,-2)
\psline(-5,-2)(-4,-1)

\psline(-6,-1)(-5,-2)
\psline(-6,-2)(-5,-1)

\psline(0,-2)(1,-3) 
\psline(0,-3)(1,-2) 

\psline(1,-2)(2,-3) 
\psline(1,-3)(2,-2) 

\psline(2,-2)(3,-3) 
\psline(2,-3)(3,-2) 
 
\psline(3,-2)(4,-3) 
\psline(3,-3)(4,-2) 

\psline(4,-2)(5,-3) 
\psline(4,-3)(5,-2) 

\psline(5,-2)(6,-3) 
\psline(5,-3)(6,-2) 

\psline(-1,-2)(0,-3) 
\psline(-1,-3)(0,-2) 

\psline(-2,-2)(-1,-3) 
\psline(-2,-3)(-1,-2) 

\psline(-3,-2)(-2,-3) 
\psline(-3,-3)(-2,-2) 

\psline(-4,-2)(-3,-3) 
\psline(-4,-3)(-3,-2) 

\psline(-5,-2)(-4,-3) 
\psline(-5,-3)(-4,-2) 

\psline(-6,-2)(-5,-3) 
\psline(-6,-3)(-5,-2)

\psline(0,-3)( 1,-4) 
\psline(0,-4)( 1,-3) 

\psline(1,-3)( 2,-4) 
\psline(1,-4)( 2,-3) 

\psline(2,-3)( 3,-4) 
\psline(2,-4)( 3,-3) 
 
\psline(3,-3)( 4,-4) 
\psline(3,-4)( 4,-3) 

\psline(4,-3)( 5,-4) 
\psline(4,-4)( 5,-3) 

\psline(5,-3)( 6,-4) 
\psline(5,-4)( 6,-3) 

\psline(-1,-3)( 0,-4) 
\psline(-1,-4)( 0,-3) 

\psline(-2,-3)( -1,-4) 
\psline(-2,-4)( -1,-3) 

\psline(-3,-3)( -2,-4) 
\psline(-3,-4)( -2,-3) 

\psline(-4,-3)( -3,-4) 
\psline(-4,-4)( -3,-3) 

\psline(-5,-3)( -4,-4) 
\psline(-5,-4)( -4,-3) 

\psline(-6,-3)( -5,-4) 
\psline(-6,-4)( -5,-3)

\psline(0,-4)(1,-5)
\psline(0,-5)(1,-4)

\psline(1,-4)(2,-5)
\psline(1,-5)(2,-4)

\psline(2,-4)(3,-5)
\psline(2,-5)(3,-4)

\psline(3,-4)(4,-5)
\psline(3,-5)(4,-4)

\psline(4,-4)(5,-5)
\psline(4,-5)(5,-4)

\psline(5,-4)(6,-5)
\psline(5,-5)(6,-4)

\psline(-1,-4)(0,-5)
\psline(-1,-5)(0,-4)

\psline(-2,-4)(-1,-5)
\psline(-2,-5)(-1,-4)

\psline(-3,-4)(-2,-5)
\psline(-3,-5)(-2,-4)

\psline(-4,-4)(-3,-5)
\psline(-4,-5)(-3,-4)

\psline(-5,-4)(-4,-5)
\psline(-5,-5)(-4,-4)

\psline(-6,-4)(-5,-5)
\psline(-6,-5)(-5,-4)

\psline(0,-5)(1,-6)
\psline(0,-6)(1,-5)

\psline(1,-5)(2,-6)
\psline(1,-6)(2,-5)

\psline(2,-5)(3,-6)
\psline(2,-6)(3,-5)

\psline(3,-5)(4,-6)
\psline(3,-6)(4,-5)

\psline(4,-5)(5,-6)
\psline(4,-6)(5,-5)

\psline(5,-5)(6,-6)
\psline(5,-6)(6,-5)

\psline(-1,-5)(0,-6)
\psline(-1,-6)(0,-5)

\psline(-2,-5)(-1,-6)
\psline(-2,-6)(-1,-5)

\psline(-3,-5)(-2,-6)
\psline(-3,-6)(-2,-5)

\psline(-4,-5)(-3,-6)
\psline(-4,-6)(-3,-5)

\psline(-5,-5)(-4,-6)
\psline(-5,-6)(-4,-5)

\psline(-6,-5)(-5,-6)
\psline(-6,-6)(-5,-5)

\rput(0,-6.7){The set $\bc^{L}_{\min}$ for $a>c$ and $b>0$}
\end{pspicture}
\end{center}

\end{textblock}


\begin{textblock}{10}(9.55,10.5)
\begin{center}
\begin{pspicture}(-6,-6)(6,6)

\pspolygon[fillstyle=solid,fillcolor=black](0,0)(0,1)(.5,.5)

\pspolygon[fillstyle=solid,fillcolor=lightgray](0,0)(0,-6)(-6,-6)
\pspolygon[fillstyle=solid,fillcolor=lightgray](1,-1)(1,-6)(6,-6)
\pspolygon[fillstyle=solid,fillcolor=lightgray](2,0)(6,0)(6,-4)
\pspolygon[fillstyle=solid,fillcolor=lightgray](1,1)(6,6)(6,1)
\pspolygon[fillstyle=solid,fillcolor=lightgray](1,3)(1,6)(4,6)
\pspolygon[fillstyle=solid,fillcolor=lightgray](0,2)(0,6)(-4,6)
\pspolygon[fillstyle=solid,fillcolor=lightgray](-1,1)(-6,1)(-6,6)
\pspolygon[fillstyle=solid,fillcolor=lightgray](-2,0)(-6,0)(-6,-4)

\pspolygon[fillstyle=solid,fillcolor=lightgray](0,0)(0,-6)(1,-6)(1,-1)
\pspolygon[fillstyle=solid,fillcolor=lightgray](1,1)(6,1)(6,0)(2,0)
\pspolygon[fillstyle=solid,fillcolor=lightgray](-1,1)(-6,1)(-6,0)(-2,0)
\pspolygon[fillstyle=solid,fillcolor=lightgray](0,2)(0,6)(1,6)(1,3)

\rput(-.2,-.5){{\footnotesize $\ast$}}
\rput(.2,-.5){{\footnotesize $\ast$}}

\psline(-6,-6)(-6,6)
\psline(-5,-6)(-5,6)
\psline(-4,-6)(-4,6)
\psline(-3,-6)(-3,6)
\psline(-2,-6)(-2,6)
\psline(-1,-6)(-1,6)
\psline(0,-6)(0,6)
\psline(1,-6)(1,6)
\psline(2,-6)(2,6)
\psline(3,-6)(3,6)
\psline(4,-6)(4,6)
\psline(5,-6)(5,6)
\psline(6,-6)(6,6)

\psline(-6,-6)(6,-6)
\psline(-6,-5)(6,-5)
\psline(-6,-4)(6,-4)
\psline(-6,-3)(6,-3)
\psline(-6,-2)(6,-2)
\psline(-6,-1)(6,-1)
\psline(-6,0)(6,0)
\psline(-6,1)(6,1)
\psline(-6,2)(6,2)
\psline(-6,3)(6,3)
\psline(-6,4)(6,4)
\psline(-6,5)(6,5)
\psline(-6,6)(6,6)

\psline(0,0)(1,1)
\psline(0,1)(1,0)

\psline(1,0)(2,1)
\psline(1,1)(2,0)

\psline(2,0)(3,1)
\psline(2,1)(3,0)

\psline(3,0)(4,1)
\psline(3,1)(4,0)

\psline(4,0)(5,1)
\psline(4,1)(5,0)

\psline(5,0)(6,1)
\psline(5,1)(6,0)

\psline(-1,0)(0,1)
\psline(-1,1)(0,0)

\psline(-2,0)(-1,1)
\psline(-2,1)(-1,0)

\psline(-3,0)(-2,1)
\psline(-3,1)(-2,0)

\psline(-4,0)(-3,1)
\psline(-4,1)(-3,0)

\psline(-5,0)(-4,1)
\psline(-5,1)(-4,0)

\psline(-6,0)(-5,1)
\psline(-6,1)(-5,0)

\psline(0,1)(1,2)
\psline(0,2)(1,1)

\psline(1,1)(2,2)
\psline(1,2)(2,1)

\psline(2,1)(3,2)
\psline(2,2)(3,1)

\psline(3,1)(4,2)
\psline(3,2)(4,1)

\psline(4,1)(5,2)
\psline(4,2)(5,1)

\psline(5,1)(6,2)
\psline(5,2)(6,1)

\psline(-1,1)(0,2)
\psline(-1,2)(0,1)

\psline(-2,1)(-1,2)
\psline(-2,2)(-1,1)

\psline(-3,1)(-2,2)
\psline(-3,2)(-2,1)

\psline(-4,1)(-3,2)
\psline(-4,2)(-3,1)

\psline(-5,1)(-4,2)
\psline(-5,2)(-4,1)

\psline(-6,1)(-5,2)
\psline(-6,2)(-5,1)

\psline(0,2)(1,3) 
\psline(0,3)(1,2) 

\psline(1,2)(2,3) 
\psline(1,3)(2,2) 

\psline(2,2)(3,3) 
\psline(2,3)(3,2) 
 
\psline(3,2)(4,3) 
\psline(3,3)(4,2) 

\psline(4,2)(5,3) 
\psline(4,3)(5,2) 

\psline(5,2)(6,3) 
\psline(5,3)(6,2) 

\psline(-1,2)(0,3) 
\psline(-1,3)(0,2) 

\psline(-2,2)(-1,3) 
\psline(-2,3)(-1,2) 

\psline(-3,2)(-2,3) 
\psline(-3,3)(-2,2) 

\psline(-4,2)(-3,3) 
\psline(-4,3)(-3,2) 

\psline(-5,2)(-4,3) 
\psline(-5,3)(-4,2) 

\psline(-6,2)(-5,3) 
\psline(-6,3)(-5,2)

\psline(0,3)( 1,4) 
\psline(0,4)( 1,3) 

\psline(1,3)( 2,4) 
\psline(1,4)( 2,3) 

\psline(2,3)( 3,4) 
\psline(2,4)( 3,3) 
 
\psline(3,3)( 4,4) 
\psline(3,4)( 4,3) 

\psline(4,3)( 5,4) 
\psline(4,4)( 5,3) 

\psline(5,3)( 6,4) 
\psline(5,4)( 6,3) 

\psline(-1,3)( 0,4) 
\psline(-1,4)( 0,3) 

\psline(-2,3)( -1,4) 
\psline(-2,4)( -1,3) 

\psline(-3,3)( -2,4) 
\psline(-3,4)( -2,3) 

\psline(-4,3)( -3,4) 
\psline(-4,4)( -3,3) 

\psline(-5,3)( -4,4) 
\psline(-5,4)( -4,3) 

\psline(-6,3)( -5,4) 
\psline(-6,4)( -5,3)

\psline(0,4)(1,5)
\psline(0,5)(1,4)

\psline(1,4)(2,5)
\psline(1,5)(2,4)

\psline(2,4)(3,5)
\psline(2,5)(3,4)

\psline(3,4)(4,5)
\psline(3,5)(4,4)

\psline(4,4)(5,5)
\psline(4,5)(5,4)

\psline(5,4)(6,5)
\psline(5,5)(6,4)

\psline(-1,4)(0,5)
\psline(-1,5)(0,4)

\psline(-2,4)(-1,5)
\psline(-2,5)(-1,4)

\psline(-3,4)(-2,5)
\psline(-3,5)(-2,4)

\psline(-4,4)(-3,5)
\psline(-4,5)(-3,4)

\psline(-5,4)(-4,5)
\psline(-5,5)(-4,4)

\psline(-6,4)(-5,5)
\psline(-6,5)(-5,4)

\psline(0,5)(1,6)
\psline(0,6)(1,5)

\psline(1,5)(2,6)
\psline(1,6)(2,5)

\psline(2,5)(3,6)
\psline(2,6)(3,5)

\psline(3,5)(4,6)
\psline(3,6)(4,5)

\psline(4,5)(5,6)
\psline(4,6)(5,5)

\psline(5,5)(6,6)
\psline(5,6)(6,5)

\psline(-1,5)(0,6)
\psline(-1,6)(0,5)

\psline(-2,5)(-1,6)
\psline(-2,6)(-1,5)

\psline(-3,5)(-2,6)
\psline(-3,6)(-2,5)

\psline(-4,5)(-3,6)
\psline(-4,6)(-3,5)

\psline(-5,5)(-4,6)
\psline(-5,6)(-4,5)

\psline(-6,5)(-5,6)
\psline(-6,6)(-5,5)

\psline(0,0)(1,-1)
\psline(0,-1)(1,0)

\psline(1,0)(2,-1)
\psline(1,-1)(2,0)

\psline(2,0)(3,-1)
\psline(2,-1)(3,0)

\psline(3,0)(4,-1)
\psline(3,-1)(4,0)

\psline(4,0)(5,-1)
\psline(4,-1)(5,0)

\psline(5,0)(6,-1)
\psline(5,-1)(6,0)

\psline(-1,0)(0,-1)
\psline(-1,-1)(0,0)

\psline(-2,0)(-1,-1)
\psline(-2,-1)(-1,0)

\psline(-3,0)(-2,-1)
\psline(-3,-1)(-2,0)

\psline(-4,0)(-3,-1)
\psline(-4,-1)(-3,0)

\psline(-5,0)(-4,-1)
\psline(-5,-1)(-4,0)

\psline(-6,0)(-5,-1)
\psline(-6,-1)(-5,0)

\psline(0,-1)(1,-2)
\psline(0,-2)(1,-1)

\psline(1,-1)(2,-2)
\psline(1,-2)(2,-1)

\psline(2,-1)(3,-2)
\psline(2,-2)(3,-1)

\psline(3,-1)(4,-2)
\psline(3,-2)(4,-1)

\psline(4,-1)(5,-2)
\psline(4,-2)(5,-1)

\psline(5,-1)(6,-2)
\psline(5,-2)(6,-1)

\psline(-1,-1)(0,-2)
\psline(-1,-2)(0,-1)

\psline(-2,-1)(-1,-2)
\psline(-2,-2)(-1,-1)

\psline(-3,-1)(-2,-2)
\psline(-3,-2)(-2,-1)

\psline(-4,-1)(-3,-2)
\psline(-4,-2)(-3,-1)

\psline(-5,-1)(-4,-2)
\psline(-5,-2)(-4,-1)

\psline(-6,-1)(-5,-2)
\psline(-6,-2)(-5,-1)

\psline(0,-2)(1,-3) 
\psline(0,-3)(1,-2) 

\psline(1,-2)(2,-3) 
\psline(1,-3)(2,-2) 

\psline(2,-2)(3,-3) 
\psline(2,-3)(3,-2) 
 
\psline(3,-2)(4,-3) 
\psline(3,-3)(4,-2) 

\psline(4,-2)(5,-3) 
\psline(4,-3)(5,-2) 

\psline(5,-2)(6,-3) 
\psline(5,-3)(6,-2) 

\psline(-1,-2)(0,-3) 
\psline(-1,-3)(0,-2) 

\psline(-2,-2)(-1,-3) 
\psline(-2,-3)(-1,-2) 

\psline(-3,-2)(-2,-3) 
\psline(-3,-3)(-2,-2) 

\psline(-4,-2)(-3,-3) 
\psline(-4,-3)(-3,-2) 

\psline(-5,-2)(-4,-3) 
\psline(-5,-3)(-4,-2) 

\psline(-6,-2)(-5,-3) 
\psline(-6,-3)(-5,-2)

\psline(0,-3)( 1,-4) 
\psline(0,-4)( 1,-3) 

\psline(1,-3)( 2,-4) 
\psline(1,-4)( 2,-3) 

\psline(2,-3)( 3,-4) 
\psline(2,-4)( 3,-3) 
 
\psline(3,-3)( 4,-4) 
\psline(3,-4)( 4,-3) 

\psline(4,-3)( 5,-4) 
\psline(4,-4)( 5,-3) 

\psline(5,-3)( 6,-4) 
\psline(5,-4)( 6,-3) 

\psline(-1,-3)( 0,-4) 
\psline(-1,-4)( 0,-3) 

\psline(-2,-3)( -1,-4) 
\psline(-2,-4)( -1,-3) 

\psline(-3,-3)( -2,-4) 
\psline(-3,-4)( -2,-3) 

\psline(-4,-3)( -3,-4) 
\psline(-4,-4)( -3,-3) 

\psline(-5,-3)( -4,-4) 
\psline(-5,-4)( -4,-3) 

\psline(-6,-3)( -5,-4) 
\psline(-6,-4)( -5,-3)

\psline(0,-4)(1,-5)
\psline(0,-5)(1,-4)

\psline(1,-4)(2,-5)
\psline(1,-5)(2,-4)

\psline(2,-4)(3,-5)
\psline(2,-5)(3,-4)

\psline(3,-4)(4,-5)
\psline(3,-5)(4,-4)

\psline(4,-4)(5,-5)
\psline(4,-5)(5,-4)

\psline(5,-4)(6,-5)
\psline(5,-5)(6,-4)

\psline(-1,-4)(0,-5)
\psline(-1,-5)(0,-4)

\psline(-2,-4)(-1,-5)
\psline(-2,-5)(-1,-4)

\psline(-3,-4)(-2,-5)
\psline(-3,-5)(-2,-4)

\psline(-4,-4)(-3,-5)
\psline(-4,-5)(-3,-4)

\psline(-5,-4)(-4,-5)
\psline(-5,-5)(-4,-4)

\psline(-6,-4)(-5,-5)
\psline(-6,-5)(-5,-4)

\psline(0,-5)(1,-6)
\psline(0,-6)(1,-5)

\psline(1,-5)(2,-6)
\psline(1,-6)(2,-5)

\psline(2,-5)(3,-6)
\psline(2,-6)(3,-5)

\psline(3,-5)(4,-6)
\psline(3,-6)(4,-5)

\psline(4,-5)(5,-6)
\psline(4,-6)(5,-5)

\psline(5,-5)(6,-6)
\psline(5,-6)(6,-5)

\psline(-1,-5)(0,-6)
\psline(-1,-6)(0,-5)

\psline(-2,-5)(-1,-6)
\psline(-2,-6)(-1,-5)

\psline(-3,-5)(-2,-6)
\psline(-3,-6)(-2,-5)

\psline(-4,-5)(-3,-6)
\psline(-4,-6)(-3,-5)

\psline(-5,-5)(-4,-6)
\psline(-5,-6)(-4,-5)

\psline(-6,-5)(-5,-6)
\psline(-6,-6)(-5,-5)

\rput(0,-6.7){The set $\bc^{L}_{\min}$ for $a>c$ and $b=0$}

\end{pspicture}
\end{center}
\end{textblock}


\begin{textblock}{10}(2,17.5)

\begin{center}
\begin{pspicture}(-6,6)(6,-6)
\psset{linewidth=.13mm}

\pspolygon[fillstyle=solid,fillcolor=black](0,0)(0,1)(.5,.5)

\pspolygon[fillstyle=solid,fillcolor=lightgray](0,0)(0,-6)(-6,-6)
\pspolygon[fillstyle=solid,fillcolor=lightgray](1,-1)(1,-6)(6,-6)
\pspolygon[fillstyle=solid,fillcolor=lightgray](2,0)(6,0)(6,-4)
\pspolygon[fillstyle=solid,fillcolor=lightgray](1,1)(6,6)(6,1)
\pspolygon[fillstyle=solid,fillcolor=lightgray](1,3)(1,6)(4,6)
\pspolygon[fillstyle=solid,fillcolor=lightgray](0,2)(0,6)(-4,6)
\pspolygon[fillstyle=solid,fillcolor=lightgray](-1,1)(-6,1)(-6,6)
\pspolygon[fillstyle=solid,fillcolor=lightgray](-2,0)(-6,0)(-6,-4)

\pspolygon[fillstyle=solid,fillcolor=lightgray](1,0)(2,0)(6,-4)(6,-5)(0,1)
\pspolygon[fillstyle=solid,fillcolor=lightgray](1,2)(1,3)(4,6)(5,6)
\pspolygon[fillstyle=solid,fillcolor=lightgray](0,1)(0,2)(-4,6)(-5,6)
\pspolygon[fillstyle=solid,fillcolor=lightgray](-1,0)(-2,0)(-6,-4)(-6,-5)

\rput(-.2,-.5){{\footnotesize $\ast$}}
\rput(-.2,1.5){{\footnotesize $\ast$}}

\psline(-6,-6)(-6,6)
\psline(-5,-6)(-5,6)
\psline(-4,-6)(-4,6)
\psline(-3,-6)(-3,6)
\psline(-2,-6)(-2,6)
\psline(-1,-6)(-1,6)
\psline(0,-6)(0,6)
\psline(1,-6)(1,6)
\psline(2,-6)(2,6)
\psline(3,-6)(3,6)
\psline(4,-6)(4,6)
\psline(5,-6)(5,6)
\psline(6,-6)(6,6)

\psline(-6,-6)(6,-6)
\psline(-6,-5)(6,-5)
\psline(-6,-4)(6,-4)
\psline(-6,-3)(6,-3)
\psline(-6,-2)(6,-2)
\psline(-6,-1)(6,-1)
\psline(-6,0)(6,0)
\psline(-6,1)(6,1)
\psline(-6,2)(6,2)
\psline(-6,3)(6,3)
\psline(-6,4)(6,4)
\psline(-6,5)(6,5)
\psline(-6,6)(6,6)

\psline(0,0)(1,1)
\psline(0,1)(1,0)

\psline(1,0)(2,1)
\psline(1,1)(2,0)

\psline(2,0)(3,1)
\psline(2,1)(3,0)

\psline(3,0)(4,1)
\psline(3,1)(4,0)

\psline(4,0)(5,1)
\psline(4,1)(5,0)

\psline(5,0)(6,1)
\psline(5,1)(6,0)

\psline(-1,0)(0,1)
\psline(-1,1)(0,0)

\psline(-2,0)(-1,1)
\psline(-2,1)(-1,0)

\psline(-3,0)(-2,1)
\psline(-3,1)(-2,0)

\psline(-4,0)(-3,1)
\psline(-4,1)(-3,0)

\psline(-5,0)(-4,1)
\psline(-5,1)(-4,0)

\psline(-6,0)(-5,1)
\psline(-6,1)(-5,0)

\psline(0,1)(1,2)
\psline(0,2)(1,1)

\psline(1,1)(2,2)
\psline(1,2)(2,1)

\psline(2,1)(3,2)
\psline(2,2)(3,1)

\psline(3,1)(4,2)
\psline(3,2)(4,1)

\psline(4,1)(5,2)
\psline(4,2)(5,1)

\psline(5,1)(6,2)
\psline(5,2)(6,1)

\psline(-1,1)(0,2)
\psline(-1,2)(0,1)

\psline(-2,1)(-1,2)
\psline(-2,2)(-1,1)

\psline(-3,1)(-2,2)
\psline(-3,2)(-2,1)

\psline(-4,1)(-3,2)
\psline(-4,2)(-3,1)

\psline(-5,1)(-4,2)
\psline(-5,2)(-4,1)

\psline(-6,1)(-5,2)
\psline(-6,2)(-5,1)

\psline(0,2)(1,3) 
\psline(0,3)(1,2) 

\psline(1,2)(2,3) 
\psline(1,3)(2,2) 

\psline(2,2)(3,3) 
\psline(2,3)(3,2) 
 
\psline(3,2)(4,3) 
\psline(3,3)(4,2) 

\psline(4,2)(5,3) 
\psline(4,3)(5,2) 

\psline(5,2)(6,3) 
\psline(5,3)(6,2) 

\psline(-1,2)(0,3) 
\psline(-1,3)(0,2) 

\psline(-2,2)(-1,3) 
\psline(-2,3)(-1,2) 

\psline(-3,2)(-2,3) 
\psline(-3,3)(-2,2) 

\psline(-4,2)(-3,3) 
\psline(-4,3)(-3,2) 

\psline(-5,2)(-4,3) 
\psline(-5,3)(-4,2) 

\psline(-6,2)(-5,3) 
\psline(-6,3)(-5,2)

\psline(0,3)( 1,4) 
\psline(0,4)( 1,3) 

\psline(1,3)( 2,4) 
\psline(1,4)( 2,3) 

\psline(2,3)( 3,4) 
\psline(2,4)( 3,3) 
 
\psline(3,3)( 4,4) 
\psline(3,4)( 4,3) 

\psline(4,3)( 5,4) 
\psline(4,4)( 5,3) 

\psline(5,3)( 6,4) 
\psline(5,4)( 6,3) 

\psline(-1,3)( 0,4) 
\psline(-1,4)( 0,3) 

\psline(-2,3)( -1,4) 
\psline(-2,4)( -1,3) 

\psline(-3,3)( -2,4) 
\psline(-3,4)( -2,3) 

\psline(-4,3)( -3,4) 
\psline(-4,4)( -3,3) 

\psline(-5,3)( -4,4) 
\psline(-5,4)( -4,3) 

\psline(-6,3)( -5,4) 
\psline(-6,4)( -5,3)

\psline(0,4)(1,5)
\psline(0,5)(1,4)

\psline(1,4)(2,5)
\psline(1,5)(2,4)

\psline(2,4)(3,5)
\psline(2,5)(3,4)

\psline(3,4)(4,5)
\psline(3,5)(4,4)

\psline(4,4)(5,5)
\psline(4,5)(5,4)

\psline(5,4)(6,5)
\psline(5,5)(6,4)

\psline(-1,4)(0,5)
\psline(-1,5)(0,4)

\psline(-2,4)(-1,5)
\psline(-2,5)(-1,4)

\psline(-3,4)(-2,5)
\psline(-3,5)(-2,4)

\psline(-4,4)(-3,5)
\psline(-4,5)(-3,4)

\psline(-5,4)(-4,5)
\psline(-5,5)(-4,4)

\psline(-6,4)(-5,5)
\psline(-6,5)(-5,4)

\psline(0,5)(1,6)
\psline(0,6)(1,5)

\psline(1,5)(2,6)
\psline(1,6)(2,5)

\psline(2,5)(3,6)
\psline(2,6)(3,5)

\psline(3,5)(4,6)
\psline(3,6)(4,5)

\psline(4,5)(5,6)
\psline(4,6)(5,5)

\psline(5,5)(6,6)
\psline(5,6)(6,5)

\psline(-1,5)(0,6)
\psline(-1,6)(0,5)

\psline(-2,5)(-1,6)
\psline(-2,6)(-1,5)

\psline(-3,5)(-2,6)
\psline(-3,6)(-2,5)

\psline(-4,5)(-3,6)
\psline(-4,6)(-3,5)

\psline(-5,5)(-4,6)
\psline(-5,6)(-4,5)

\psline(-6,5)(-5,6)
\psline(-6,6)(-5,5)

\psline(0,0)(1,-1)
\psline(0,-1)(1,0)

\psline(1,0)(2,-1)
\psline(1,-1)(2,0)

\psline(2,0)(3,-1)
\psline(2,-1)(3,0)

\psline(3,0)(4,-1)
\psline(3,-1)(4,0)

\psline(4,0)(5,-1)
\psline(4,-1)(5,0)

\psline(5,0)(6,-1)
\psline(5,-1)(6,0)

\psline(-1,0)(0,-1)
\psline(-1,-1)(0,0)

\psline(-2,0)(-1,-1)
\psline(-2,-1)(-1,0)

\psline(-3,0)(-2,-1)
\psline(-3,-1)(-2,0)

\psline(-4,0)(-3,-1)
\psline(-4,-1)(-3,0)

\psline(-5,0)(-4,-1)
\psline(-5,-1)(-4,0)

\psline(-6,0)(-5,-1)
\psline(-6,-1)(-5,0)

\psline(0,-1)(1,-2)
\psline(0,-2)(1,-1)

\psline(1,-1)(2,-2)
\psline(1,-2)(2,-1)

\psline(2,-1)(3,-2)
\psline(2,-2)(3,-1)

\psline(3,-1)(4,-2)
\psline(3,-2)(4,-1)

\psline(4,-1)(5,-2)
\psline(4,-2)(5,-1)

\psline(5,-1)(6,-2)
\psline(5,-2)(6,-1)

\psline(-1,-1)(0,-2)
\psline(-1,-2)(0,-1)

\psline(-2,-1)(-1,-2)
\psline(-2,-2)(-1,-1)

\psline(-3,-1)(-2,-2)
\psline(-3,-2)(-2,-1)

\psline(-4,-1)(-3,-2)
\psline(-4,-2)(-3,-1)

\psline(-5,-1)(-4,-2)
\psline(-5,-2)(-4,-1)

\psline(-6,-1)(-5,-2)
\psline(-6,-2)(-5,-1)

\psline(0,-2)(1,-3) 
\psline(0,-3)(1,-2) 

\psline(1,-2)(2,-3) 
\psline(1,-3)(2,-2) 

\psline(2,-2)(3,-3) 
\psline(2,-3)(3,-2) 
 
\psline(3,-2)(4,-3) 
\psline(3,-3)(4,-2) 

\psline(4,-2)(5,-3) 
\psline(4,-3)(5,-2) 

\psline(5,-2)(6,-3) 
\psline(5,-3)(6,-2) 

\psline(-1,-2)(0,-3) 
\psline(-1,-3)(0,-2) 

\psline(-2,-2)(-1,-3) 
\psline(-2,-3)(-1,-2) 

\psline(-3,-2)(-2,-3) 
\psline(-3,-3)(-2,-2) 

\psline(-4,-2)(-3,-3) 
\psline(-4,-3)(-3,-2) 

\psline(-5,-2)(-4,-3) 
\psline(-5,-3)(-4,-2) 

\psline(-6,-2)(-5,-3) 
\psline(-6,-3)(-5,-2)

\psline(0,-3)( 1,-4) 
\psline(0,-4)( 1,-3) 

\psline(1,-3)( 2,-4) 
\psline(1,-4)( 2,-3) 

\psline(2,-3)( 3,-4) 
\psline(2,-4)( 3,-3) 
 
\psline(3,-3)( 4,-4) 
\psline(3,-4)( 4,-3) 

\psline(4,-3)( 5,-4) 
\psline(4,-4)( 5,-3) 

\psline(5,-3)( 6,-4) 
\psline(5,-4)( 6,-3) 

\psline(-1,-3)( 0,-4) 
\psline(-1,-4)( 0,-3) 

\psline(-2,-3)( -1,-4) 
\psline(-2,-4)( -1,-3) 

\psline(-3,-3)( -2,-4) 
\psline(-3,-4)( -2,-3) 

\psline(-4,-3)( -3,-4) 
\psline(-4,-4)( -3,-3) 

\psline(-5,-3)( -4,-4) 
\psline(-5,-4)( -4,-3) 

\psline(-6,-3)( -5,-4) 
\psline(-6,-4)( -5,-3)

\psline(0,-4)(1,-5)
\psline(0,-5)(1,-4)

\psline(1,-4)(2,-5)
\psline(1,-5)(2,-4)

\psline(2,-4)(3,-5)
\psline(2,-5)(3,-4)

\psline(3,-4)(4,-5)
\psline(3,-5)(4,-4)

\psline(4,-4)(5,-5)
\psline(4,-5)(5,-4)

\psline(5,-4)(6,-5)
\psline(5,-5)(6,-4)

\psline(-1,-4)(0,-5)
\psline(-1,-5)(0,-4)

\psline(-2,-4)(-1,-5)
\psline(-2,-5)(-1,-4)

\psline(-3,-4)(-2,-5)
\psline(-3,-5)(-2,-4)

\psline(-4,-4)(-3,-5)
\psline(-4,-5)(-3,-4)

\psline(-5,-4)(-4,-5)
\psline(-5,-5)(-4,-4)

\psline(-6,-4)(-5,-5)
\psline(-6,-5)(-5,-4)

\psline(0,-5)(1,-6)
\psline(0,-6)(1,-5)

\psline(1,-5)(2,-6)
\psline(1,-6)(2,-5)

\psline(2,-5)(3,-6)
\psline(2,-6)(3,-5)

\psline(3,-5)(4,-6)
\psline(3,-6)(4,-5)

\psline(4,-5)(5,-6)
\psline(4,-6)(5,-5)

\psline(5,-5)(6,-6)
\psline(5,-6)(6,-5)

\psline(-1,-5)(0,-6)
\psline(-1,-6)(0,-5)

\psline(-2,-5)(-1,-6)
\psline(-2,-6)(-1,-5)

\psline(-3,-5)(-2,-6)
\psline(-3,-6)(-2,-5)

\psline(-4,-5)(-3,-6)
\psline(-4,-6)(-3,-5)

\psline(-5,-5)(-4,-6)
\psline(-5,-6)(-4,-5)

\psline(-6,-5)(-5,-6)
\psline(-6,-6)(-5,-5)

\rput(0,-6.7){The set $\bc^{L}_{\min}$ for $a=c>0$ and $b>0$}

\end{pspicture}
\end{center}
\end{textblock}


\begin{textblock}{10}(9.55,17.5)
\begin{center}
\begin{pspicture}(-6,-6)(6,6)

\pspolygon[fillstyle=solid,fillcolor=black](0,0)(0,1)(.5,.5)

\pspolygon[fillstyle=solid,fillcolor=lightgray](0,0)(0,-6)(-6,-6)
\pspolygon[fillstyle=solid,fillcolor=lightgray](1,-1)(1,-6)(6,-6)
\pspolygon[fillstyle=solid,fillcolor=lightgray](2,0)(6,0)(6,-4)
\pspolygon[fillstyle=solid,fillcolor=lightgray](1,1)(6,6)(6,1)
\pspolygon[fillstyle=solid,fillcolor=lightgray](1,3)(1,6)(4,6)
\pspolygon[fillstyle=solid,fillcolor=lightgray](0,2)(0,6)(-4,6)
\pspolygon[fillstyle=solid,fillcolor=lightgray](-1,1)(-6,1)(-6,6)
\pspolygon[fillstyle=solid,fillcolor=lightgray](-2,0)(-6,0)(-6,-4)
\pspolygon[fillstyle=solid,fillcolor=lightgray](-1,0)(-2,0)(-6,-4)(-6,-5)
\pspolygon[fillstyle=solid,fillcolor=lightgray](1,0)(2,0)(6,-4)(6,-5)
\pspolygon[fillstyle=solid,fillcolor=lightgray](1,2)(1,3)(4,6)(5,6)
\pspolygon[fillstyle=solid,fillcolor=lightgray](-0.5,1.5)(-5,6)(-4,6)(0,2)(0,1)

\pspolygon[fillstyle=solid,fillcolor=lightgray](0,1)(.5,1.5)(0,2)

\pspolygon[fillstyle=solid,fillcolor=lightgray](0,0)(0,-1)(0.5,-.5)

\pspolygon[fillstyle=solid,fillcolor=lightgray](0.5,1.5)(0,2)(0,6)(1,6)(1,2)
\pspolygon[fillstyle=solid,fillcolor=lightgray](.5,.5)(1,1)(6,1)(6,0)(1,0)
\pspolygon[fillstyle=solid,fillcolor=lightgray](.5,-.5)(0,-1)(0,-6)(1,-6)(1,-1)
\pspolygon[fillstyle=solid,fillcolor=lightgray](-.5,.5)(-1,1)(-6,1)(-6,0)(-1,0)

\rput(-.2,-.5){{\footnotesize $\ast$}}
\rput(.2,-.5){{\footnotesize $\ast$}}

\rput(-.2,1.5){{\footnotesize $\ast$}}
\rput(.2,1.5){{\footnotesize $\ast$}}

\rput(.8,.5){{\footnotesize $\ast$}}

\psline(-6,-6)(-6,6)
\psline(-5,-6)(-5,6)
\psline(-4,-6)(-4,6)
\psline(-3,-6)(-3,6)
\psline(-2,-6)(-2,6)
\psline(-1,-6)(-1,6)
\psline(0,-6)(0,6)
\psline(1,-6)(1,6)
\psline(2,-6)(2,6)
\psline(3,-6)(3,6)
\psline(4,-6)(4,6)
\psline(5,-6)(5,6)
\psline(6,-6)(6,6)

\psline(-6,-6)(6,-6)
\psline(-6,-5)(6,-5)
\psline(-6,-4)(6,-4)
\psline(-6,-3)(6,-3)
\psline(-6,-2)(6,-2)
\psline(-6,-1)(6,-1)
\psline(-6,0)(6,0)
\psline(-6,1)(6,1)
\psline(-6,2)(6,2)
\psline(-6,3)(6,3)
\psline(-6,4)(6,4)
\psline(-6,5)(6,5)
\psline(-6,6)(6,6)

\psline(0,0)(1,1)
\psline(0,1)(1,0)

\psline(1,0)(2,1)
\psline(1,1)(2,0)

\psline(2,0)(3,1)
\psline(2,1)(3,0)

\psline(3,0)(4,1)
\psline(3,1)(4,0)

\psline(4,0)(5,1)
\psline(4,1)(5,0)

\psline(5,0)(6,1)
\psline(5,1)(6,0)

\psline(-1,0)(0,1)
\psline(-1,1)(0,0)

\psline(-2,0)(-1,1)
\psline(-2,1)(-1,0)

\psline(-3,0)(-2,1)
\psline(-3,1)(-2,0)

\psline(-4,0)(-3,1)
\psline(-4,1)(-3,0)

\psline(-5,0)(-4,1)
\psline(-5,1)(-4,0)

\psline(-6,0)(-5,1)
\psline(-6,1)(-5,0)

\psline(0,1)(1,2)
\psline(0,2)(1,1)

\psline(1,1)(2,2)
\psline(1,2)(2,1)

\psline(2,1)(3,2)
\psline(2,2)(3,1)

\psline(3,1)(4,2)
\psline(3,2)(4,1)

\psline(4,1)(5,2)
\psline(4,2)(5,1)

\psline(5,1)(6,2)
\psline(5,2)(6,1)

\psline(-1,1)(0,2)
\psline(-1,2)(0,1)

\psline(-2,1)(-1,2)
\psline(-2,2)(-1,1)

\psline(-3,1)(-2,2)
\psline(-3,2)(-2,1)

\psline(-4,1)(-3,2)
\psline(-4,2)(-3,1)

\psline(-5,1)(-4,2)
\psline(-5,2)(-4,1)

\psline(-6,1)(-5,2)
\psline(-6,2)(-5,1)

\psline(0,2)(1,3) 
\psline(0,3)(1,2) 

\psline(1,2)(2,3) 
\psline(1,3)(2,2) 

\psline(2,2)(3,3) 
\psline(2,3)(3,2) 
 
\psline(3,2)(4,3) 
\psline(3,3)(4,2) 

\psline(4,2)(5,3) 
\psline(4,3)(5,2) 

\psline(5,2)(6,3) 
\psline(5,3)(6,2) 

\psline(-1,2)(0,3) 
\psline(-1,3)(0,2) 

\psline(-2,2)(-1,3) 
\psline(-2,3)(-1,2) 

\psline(-3,2)(-2,3) 
\psline(-3,3)(-2,2) 

\psline(-4,2)(-3,3) 
\psline(-4,3)(-3,2) 

\psline(-5,2)(-4,3) 
\psline(-5,3)(-4,2) 

\psline(-6,2)(-5,3) 
\psline(-6,3)(-5,2)

\psline(0,3)( 1,4) 
\psline(0,4)( 1,3) 

\psline(1,3)( 2,4) 
\psline(1,4)( 2,3) 

\psline(2,3)( 3,4) 
\psline(2,4)( 3,3) 
 
\psline(3,3)( 4,4) 
\psline(3,4)( 4,3) 

\psline(4,3)( 5,4) 
\psline(4,4)( 5,3) 

\psline(5,3)( 6,4) 
\psline(5,4)( 6,3) 

\psline(-1,3)( 0,4) 
\psline(-1,4)( 0,3) 

\psline(-2,3)( -1,4) 
\psline(-2,4)( -1,3) 

\psline(-3,3)( -2,4) 
\psline(-3,4)( -2,3) 

\psline(-4,3)( -3,4) 
\psline(-4,4)( -3,3) 

\psline(-5,3)( -4,4) 
\psline(-5,4)( -4,3) 

\psline(-6,3)( -5,4) 
\psline(-6,4)( -5,3)

\psline(0,4)(1,5)
\psline(0,5)(1,4)

\psline(1,4)(2,5)
\psline(1,5)(2,4)

\psline(2,4)(3,5)
\psline(2,5)(3,4)

\psline(3,4)(4,5)
\psline(3,5)(4,4)

\psline(4,4)(5,5)
\psline(4,5)(5,4)

\psline(5,4)(6,5)
\psline(5,5)(6,4)

\psline(-1,4)(0,5)
\psline(-1,5)(0,4)

\psline(-2,4)(-1,5)
\psline(-2,5)(-1,4)

\psline(-3,4)(-2,5)
\psline(-3,5)(-2,4)

\psline(-4,4)(-3,5)
\psline(-4,5)(-3,4)

\psline(-5,4)(-4,5)
\psline(-5,5)(-4,4)

\psline(-6,4)(-5,5)
\psline(-6,5)(-5,4)

\psline(0,5)(1,6)
\psline(0,6)(1,5)

\psline(1,5)(2,6)
\psline(1,6)(2,5)

\psline(2,5)(3,6)
\psline(2,6)(3,5)

\psline(3,5)(4,6)
\psline(3,6)(4,5)

\psline(4,5)(5,6)
\psline(4,6)(5,5)

\psline(5,5)(6,6)
\psline(5,6)(6,5)

\psline(-1,5)(0,6)
\psline(-1,6)(0,5)

\psline(-2,5)(-1,6)
\psline(-2,6)(-1,5)

\psline(-3,5)(-2,6)
\psline(-3,6)(-2,5)

\psline(-4,5)(-3,6)
\psline(-4,6)(-3,5)

\psline(-5,5)(-4,6)
\psline(-5,6)(-4,5)

\psline(-6,5)(-5,6)
\psline(-6,6)(-5,5)

\psline(0,0)(1,-1)
\psline(0,-1)(1,0)

\psline(1,0)(2,-1)
\psline(1,-1)(2,0)

\psline(2,0)(3,-1)
\psline(2,-1)(3,0)

\psline(3,0)(4,-1)
\psline(3,-1)(4,0)

\psline(4,0)(5,-1)
\psline(4,-1)(5,0)

\psline(5,0)(6,-1)
\psline(5,-1)(6,0)

\psline(-1,0)(0,-1)
\psline(-1,-1)(0,0)

\psline(-2,0)(-1,-1)
\psline(-2,-1)(-1,0)

\psline(-3,0)(-2,-1)
\psline(-3,-1)(-2,0)

\psline(-4,0)(-3,-1)
\psline(-4,-1)(-3,0)

\psline(-5,0)(-4,-1)
\psline(-5,-1)(-4,0)

\psline(-6,0)(-5,-1)
\psline(-6,-1)(-5,0)

\psline(0,-1)(1,-2)
\psline(0,-2)(1,-1)

\psline(1,-1)(2,-2)
\psline(1,-2)(2,-1)

\psline(2,-1)(3,-2)
\psline(2,-2)(3,-1)

\psline(3,-1)(4,-2)
\psline(3,-2)(4,-1)

\psline(4,-1)(5,-2)
\psline(4,-2)(5,-1)

\psline(5,-1)(6,-2)
\psline(5,-2)(6,-1)

\psline(-1,-1)(0,-2)
\psline(-1,-2)(0,-1)

\psline(-2,-1)(-1,-2)
\psline(-2,-2)(-1,-1)

\psline(-3,-1)(-2,-2)
\psline(-3,-2)(-2,-1)

\psline(-4,-1)(-3,-2)
\psline(-4,-2)(-3,-1)

\psline(-5,-1)(-4,-2)
\psline(-5,-2)(-4,-1)

\psline(-6,-1)(-5,-2)
\psline(-6,-2)(-5,-1)

\psline(0,-2)(1,-3) 
\psline(0,-3)(1,-2) 

\psline(1,-2)(2,-3) 
\psline(1,-3)(2,-2) 

\psline(2,-2)(3,-3) 
\psline(2,-3)(3,-2) 
 
\psline(3,-2)(4,-3) 
\psline(3,-3)(4,-2) 

\psline(4,-2)(5,-3) 
\psline(4,-3)(5,-2) 

\psline(5,-2)(6,-3) 
\psline(5,-3)(6,-2) 

\psline(-1,-2)(0,-3) 
\psline(-1,-3)(0,-2) 

\psline(-2,-2)(-1,-3) 
\psline(-2,-3)(-1,-2) 

\psline(-3,-2)(-2,-3) 
\psline(-3,-3)(-2,-2) 

\psline(-4,-2)(-3,-3) 
\psline(-4,-3)(-3,-2) 

\psline(-5,-2)(-4,-3) 
\psline(-5,-3)(-4,-2) 

\psline(-6,-2)(-5,-3) 
\psline(-6,-3)(-5,-2)

\psline(0,-3)( 1,-4) 
\psline(0,-4)( 1,-3) 

\psline(1,-3)( 2,-4) 
\psline(1,-4)( 2,-3) 

\psline(2,-3)( 3,-4) 
\psline(2,-4)( 3,-3) 
 
\psline(3,-3)( 4,-4) 
\psline(3,-4)( 4,-3) 

\psline(4,-3)( 5,-4) 
\psline(4,-4)( 5,-3) 

\psline(5,-3)( 6,-4) 
\psline(5,-4)( 6,-3) 

\psline(-1,-3)( 0,-4) 
\psline(-1,-4)( 0,-3) 

\psline(-2,-3)( -1,-4) 
\psline(-2,-4)( -1,-3) 

\psline(-3,-3)( -2,-4) 
\psline(-3,-4)( -2,-3) 

\psline(-4,-3)( -3,-4) 
\psline(-4,-4)( -3,-3) 

\psline(-5,-3)( -4,-4) 
\psline(-5,-4)( -4,-3) 

\psline(-6,-3)( -5,-4) 
\psline(-6,-4)( -5,-3)

\psline(0,-4)(1,-5)
\psline(0,-5)(1,-4)

\psline(1,-4)(2,-5)
\psline(1,-5)(2,-4)

\psline(2,-4)(3,-5)
\psline(2,-5)(3,-4)

\psline(3,-4)(4,-5)
\psline(3,-5)(4,-4)

\psline(4,-4)(5,-5)
\psline(4,-5)(5,-4)

\psline(5,-4)(6,-5)
\psline(5,-5)(6,-4)

\psline(-1,-4)(0,-5)
\psline(-1,-5)(0,-4)

\psline(-2,-4)(-1,-5)
\psline(-2,-5)(-1,-4)

\psline(-3,-4)(-2,-5)
\psline(-3,-5)(-2,-4)

\psline(-4,-4)(-3,-5)
\psline(-4,-5)(-3,-4)

\psline(-5,-4)(-4,-5)
\psline(-5,-5)(-4,-4)

\psline(-6,-4)(-5,-5)
\psline(-6,-5)(-5,-4)

\psline(0,-5)(1,-6)
\psline(0,-6)(1,-5)

\psline(1,-5)(2,-6)
\psline(1,-6)(2,-5)

\psline(2,-5)(3,-6)
\psline(2,-6)(3,-5)

\psline(3,-5)(4,-6)
\psline(3,-6)(4,-5)

\psline(4,-5)(5,-6)
\psline(4,-6)(5,-5)

\psline(5,-5)(6,-6)
\psline(5,-6)(6,-5)

\psline(-1,-5)(0,-6)
\psline(-1,-6)(0,-5)

\psline(-2,-5)(-1,-6)
\psline(-2,-6)(-1,-5)

\psline(-3,-5)(-2,-6)
\psline(-3,-6)(-2,-5)

\psline(-4,-5)(-3,-6)
\psline(-4,-6)(-3,-5)

\psline(-5,-5)(-4,-6)
\psline(-5,-6)(-4,-5)

\psline(-6,-5)(-5,-6)
\psline(-6,-6)(-5,-5)

\rput(0,-6.7){The set $\bc^{L}_{\min}$ for $a=c>0$ and $b=0$}

\end{pspicture}
\end{center}
\end{textblock}

$\ $\\
\newpage


\begin{textblock}{10}(5.55,4)
\begin{center}
\begin{pspicture}(-6,-6)(6,6)

\pspolygon[fillstyle=solid,fillcolor=black](0,0)(0,1)(.5,.5)

\pspolygon[fillstyle=solid,fillcolor=lightgray](0,0)(0,-6)(-6,-6)
\pspolygon[fillstyle=solid,fillcolor=lightgray](1,-1)(1,-6)(6,-6)
\pspolygon[fillstyle=solid,fillcolor=lightgray](2,0)(6,0)(6,-4)
\pspolygon[fillstyle=solid,fillcolor=lightgray](1,1)(6,6)(6,1)
\pspolygon[fillstyle=solid,fillcolor=lightgray](1,3)(1,6)(4,6)
\pspolygon[fillstyle=solid,fillcolor=lightgray](0,2)(0,6)(-4,6)
\pspolygon[fillstyle=solid,fillcolor=lightgray](-1,1)(-6,1)(-6,6)
\pspolygon[fillstyle=solid,fillcolor=lightgray](-2,0)(-6,0)(-6,-4)

\pspolygon[fillstyle=solid,fillcolor=lightgray](-1,0)(-2,0)(-6,-4)(-6,-5)
\pspolygon[fillstyle=solid,fillcolor=lightgray](1,0)(2,0)(6,-4)(6,-5)
\pspolygon[fillstyle=solid,fillcolor=lightgray](1,2)(1,3)(4,6)(5,6)
\pspolygon[fillstyle=solid,fillcolor=lightgray](-0.5,1.5)(-5,6)(-4,6)(0,2)(0,1)

\pspolygon[fillstyle=solid,fillcolor=lightgray](0,1)(-1,1)(-6,6)(-5,6)
\pspolygon[fillstyle=solid,fillcolor=lightgray](1,2)(2,2)(6,6)(5,6)
\pspolygon[fillstyle=solid,fillcolor=lightgray](1,0)(1,-1)(6,-6)(6,-5)
\pspolygon[fillstyle=solid,fillcolor=lightgray](-1,0)(-1,-1)(-6,-6)(-6,-5)

\pspolygon[fillstyle=solid,fillcolor=lightgray](0,0)(-1,0)(-1,-1)
\pspolygon[fillstyle=solid,fillcolor=lightgray](1,1)(2,2)(1,2)

\rput(-.2,-.5){{\footnotesize $\ast$}}
\rput(-.5,-.2){{\footnotesize $\ast$}}
\rput(-.5,1.2){{\footnotesize $\ast$}}
\rput(-.2,1.5){{\footnotesize $\ast$}}

\psline(-6,-6)(-6,6)
\psline(-5,-6)(-5,6)
\psline(-4,-6)(-4,6)
\psline(-3,-6)(-3,6)
\psline(-2,-6)(-2,6)
\psline(-1,-6)(-1,6)
\psline(0,-6)(0,6)
\psline(1,-6)(1,6)
\psline(2,-6)(2,6)
\psline(3,-6)(3,6)
\psline(4,-6)(4,6)
\psline(5,-6)(5,6)
\psline(6,-6)(6,6)

\psline(-6,-6)(6,-6)
\psline(-6,-5)(6,-5)
\psline(-6,-4)(6,-4)
\psline(-6,-3)(6,-3)
\psline(-6,-2)(6,-2)
\psline(-6,-1)(6,-1)
\psline(-6,0)(6,0)
\psline(-6,1)(6,1)
\psline(-6,2)(6,2)
\psline(-6,3)(6,3)
\psline(-6,4)(6,4)
\psline(-6,5)(6,5)
\psline(-6,6)(6,6)

\psline(0,0)(1,1)
\psline(0,1)(1,0)

\psline(1,0)(2,1)
\psline(1,1)(2,0)

\psline(2,0)(3,1)
\psline(2,1)(3,0)

\psline(3,0)(4,1)
\psline(3,1)(4,0)

\psline(4,0)(5,1)
\psline(4,1)(5,0)

\psline(5,0)(6,1)
\psline(5,1)(6,0)

\psline(-1,0)(0,1)
\psline(-1,1)(0,0)

\psline(-2,0)(-1,1)
\psline(-2,1)(-1,0)

\psline(-3,0)(-2,1)
\psline(-3,1)(-2,0)

\psline(-4,0)(-3,1)
\psline(-4,1)(-3,0)

\psline(-5,0)(-4,1)
\psline(-5,1)(-4,0)

\psline(-6,0)(-5,1)
\psline(-6,1)(-5,0)

\psline(0,1)(1,2)
\psline(0,2)(1,1)

\psline(1,1)(2,2)
\psline(1,2)(2,1)

\psline(2,1)(3,2)
\psline(2,2)(3,1)

\psline(3,1)(4,2)
\psline(3,2)(4,1)

\psline(4,1)(5,2)
\psline(4,2)(5,1)

\psline(5,1)(6,2)
\psline(5,2)(6,1)

\psline(-1,1)(0,2)
\psline(-1,2)(0,1)

\psline(-2,1)(-1,2)
\psline(-2,2)(-1,1)

\psline(-3,1)(-2,2)
\psline(-3,2)(-2,1)

\psline(-4,1)(-3,2)
\psline(-4,2)(-3,1)

\psline(-5,1)(-4,2)
\psline(-5,2)(-4,1)

\psline(-6,1)(-5,2)
\psline(-6,2)(-5,1)

\psline(0,2)(1,3) 
\psline(0,3)(1,2) 

\psline(1,2)(2,3) 
\psline(1,3)(2,2) 

\psline(2,2)(3,3) 
\psline(2,3)(3,2) 
 
\psline(3,2)(4,3) 
\psline(3,3)(4,2) 

\psline(4,2)(5,3) 
\psline(4,3)(5,2) 

\psline(5,2)(6,3) 
\psline(5,3)(6,2) 

\psline(-1,2)(0,3) 
\psline(-1,3)(0,2) 

\psline(-2,2)(-1,3) 
\psline(-2,3)(-1,2) 

\psline(-3,2)(-2,3) 
\psline(-3,3)(-2,2) 

\psline(-4,2)(-3,3) 
\psline(-4,3)(-3,2) 

\psline(-5,2)(-4,3) 
\psline(-5,3)(-4,2) 

\psline(-6,2)(-5,3) 
\psline(-6,3)(-5,2)

\psline(0,3)( 1,4) 
\psline(0,4)( 1,3) 

\psline(1,3)( 2,4) 
\psline(1,4)( 2,3) 

\psline(2,3)( 3,4) 
\psline(2,4)( 3,3) 
 
\psline(3,3)( 4,4) 
\psline(3,4)( 4,3) 

\psline(4,3)( 5,4) 
\psline(4,4)( 5,3) 

\psline(5,3)( 6,4) 
\psline(5,4)( 6,3) 

\psline(-1,3)( 0,4) 
\psline(-1,4)( 0,3) 

\psline(-2,3)( -1,4) 
\psline(-2,4)( -1,3) 

\psline(-3,3)( -2,4) 
\psline(-3,4)( -2,3) 

\psline(-4,3)( -3,4) 
\psline(-4,4)( -3,3) 

\psline(-5,3)( -4,4) 
\psline(-5,4)( -4,3) 

\psline(-6,3)( -5,4) 
\psline(-6,4)( -5,3)

\psline(0,4)(1,5)
\psline(0,5)(1,4)

\psline(1,4)(2,5)
\psline(1,5)(2,4)

\psline(2,4)(3,5)
\psline(2,5)(3,4)

\psline(3,4)(4,5)
\psline(3,5)(4,4)

\psline(4,4)(5,5)
\psline(4,5)(5,4)

\psline(5,4)(6,5)
\psline(5,5)(6,4)

\psline(-1,4)(0,5)
\psline(-1,5)(0,4)

\psline(-2,4)(-1,5)
\psline(-2,5)(-1,4)

\psline(-3,4)(-2,5)
\psline(-3,5)(-2,4)

\psline(-4,4)(-3,5)
\psline(-4,5)(-3,4)

\psline(-5,4)(-4,5)
\psline(-5,5)(-4,4)

\psline(-6,4)(-5,5)
\psline(-6,5)(-5,4)

\psline(0,5)(1,6)
\psline(0,6)(1,5)

\psline(1,5)(2,6)
\psline(1,6)(2,5)

\psline(2,5)(3,6)
\psline(2,6)(3,5)

\psline(3,5)(4,6)
\psline(3,6)(4,5)

\psline(4,5)(5,6)
\psline(4,6)(5,5)

\psline(5,5)(6,6)
\psline(5,6)(6,5)

\psline(-1,5)(0,6)
\psline(-1,6)(0,5)

\psline(-2,5)(-1,6)
\psline(-2,6)(-1,5)

\psline(-3,5)(-2,6)
\psline(-3,6)(-2,5)

\psline(-4,5)(-3,6)
\psline(-4,6)(-3,5)

\psline(-5,5)(-4,6)
\psline(-5,6)(-4,5)

\psline(-6,5)(-5,6)
\psline(-6,6)(-5,5)

\psline(0,0)(1,-1)
\psline(0,-1)(1,0)

\psline(1,0)(2,-1)
\psline(1,-1)(2,0)

\psline(2,0)(3,-1)
\psline(2,-1)(3,0)

\psline(3,0)(4,-1)
\psline(3,-1)(4,0)

\psline(4,0)(5,-1)
\psline(4,-1)(5,0)

\psline(5,0)(6,-1)
\psline(5,-1)(6,0)

\psline(-1,0)(0,-1)
\psline(-1,-1)(0,0)

\psline(-2,0)(-1,-1)
\psline(-2,-1)(-1,0)

\psline(-3,0)(-2,-1)
\psline(-3,-1)(-2,0)

\psline(-4,0)(-3,-1)
\psline(-4,-1)(-3,0)

\psline(-5,0)(-4,-1)
\psline(-5,-1)(-4,0)

\psline(-6,0)(-5,-1)
\psline(-6,-1)(-5,0)

\psline(0,-1)(1,-2)
\psline(0,-2)(1,-1)

\psline(1,-1)(2,-2)
\psline(1,-2)(2,-1)

\psline(2,-1)(3,-2)
\psline(2,-2)(3,-1)

\psline(3,-1)(4,-2)
\psline(3,-2)(4,-1)

\psline(4,-1)(5,-2)
\psline(4,-2)(5,-1)

\psline(5,-1)(6,-2)
\psline(5,-2)(6,-1)

\psline(-1,-1)(0,-2)
\psline(-1,-2)(0,-1)

\psline(-2,-1)(-1,-2)
\psline(-2,-2)(-1,-1)

\psline(-3,-1)(-2,-2)
\psline(-3,-2)(-2,-1)

\psline(-4,-1)(-3,-2)
\psline(-4,-2)(-3,-1)

\psline(-5,-1)(-4,-2)
\psline(-5,-2)(-4,-1)

\psline(-6,-1)(-5,-2)
\psline(-6,-2)(-5,-1)

\psline(0,-2)(1,-3) 
\psline(0,-3)(1,-2) 

\psline(1,-2)(2,-3) 
\psline(1,-3)(2,-2) 

\psline(2,-2)(3,-3) 
\psline(2,-3)(3,-2) 
 
\psline(3,-2)(4,-3) 
\psline(3,-3)(4,-2) 

\psline(4,-2)(5,-3) 
\psline(4,-3)(5,-2) 

\psline(5,-2)(6,-3) 
\psline(5,-3)(6,-2) 

\psline(-1,-2)(0,-3) 
\psline(-1,-3)(0,-2) 

\psline(-2,-2)(-1,-3) 
\psline(-2,-3)(-1,-2) 

\psline(-3,-2)(-2,-3) 
\psline(-3,-3)(-2,-2) 

\psline(-4,-2)(-3,-3) 
\psline(-4,-3)(-3,-2) 

\psline(-5,-2)(-4,-3) 
\psline(-5,-3)(-4,-2) 

\psline(-6,-2)(-5,-3) 
\psline(-6,-3)(-5,-2)

\psline(0,-3)( 1,-4) 
\psline(0,-4)( 1,-3) 

\psline(1,-3)( 2,-4) 
\psline(1,-4)( 2,-3) 

\psline(2,-3)( 3,-4) 
\psline(2,-4)( 3,-3) 
 
\psline(3,-3)( 4,-4) 
\psline(3,-4)( 4,-3) 

\psline(4,-3)( 5,-4) 
\psline(4,-4)( 5,-3) 

\psline(5,-3)( 6,-4) 
\psline(5,-4)( 6,-3) 

\psline(-1,-3)( 0,-4) 
\psline(-1,-4)( 0,-3) 

\psline(-2,-3)( -1,-4) 
\psline(-2,-4)( -1,-3) 

\psline(-3,-3)( -2,-4) 
\psline(-3,-4)( -2,-3) 

\psline(-4,-3)( -3,-4) 
\psline(-4,-4)( -3,-3) 

\psline(-5,-3)( -4,-4) 
\psline(-5,-4)( -4,-3) 

\psline(-6,-3)( -5,-4) 
\psline(-6,-4)( -5,-3)

\psline(0,-4)(1,-5)
\psline(0,-5)(1,-4)

\psline(1,-4)(2,-5)
\psline(1,-5)(2,-4)

\psline(2,-4)(3,-5)
\psline(2,-5)(3,-4)

\psline(3,-4)(4,-5)
\psline(3,-5)(4,-4)

\psline(4,-4)(5,-5)
\psline(4,-5)(5,-4)

\psline(5,-4)(6,-5)
\psline(5,-5)(6,-4)

\psline(-1,-4)(0,-5)
\psline(-1,-5)(0,-4)

\psline(-2,-4)(-1,-5)
\psline(-2,-5)(-1,-4)

\psline(-3,-4)(-2,-5)
\psline(-3,-5)(-2,-4)

\psline(-4,-4)(-3,-5)
\psline(-4,-5)(-3,-4)

\psline(-5,-4)(-4,-5)
\psline(-5,-5)(-4,-4)

\psline(-6,-4)(-5,-5)
\psline(-6,-5)(-5,-4)

\psline(0,-5)(1,-6)
\psline(0,-6)(1,-5)

\psline(1,-5)(2,-6)
\psline(1,-6)(2,-5)

\psline(2,-5)(3,-6)
\psline(2,-6)(3,-5)

\psline(3,-5)(4,-6)
\psline(3,-6)(4,-5)

\psline(4,-5)(5,-6)
\psline(4,-6)(5,-5)

\psline(5,-5)(6,-6)
\psline(5,-6)(6,-5)

\psline(-1,-5)(0,-6)
\psline(-1,-6)(0,-5)

\psline(-2,-5)(-1,-6)
\psline(-2,-6)(-1,-5)

\psline(-3,-5)(-2,-6)
\psline(-3,-6)(-2,-5)

\psline(-4,-5)(-3,-6)
\psline(-4,-6)(-3,-5)

\psline(-5,-5)(-4,-6)
\psline(-5,-6)(-4,-5)

\psline(-6,-5)(-5,-6)
\psline(-6,-6)(-5,-5)

\rput(0,-6.7){The set $\bc^{L}_{\min}$ for $a=c=0$ and $b>0$}

\end{pspicture}
\end{center}
\end{textblock}
\end{Exa}

$\ $\\
\vspace{6cm}


\subsection{Some alternative description of $\cb^{L}_{\min}$}
Let $\LC_L(W)$ be the set of finite sequences $(w_1,\dots,w_n)$ 
of elements of $W$ such that $L(w_1\cdots w_n)=L(w_1)+\cdots + L(w_n)$. 
\begin{Exa}
For instance, $\LC_\ell(W)=\LC(W)$. Note also that $\LC(W) \subset \LC_L(W)$, 
by definition of a weight function: the inclusion might be strict, as it is 
shown by the case where $L=0$. Finally, if $L(s) > 0$ for all $s \in S$, 
then $\LC_L(W)=\LC(W)$.\finl
\end{Exa}
\begin{Exa}\label{exemple:zero}
If $L=0$, then $\nu_L=0$, $\spe_L(W)$ is the set of $0$-dimensional facets, $\LC_L(W)$ is the 
set of finite sequences of elements of $W$ and 
$\cb^L_{\min}=W$.\finl
\end{Exa}

For $A,B\in\alc(\Phi)$, we set
$$H^{L}(A,B)=\{H\in \cF^{L}| H\text{ separates $A$ and $B$}\}.$$
Then we have (compare to Proposition \ref{length-hyp}):
\begin{Prop}
\label{geom-L}
Let $x,y\in W$ and $A\in\alc(W)$. We have
\begin{enumerate}
\item $L(x)=\sum_{H\in H^{L}(A,xA)} L_{H}$;
\item $(x,y)\in\cL_{L}(W)$ if and only if $H^{L}(A,yA)\cap H^{L}(yA,xyA)=\emptyset$.
\end{enumerate}
\end{Prop}

\bigskip

The set $\cb^{L}_{\min}$ can be described as follows.
\bigskip

\begin{Prop}
\label{c-min}
The following equalities hold:
\eqna
\cb^L_{\min} &=& \{xwy~|~w \in \WC, (x,w,y) \in \LC_L(W) \text{ and } L(w)=\nu_L\} \\
&=& \{xw_\l y~|~\l \in \spe_L(W)\text{ and } (x,w_\l,y) \in \LC_L(W)\} \\
&=& \{w~|~w(A_0) \not\subset \UC^L(A_0)\}.
\endeqna
\end{Prop}

\bigskip
\begin{proof}
Let 
\eqna
A &=& \{xwy~|~w \in \WC, (x,w,y) \in \LC_L(W) \text{ and } L(w)=\nu_L\}, \\
B &=& \{xw_\l y~|~\l \in \spe_L(W)\text{ and } (x,w_\l,y) \in \LC_L(W)\}, \\
C &=& \{w~|~w A_0 \not\subset \UC^L(A_0)\}.
\endeqna
It is clear that $\cb^L_{\min}$, $B \subset A$. Now, let $z \in A$. 
Then there exists $w \in \WC$ and $x$, $y \in W$ such that $z=xwy$, 
$L(z)=L(x)+L(w)+L(y)$ and $L(w)=\nu_L$. 

\medskip

Let us first prove that $z \in B$. 
There exists a $0$-dimensional facet $\l$ 
such that $w \in W_\l$. If $L(w_\l) < \nu_L$, then $L(w) \le L(w_\l) < \nu_L$, 
which is impossible. Therefore, $L(w_\l)=\nu_L$, so $\l \in \spe_L(W)$. 
Write $w=w_\l a$, with $a \in W_\l$. Then, since $w_\l$ is the longest element of $W_\l$, 
we get that $\ell(w_\l)=\ell(w)+\ell(a)$, so $L(w)+L(a)=L(w_\l)$, so $L(a)=0$. 
By Proposition \ref{addition L}, it follows that $L(ay)=L(y)$. 
Then $z = x w_\l a y$, and $L(z)=L(x)+L(w)+L(y)=L(x)+L(w_\l)+L(ay)$. 
This shows that $z \in B$. So $A=B$.

\medskip

Let us now prove that $z \in \cb^L_{\min}$. We shall argue by induction on $\ell(x)+\ell(y)$. 
The result being obvious if $\ell(x)+\ell(y)=0$, we assume that $\ell(x)+\ell(y) > 0$. 
By symmetry, we may assume that $x=sx'$, with $s \in S$ and $sx' > x'$. Let $z'=sz$. 
Therefore, $L(z)=L(sx')+L(w)+L(y)=L(s)+L(x')+L(w)+L(y) \ge L(s) + L(x'wy)=L(s) + L(z') \ge L(z')$. 
Consequently, $L(z)=L(s)+L(z')$ and $L(z')=L(x')+L(w)+L(y)$. So $z' \in A$ and, 
by the induction hypothesis, $z' \in \cb^L_{\min}$. 
Two cases may occur:

\medskip

$\bullet$ If $sz' > z'$, then $z=sz' \in \cb^L_{\min}$ by Lemma \ref{cmin addition}, as desired.

\medskip

$\bullet$ If $sz' < z'$, then $L(z) = L(sz') \le L(z')$. Since we have already proved that 
$L(z) \ge L(z')$, this forces $L(s)=0$. Write $z'=a w' b$ with $w' \in \WC$, $L(w')=\nu_L$ and 
$a \add w' \add b$. Since $z=sz' < z'$, this means that $z$ is obtained from the expression 
$aw'b$ by removing a simple reflection $s'$ conjugate to $s$ from a reduced expression of 
$a$, $b$ or $w'$. It it is removed from $a$ or $b$, then $z=a'w'b'$ with $a' \add w' \add b'$, 
so $z \in \cb^L_{\min}$. If it is removed from $w'$, then $z=a w'' b'$ with $L(w'')=L(w')=\nu_L$ 
and $a \add w'' \add b$, so $z \in \cb^L_{\min}$. 

\medskip

Therefore, we have proved that $A=B=\cb^L_{\min}$.

\medskip

It remains to show that $C=\cb^L_{\min}$. Let $z\in\cb^{L}_{\min}=B$. Then there exist $x,y\in W$ and $\la\in\spe_{L}(W)$ such that $z=xw_{\la}y$ and $(x,w_{\la},y)\in\cL_{L}(W)$. In particular we have $(w_{\la},y)\in\cL_{L}(W)$ hence, using Proposition \ref{geom-L}, we get
\begin{equation*}
H^{L}(A_{0},yA_{0})\cap H^{L}(yA_{0},w_{\la}yA_{0})=\emptyset\tag{$\ast$}.
\end{equation*}
Let $\al\in\Phi^{L}$. Since $\la$ is a special point there exists an hyperplane $H_{\al,n_{\la}}$ of weight $L_{\al}$ which contains $\la$. The hyperplane $H_{\al,n_{\la}}$ separates $yA_{0}$ and $w_{\la}A_{0}$ and is of maximal weight, therefore by $(\ast)$ it cannot lie in $H^{L}(A_{0},yA_{0})$ and it follows that it separates  $A_{0}$ and $w_{\la}yA_{0}$. Therefore for all $\mu\in w_{\la}yA_{0}$ we have
$$\begin{array}{ccc}
\sg \mu, \check{\al}\sd>n_{\la}& \mbox{ if $n_{\la}\geq 1$}\\
\sg \mu, \check{\al}\sd<n_{\la}& \mbox{ if $n_{\la}\leq 0$}\\
\end{array}$$
and $w_{\la} y\notin \UC_{\al}^L(A_0)$. But this holds for all $\al\in\Phi^{L}$, thus $w_{\la} yA_{0}\notin \UC^L(A_0)$.  Now we have $(x,w_{\la},y)\in \cL_{L}(W)$ therefore 
$$H^{L}(A_{0},w_{\la}yA_{0})\cap H^{L}(w_{\la}yA_{0},xw_{\la}yA_{0})=\emptyset$$
from where we see that $H_{\al,n_{\la}}$ does not lie $H^{L}(w_{\la}yA_{0},xw_{\la}yA_{0})$. Hence  $x_{\la}w_{\la} yA_{0}\notin \UC_{\al}^L(A_0)$ as required.

\medskip

Let us now prove that $C\subset \cb^L_{\min}$. Let $w\in C$. The idea is to follow the proof of \cite[Proposition 5.5]{Bremke} using $\Phi^{L}$ instead of $\Phi$. The alcove $wA_{0}$ lies in some connected component of
$$V-\bigcup_{\al\in\Phi^{L}} H_{\al,0}.$$
The group  $\Om^{L}_{0}=\sg \sigma_{H_{\al,0}}\mid \al\in\Phi^{L}\sd$ is easily seen to act simply transitively on this set of connected components, therefore there exists $\sigma\in \Omega^{L}_{0}$ such that
$$wA_{0}\subset \cC_{\si}:=\{v\in V\mid \sg v,\check{\al}\sd>0 \text{ for $\al\in \sigma(\De^{L})$}\}$$
This implies that there exist $r$ linearly independent roots $\beta_{1},\ldots,\beta_{r}$ in $\Phi^{L}\cap \Phi^{+}$ such that 
$$\cC_{\si}:=\{v\in V\mid \sg v,\check{\beta_{i}}\sd<0 \text{ for $1\leq i\leq k$, }\sg v,\check{\beta_{i}}\sd>0 \text{ for $k+1\leq i\leq r$}\sd\}.$$
Removing the alcoves which lies in $ \UC^L(A_0)$ we obtain the following $L$-quarter, which is a translate of $\cC_{\si}$:
$$\cC_{\si}':=\{v\in V\mid \sg v,\check{\beta_{i}}\sd<b_{i} \text{ for $1\leq i\leq k$, }\sg v,\check{\beta_{i}}\sd>b_{i} \text{ for $k+1\leq i\leq r$}\sd\}$$
where
$$b_{i}=
\begin{cases}
0& \mbox{ if $1\leq i\leq k$,}\\
1& \mbox{if $k+1\leq i\leq r$ and $L(H_{\beta_{i},0})=L(H_{\beta_{i},1})$,}\\
2& \mbox{otherwise.}
\end{cases}
$$
Let $\la_{\si}$ be the vertex of $\cC_{\si}'$ that is the point of $V$ which satisfies $\sg\la_{\si},\check{\beta_{i}}\sd=b_{i}$ for all $1\leq i\leq r$. Then $\la_{\si}$ is a special point: for $\al\in\Phi^{L}$ we set $n^{\la}_{\al}=\sg \la_{\si}, \check{\al}\sd$. Note that for all $\al\in\Phi^{L}$ we have
\begin{equation*}
\tag{$\dag$}\cC_{\si}\subset V^{+}_{H_{\al,n^{\la}_{\al}}}\text{ if $n^{\la}_{\al}>0$}\quad\text{ and }\cC_{\si}\subset V^{-}_{H_{\al,n^{\la}_{\al}}}\text{ if $n^{\la}_{\al}\leq  0$}.
\end{equation*}
Now let $z\in W$ be such that $z(A_{0})\subset \cC_{\si}'$, $\la\in\ov{z_{\si}A_{0}}$.
We get for $\al\in\Phi^{L}$ (using $(\dag)$)
\begin{itemize}
\item if $H_{\al,n}\in H^{L}(wA_{0},zA_{0})$ then $|n|> |n_{\al}^{\la}|$; 
\item if $H_{\al,n}\in H^{L}(A_{0},zA_{0})$ then $|n|\leq |n_{\al}^{\la}|$;
\item if $H_{\al,n}\in H^{L}(w_{\la}zA_{0},zA_{0})$ then $n=n_{\al}^{\la}$.
\end{itemize}
Finally putting all this together we get that
$$\begin{array}{ccccc}
&H^{L}(A_{0},w_{\la}zA_{0})\cap H^{L}(w_{\la}zA_{0},zA_{0})&=& \emptyset\\
\text{and} & H^{L}(A_{0},zA_{0})\cap H^{L}(zA_{0},wz^{-1}(zA_{0}))&=& \emptyset.\\
\end{array}$$
Hence $w=wz^{-1}w_{\la}w_{\la}z$ and $(wz^{-1},w_{\la},w_{\la}z)\in\cL_{L}(W)$. 
\end{proof}
\begin{Rem}
By direct product, one can easily show that Proposition \ref{c-min} still holds when $W$ is not irreducible.\finl
\end{Rem}


\subsection{The elements of $\cb^{L}_{\min}$}

Keeping the notation of the proof of Proposition \ref{c-min}, every element $\si\in\Om^{L}_{0}$ defines an $L$-quarter $\cC_{\si}$ with vertex $0$ and an $L$-quarter $\cC'_{\si}$ (which is a translate of $\cC_{\si}$) with vertex $\la_{\si}$. We get the following equality
$$\cb^{L}_{\min}(W)=\bigcup_{\si\in\Om^{L}_{0}} \{w\in W\mid wA_{0}\subset \cC'_{\si}\}=\bigcup_{\si\in\Om^{L}_{0}} N^{L}_{\si}(W).$$
We will simply write $N^{L}_{\si}$ if it is clear from the context what the  group $W$ is. Note that any two sets $\cC_{\si},\cC_{\si'}$ are separated by at least one maximal strip, hence the above union is disjoint. In fact, the sets $\ov{\cC_{\si}}$ are the connected components of the closure of $\{\mu\in V\mid \mu \in wA_{0}, w\in \cb^{L}_{\min}\}$.

\bigskip

Let $b_{\si}$ be the unique element such that $\la_{\si}\in\ov{b_{\si}A_{0}}$ and $b_{\si}$ has minimal length in the coset $W_{\la_{\si}}b_{\si}$. For a $0$-dimensional facet $\la$ of an alcove, we set $S^{\circ}_{\la}:=\{s\in S_{\la}\mid L(s)=0\}$  and we denote by $\hw_{\la}$ the element of minimal length in $w_{\la}W_{S^{\circ}_{\la}}$

\begin{Lem}
Every element $w\in\cb^{L}_{\min}$ can be uniquely written under the form $x_{w}a_{w}\hw_{\la_{\si}}b_{\si}$ where $\si\in\Om^{L}_{0}$, $a_{w}\in W_{S_{\la_\si}^{\circ}}$ and $x_{w}\in X_{\la_{\si}}$. 
\end{Lem}
\begin{proof}
Let $w\in\cb^{L}_{\min}$. We know by the previous proof that $w=wz^{-1}w_{\la}w_{\la}z$ and $(wz^{-1},w_{\la},w_{\la}z)\in\cL_{L}(W)$ where $z$ is such that $z(A_{0})\subset \cC_{\si}'$ and  $\la_{\si}\in\ov{zA_{0}}$. Both $b_{\si}A_{0}$ and $w_{\la}zA_{0}$ contains $\la_{\si}$ in their closure hence they lie in the same right coset with respect to $W_{\la_{\si}}$. Since $b_{\si}$ has minimal length in $W_{\la_{\si}}b_{\si}$ there exists $y'\in W_{\la_{\si}}$ such that $w_{\la}z=y'b_{\si}$ and $(y',b_{\si})\in\cL(W)$. Next let $x_{w}$ be the element of minimal length in $(wz^{-1})W_{\la_{\si}}$ and let $x'$ be such that $wz^{-1}=x_{w}x'$. Assume for a moment that $x',y'\in W_{S_{\la_{\si}}^{\circ}}$. Then 
$$x'w_{\la_{\si}}y'\in W_{S^{\circ}_{\la_{\si}}}w_{\la_{\si}}W_{S^{\circ}_{\la_{\si}}}=W_{S^{\circ}_{\la_{\si}}}w_{\la_{\si}}$$ 
and we write $x'w_{\la_{\si}}y'=a_{w}\hw_{\la_{\si}}$. Finally $w$ can be written uniquely as $x_{w}a_{w}w^{\circ}_{\la}b_{\si}$. 

\bigskip

Let us now prove that $x',y'\in W_{S_{\la_{\si}}^{\circ}}$ that is $L(x')=L(y')=0$. Recall that $b_{\si}$ has minimal length in $W_{\la}b_{\si}$. On the one hand we have
$$L(w_{\la}w_{\la}z)=L(w_{\la})+L(w_{\la}z)=L(w_{\la})+L(y'b_{\si})=L(w_{\la})+L(y')+L(b_{\si}).$$
On the other hand
$$L(w_{\la}w_{\la}z)=L(w_{\la}y'b_{\si})=L(w_{\la}y')+L(b_{\si})=L(w_{\la})-L(y')+L(b_{\si})$$
hence $L(y')=0$. Similarly one can show that $L(x')=0$. The proof is complete
\end{proof}
%

\bigskip

Later on, we will need the following result. We put it and prove it here because it uses the notation introduced in this section. 
\begin{Lem}
\label{dot-property}
Let $w=x_{w}a_{w}\hw_{\la_{\si}}b_{\si}$ where $\si\in\Om^{L}_{0}$, $a_{w}\in W_{S_{\la_\si}^{\circ}}$ and $x_{w}\in X_{\la_{\si}}$. 
Then we have
\begin{enumerate}
\item $x_{w}\add a_{w}\add w^{\circ}_{\la_{\si}}\add b_{\si}$;
\item if $w'<a_{w}\hw_{\la}b_{\si}$ and $w'\in\bc^{L}_{\min}$ then either $w'=a_{w'}\hw_{\la_{\si}}b_{\si}$ where $a_{w'}<a_{w}$ or  $w'=x_{w'}a_{w'}\hw_{\la_{\si'}}y_{\si'}$ where $y_{\si'}<b_{\si}$.
\end{enumerate}
\end{Lem}
Eventhough this result might look fairly natural, it is in fact quite long to prove and involved a case by case analysis.
\begin{proof}
We prove 1. The fact that $a_{w}\add w^{\circ}_{\la_{\si}}\add b_{\si}$ is clear by definition therefore we only need to show that $x_{w}\add (a_{w}w^{\circ}_{\la_{\si}} b_{\si})$. To this end we show that
$$D:=H(A_{0},a_{w}w^{\circ}_{\la_{\si}}b_{\si}A_{0})\cap H(a_{w}w^{\circ}_{\la_{\si}}b_{\si}A_{0},x_{w}a_{w}w^{\circ}_{\la_{\si}}b_{\si}A_{0})=\emptyset.$$
{\bf Claim 1.} {\it If $\sg\la_{\si},\check{\beta}\sd\in\nZ$ then $H_{\beta,n}\notin D$ for all $n\in\nZ$. }

\begin{proof}
Let $n_{\b}=\sg\la_{\si},\check{\beta}\sd$ and assume that $n_{\b}> 0$, the case $n_{\b}\leq 0$ is similar. Since $x_{w}\in X_{\la_{\si}}$, there are no hyperplane containing $\la_{\si}$ which lies in 
$$H(a_{w}w^{\circ}_{\la_{\si}}b_{\si}A_{0},x_{w}a_{w}w^{\circ}_{\la_{\si}}b_{\si}A_{0}).$$
Hence $H_{\b,n_{\b}}\notin D$. Now let $n\neq n_{\b}$. Then since $\la_{\si}\in \ov{a_{w}w^{\circ}_{\la_{\si}}b_{\si}A_{0}}$, we have
$$n<\sg \mu,\check{\b} \sd<n_{\b}+1\text{ for all $\mu\in a_{w}w^{\circ}_{\la_{\si}}b_{\si}A_{0}$}.$$
If  $n>n_{\b}$ we have $H_{\b,n}\notin H(A_{0},a_{w}w^{\circ}_{\la_{\si}}b_{\si}A_{0})$ and $H_{\b,n}\notin D$. If $n<n_{\b}$ then $H_{\b,n}\notin H(a_{w}w^{\circ}_{\la_{\si}}b_{\si}A_{0},x_{w}a_{w}w^{\circ}_{\la_{\si}}b_{\si}A_{0})$ and $H_{\b,n}\notin D$.  
\end{proof}
\noindent
{\bf Claim 2.} {\it Let $\beta\in\Phi^{+}$ be such that there exists a hyperplane of direction $\beta$ in $D$. Then $H_{\beta,0}\cap \cC_{\si}\neq\emptyset$ }
\begin{proof}
 Let $\beta\in\Phi^{+}$ be such that there exists a hyperplane of direction $\beta$ in $D$. By the previous claim, we know that $\sg \la_{\si},\check{\b}\sd\notin \nZ$. Let $n\in\nZ$ be such that $n<\sg \la_{\si},\check{\b}\sd<n+1$. We will assume that $n>0$. (The case $n\leq 0$ is similar.) Note that we must have $H_{\beta,n}\cap \cC'_{\si}\neq 0$. Translating by $-\la_{\si}$ we get 
$$t_{-\la_{\si}}(H_{\beta,n})\cap \cC_{\si}\neq\emptyset.$$
Let $x\in t_{-\la_{\si}}(H_{\beta,n})\cap \cC_{\si}$. Note that we have $\sg x,\check{\beta}\sd=n-\sg \la_{\si},\check{\b} \sd$. Let 
$$\delta=\frac{\sg \la_{y},\check{\b} \sd -n}{\sg \la_{y},\check{\b}\sd}<1.$$
Then $0 < \d < 1$ and an easy calculation to show that 
$x+\d \la_{\si}\in H_{\beta,0}\cap \cC_{\si}$.
\end{proof}

\noindent
{\bf Claim 3.} {\it  Let $\beta\in\Phi^{+}$ be such that $H_{\beta,0}\cap \cC_{\si}\neq \emptyset$. Then we have either
\begin{enumerate}
\item $\sg \la_{\si},\check{\beta} \sd \in\nZ$,
\item $0< \sg \la_{\si},\check{\beta} \sd <1 $.
\end{enumerate}
}

\medskip
\noindent
We now prove that Claims (1)--(3) implies that $D=\emptyset$.  By Claim 2, the only hyperplanes that can lie in $D$ are those of the form $H_{\beta,n}$ where $H_{\beta,0}\cap \cC_{\si}\neq \emptyset$. But then Claim 3 implies that we have either 
\begin{enumerate}
\item $\sg \la_{\si},\beta \sd \in\nZ$,
\item $0< \sg \la_{\si},\beta \sd <1 $.
\end{enumerate}
If we are in case (1), we have $H_{\beta,n}\notin D$ by Claim 1. If we are in case (2), then the alcove $a_{w}w^{\circ}_{\la_{\si}}b_{\si}A_{0}$,  which contains $\la_{\si}$ in its closure, must satisfies
$$a_{w}w^{\circ}_{\la_{\si}}b_{\si}A_{0}\subset\{x\in V\mid 0<\sg x,\check{\beta}\sd<1 \}.$$
But so does $A_{0}$, therefore there are no hyperplane of direction $\beta$ which lies on $H(A_{0},a_{w}w^{\circ}_{\la_{\si}}b_{\si}A_{0})$ and $H_{\beta,n}\notin D$. Thus $D=\emptyset$ as required.\\

\medskip
\noindent

%


\bigskip
\noindent

It remains to prove Claim 3. We will proceed by a case by case analysis but first we want to express Claim 3 in a form which is easier to check. To do so, we need to introduce some more notation. \\

Any $\si\in\Omega^{L}_{0}$ defines a partition $\De^{+}_{\si}\cup \De_{\si}^{-}$ of $\De^{L}$ where 
$$\De_{\si}^{+}=\{\al\in\De^{L}\mid \al\si\in \Phi^{+}\} \text{ and } \De_{\si}^{-}=\{\al\in\De^{L}\mid \al\si\in \Phi^{-}\}.$$
\begin{Rem}Note that we can obtain all partition of $\De^{L}$ in this way, but that two distinct $\si$ might give rise to the same partition.\finl 
\end{Rem}
To such a partition, we associate $\la_{\De^{+}_{\si},\De^{-}_{\si}}\in V$ defined by
$$\sg \la_{\De^{+}_{\si},\De^{-}_{\si}},\check{\al}\sd=0  \text{ if $\al\in\De^{-}_{\si}$ and }  \sg \la_{\De^{+}_{\si},\De^{-}_{\si}},\check{\al}\sd=b_{\al}  \text{ if $\al\in\De^{+}_{\si}$}$$
where
$$b_{\al}=
\begin{cases}
1& \mbox{if $k+1\leq i\leq r$ and $L(H_{\al,0})=L(H_{\al,1})$};\\
2& \mbox{otherwise}.
\end{cases}
$$
Then we have $\la_{\si}=(\la_{\De^{+}_{\si},\De^{-}_{\si}})\si$. Claim 3 is then easily seen to be equivalent to the following statement, by applying $\si^{-1}$. (In the expression $\cC_{1}$, the 1 denotes the identity of $\Om^{L}_{0}$.)

\bigskip
\noindent
{\bf Claim 3'.} {\it  Let $\gamma\in\Phi^{+}$ such that $H_{\gamma,0}\cap \cC_{1}$ and let $\si\in\Om_{0}^{L}$. Then we have either
\begin{enumerate}
\item $\sg \la_{\tDe^{+}_{\si},\tDe^{-}_{\si}}, \check{\gamma} \sd \in\nZ$
\item $0< \sg \la_{\tDe^{+}_{\si},\tDe^{-}_{\si}},\check{\gamma} \sd 
<1 \text{ if $\gamma\si\in \Phi^{+}$}$
\item $-1< \sg \la_{\tDe^{+}_{\si},\tDe^{-}_{\si}},\check{\gamma} \sd <0 \text{ if $\gamma\si\in \Phi^{-}$}$
\end{enumerate}
}

\begin{proof}
As mentioned earlier, we proceed by a case by case analysis. Note that it is enough to prove the claim for $\gamma\notin\Phi^{L}$, since for all $\gamma\in\Phi^{L}$ we have $\sg \la_{\tDe^{+}_{\si},\tDe^{-}_{\si}},\check{\gamma} \sd \in\nZ$ as $\la_{\tDe^{+}_{\si},\tDe^{-}_{\si}}$ is a $L$-special point. 

\bigskip
 \noindent
{\bf Type $\tilde{\Grm}_{2}$.} It is a straightforward verification. \\

\noindent
{\bf Type $\tilde{\Frm}_{4}$}. Let $V=\nR^{4}$ with orthonormal basis 
$(\e_{i})_{1\leq i\leq 4}$. The root sytem $\Phi$ of type $\Frm_{4}$ consists of $24$ long roots and $24$ short roots:
$$\pm\e_{i}\pm\e_{j}\text{ and } \pm\e_{i}, \frac{1}{2}(\pm\e_{1}\pm\e_{2}\pm\e_{3}\pm\e_{4}).$$
Assume that $S^{\circ}=\{s_{1},s_{2}\}$. We get that $\Phi^{L}$ is of type $\Drm_{4}$ and consists of the roots $\pm\e_{i}\pm\e_{j}$. We choose the following simple system:
$$\De^{L}=\{\e_{1}-\e_{2},\e_{2}-\e_{3},\e_{3}-\e_{4}, \e_{3}+\e_{4}\}.$$
We have $\cC_{1}=\{x\in V| \sg x,\check{\al}\sd>0, \text{for all $\al\in\De^{L}$}\}$. In other words
$$\begin{cases}
x_{1}-x_{2}&>0\\
x_{2}-x_{3}&>0\\
x_{3}-x_{4}&>0\\
x_{3}+x_{4}&>0\\
\end{cases}
$$
for all $x_{1},x_{2},x_{3},x_{4}\in \cC_{1}$. In particular, we have $x_{1}>x_{2}>x_{3}>|x_{4}|$. The first step is to determine the set $\fB$ of roots $\gamma\in\Phi\backslash \Phi^{L}$ which satisfies 
$$H_{\gamma,0}\cap \cC_{1}\neq\emptyset.$$
We find 
$$\fB=\{\pm\frac{1}{2}(\e_{1}-\e_{2}-\e_{3}+\e_{4}), \pm\frac{1}{2}(\e_{1}-\e_{2}-\e_{3}-\e_{4}) \}.$$
The set of points $\{\la_{\De^{+}_{\si},\De^{-}_{\si}}\mid \si\in\Om^{L}_{0}\}$ is the set of points $(x_{1},x_{2},x_{3},x_{4})\in V$ which are solutions to the systems
$$\begin{cases}
x_{1}-x_{2}&=\delta_{1}\\
x_{2}-x_{3}&=\delta_{2}\\
x_{3}-x_{4}&=\delta_{3}\\
x_{3}+x_{4}&=\delta_{4}\\
\end{cases}
$$
where $\delta_{i}=0$ or $1$. Claim 3' then follows by a straightforward computations. We find that $\sg \la_{\De^{+}_{\si},\De^{-}_{\si}},\check{\gamma}\sd\in\nZ$ in all cases.\\

\noindent
Assume that $S^{\circ}=\{t_{1},t_{2},t_{3}\}$. We get that $\Phi^L$ is of type $D_{4}$ and consists of the roots $\pm\e_{i}, \frac{1}{2}(\pm\e_{1}\pm\e_{2}\pm\e_{3}\pm\e_{4})$. We choose the following simple system:
$$\De^{L}=\{\frac{1}{2}(\e_{1}-\e_{2}-\e_{3}-\e_{4}),\e_{2},\e_{3}, \e_{4}\}.$$
We have $\cC_{1}=\{x\in V| \sg x,\check{\al}\sd>0, \text{for all $\al\in\De^{L}$}\}$. 
The set $\fB$ of roots $\gamma\in\Phi\backslash \Phi^{L}$ which satisfies  $H_{\gamma,0}\cap \cC_{0}^{L}\neq\emptyset$ is 
$$\fB=\{\pm(\e_{2}-\e_{3}), \pm(\e_{3}-\e_{4}), \pm(\e_{2}-\e_{4})\}.$$
The set of points $\la_{\De^{+}_{\si},\De^{-}_{\si}}$ is the set of points $(x_{1},x_{2},x_{3},x_{4})\in V$ which are solutions to the systems
$$\begin{cases}
x_{1}-x_{2}-x_{3}-x_{4}&=\delta_{1}\\
x_{2}&=\delta_{2}\\
x_{3}&=\delta_{3}\\
x_{4}&=\delta_{4}\\
\end{cases}
$$
where $\delta_{i}=0$ or $\frac{1}{2}$ (since for all $\al\in\De^{L}$ we have $\check{\al}=2\al$). Let $\gamma=\e_{i}-\e_{j}$ where $j>i>1$ and let $\si\in\Om^{0}_{L}$. Then $\si$ defines a partition of $\De^{L}$, which in turns defines the $\delta$'s. There are 3 cases to consider:
\begin{itemize}
\item Suppose that $\delta_{i}=\delta_{j}$. Then $\sg\la_{\De^{+}_{\si},\De^{-}_{\si}},\check{\gamma}\sd=x_{i}-x_{j}=0\in\nZ$ as required;
\item Suppose that $\frac{1}{2}=\delta_{i}>\delta_{j}=0$. Then $\si$ sends $\e_{i}$ to a positive root and $\e_{j}$ to a negative one. Hence it sends $\gamma$ to a positive root. We get 
$$\sg\la_{\De^{+}_{\si},\De^{-}_{\si}},\check{\gamma}\sd=x_{i}-x_{j}=\frac{1}{2}$$ as required.
\item Suppose that $\frac{1}{2}=\delta_{j}>\delta_{i}=0$. Then $\si$ sends $\e_{i}$ to a negative root and $\e_{j}$ to a positive one. Hence it sends $\gamma$ to a negative root. We get $$\sg\la_{\De^{+}_{\si},\De^{-}_{\si}},\check{\gamma}\sd=x_{i}-x_{j}=-\frac{1}{2}$$ as required.
\end{itemize}

\vspace{.3cm}
\noindent
{\bf Type $\tilde{\Brm}_{n}$}. Let $V=\nR^{n}$ with basis $(\e_{i})_{1 \le i \le n}$. The root sytem $\Phi$ of type $\Brm_{n}$ consists of $2n$ short roots $\pm\e_{i}$ and $2n(n-1)$ long roots roots $\pm\e_{i}\pm\e_{j}$. \\

\noindent
Assume that $S^{\circ}=\{t\}$. We get that $\Phi^L$ of type $\Drm_{n}$ and consists of the roots $\pm\e_{i}\pm\e_{j}$. We choose the following simple system:
$$\De^{L}=\{\e_{1}-\e_{2},\ldots,\e_{n-1}-\e_{n}, \e_{n-1}+\e_{n}\}.$$
The set $\fB$ of roots $\gamma\in\Phi\backslash \Phi^{L}$ which satisfies  $H_{\gamma,0}\cap \cC_{1}\neq\emptyset$ is 
$$\fB=\{\pm\e_{n} \}.$$
The set of points $\la_{\De^{+}_{\si},\De^{-}_{\si}}$ is the set of points $(x_{1},x_{2},x_{3},x_{4})\in V$ which are solutions to the systems
$$\begin{cases}
x_{1}-x_{2}&=\delta_{1}\\
x_{2}-x_{3}&=\delta_{2}\\
\qquad\vdots&\\
x_{n-1}-x_{n}&=\delta_{n-1}\\
x_{n-1}+x_{n}&=\delta_{n}\\
\end{cases}
$$
where $\delta_{i}=0$ or $1$. This implies that $x_{n}=-1/2,0$ or $1/2$. Therefore we get
$$\sg\la_{\De^{+}_{\si},\De^{-}_{\si}},\check{\e}_{n} \sd =2(\la_{\De^{+}_{\si},\De^{-}_{\si}},\e_{n})=-1,0\text{ or } 1$$
as required.\\

\noindent
Assume that $I_{0}=\{s_{1},\ldots,s_{n}\}$. We get that $\Phi^L$ of type $(\Arm_{1})^{n}$ and consists of the roots $\pm\e_{i}$. We choose the following simple system:
$$\De^{L}=\{\e_{i}\}.$$
The set $\fB$ of roots $\gamma\in\Phi\backslash \Phi^{L}$ which satisfies  $H_{\gamma,0}\cap \cC_{1}\neq\emptyset$ is 
$$\fB=\{\e_{i}-\e_{j}| i<j \}$$
The set of points $\la_{\De^{+}_{\si},\De^{-}_{\si}}$ is the set of points $(x_{1},x_{2},x_{3},x_{4})\in V$ which are solutions to the systems
$$\begin{cases}
x_{1}&=\delta_{1}\\
x_{2}&=\delta_{2}\\
\qquad\vdots&\\
x_{n}&=\delta_{n}\\
\end{cases}
$$
where $\delta_{i}=0$ or $\frac{1}{2}$. Let $\gamma=\e_{i}-\e_{j}$ where $j>i$ and let $\si\in\Om^{0}_{L}$. Then $\si$ defines a partition of $\De^{L}$, which in turns defines the $\delta$'s. There are 3 cases to consider:
\begin{itemize}
\item Suppose that $\delta_{i}=\delta_{j}$. Then $\sg\la_{\De^{+}_{\si},\De^{-}_{\si}},\check{\gamma}\sd=x_{i}-x_{j}=0\in\nZ$ as required;
\item Suppose that $\frac{1}{2}=\delta_{i}>\delta_{j}=0$. Then $\si$ sends $\e_{i}$ to a positive root and $\e_{j}$ to a negative one. Hence it sends $\gamma$ to a positive root. We get 
$$\sg\la_{\De^{+}_{\si},\De^{-}_{\si}},\check{\gamma}\sd=x_{i}-x_{j}=\frac{1}{2}$$ as required.
\item Suppose that $\frac{1}{2}=\delta_{j}>\delta_{i}=0$. Then $\si$ sends $\e_{i}$ to a negative root and $\e_{j}$ to a positive one. Hence it sends $\gamma$ to a negative root. We get $$\sg\la_{\De^{+}_{\si},\De^{-}_{\si}},\check{\gamma}\sd=x_{i}-x_{j}=-\frac{1}{2}$$ as required.
\end{itemize}

\vspace{.3cm}
\noindent
{\bf Type $\tilde{\Crm}_{n}$}. Let $V=\nR^{n}$ with orthonormal basis $(\e_i)_{1 \le i \le n}$. The root sytem $\Phi$ of type $\Crm_{n}$ consists of $2n$ long roots $\pm 2\e_{i}$ and $2n(n-1)$ short roots roots $\pm\e_{i}\pm\e_{j}$. \\

\noindent
Assume that $I_{0}=\{t'\}$. Then $\Phi^L=\Phi$ and the statement is trivial since $\Phi\backslash \Phi^{L}=\emptyset$. \\

\noindent
Assume that $I_{0}=\{s_{1},\ldots,s_{n-1}\}$. We get that $\Phi^L$ is of type $(\Arm_{1})^{n}$ and consists of the roots $\pm2\e_{i}$. We choose the following simple system:
$$\De^{L}=\{2\e_{i}\}.$$
The set $\fB$ of roots $\gamma\in\Phi\backslash \Phi^{L}$ which satisfies  $H_{\gamma,0}\cap \cC_{1}\neq\emptyset$ is 
$$\fB=\{\e_{i}-\e_{j}| i<j \}$$
The set of points $\la_{\De^{+}_{\si},\De^{-}_{\si}}$ is the set of points $(x_{1},\ldots,x_{n})\in V$ which are solutions to the systems
$$\begin{cases}
x_{1}&=\delta_{1}\\
x_{2}&=\delta_{2}\\
\qquad\vdots&\\
x_{n}&=\delta_{n}\\
\end{cases}
$$
where $\delta_{i}=0$ or $1$ if $L(t)=L(t')$ or where $\delta_{i}=0$ or $2$ if $L(t)>L(t')$. We find that $\sg\la_{\De^{+}_{\si},\De^{-}_{\si}},\check{\gamma}\sd\in\nZ$ in all cases.\\

\noindent
Assume that $I_{0}=\{s_{1},\ldots,s_{n-1},t'\}$. It is the same thing as the previous case, except that the $\delta$'s only take values $0$ or $2$. \\

\noindent
Assume that $I_{0}=\{t,t'\}$. We get that $\Phi^L$ is of type $D_{n}$ and consists of the roots $\pm\e_{i}\pm \e_{j}$. We choose the following simple system:
$$\De^{L}=\{\e_{1}-\e_{2},\ldots,\e_{n-1}-\e_{n}, \e_{n-1}+\e_{n}\}.$$
The set $\fB$ of roots $\gamma\in\Phi\backslash \Phi^{L}$ which satisfies  $H_{\gamma,0}\cap \cC_{1}\neq\emptyset$ is 
$$\fB=\{\pm2\e_{n} \}.$$
The set of points $\la_{\De^{+}_{\si},\De^{-}_{\si}}$ is the set of points $(x_{1},\ldots,x_{n})\in V$ which are solutions to the systems
$$\begin{cases}
x_{1}-x_{2}&=\delta_{1}\\
x_{2}-x_{3}&=\delta_{2}\\
\qquad\vdots&\\
x_{n-1}-x_{n}&=\delta_{n-1}\\
x_{n-1}+x_{n}&=\delta_{n}\\
\end{cases}
$$
where $\delta_{i}=0$ or $1$. Let $\gamma=2\e_{n}$ and let  $\si\in\Om^{0}_{L}$. Then $\si$ defines a partition of $\De^{L}$, which in turns defines the $\delta$'s. There are 3 cases to consider:
\begin{itemize}
\item Suppose that $\delta_{n-1}=\delta_{n}$. Then $\sg\la_{\De^{+}_{\si},\De^{-}_{\si}},\check{\gamma}\sd=x_{n}=0\in\nZ$ as required;
\item Suppose that $\frac{1}{2}=\delta_{n-1}>\delta_{n}=0$. Then $\si$ sends $2\e_{n}=\gamma$ to a negative root and we have
$$\sg\la_{\De^{+}_{\si},\De^{-}_{\si}},\check{\gamma}\sd=x_{n}=-\frac{1}{2}$$ as required.
\item Suppose that $\frac{1}{2}=\delta_{n}>\delta_{n-1}=0$.  Then $\si$ sends $2\e_{n}=\gamma$ to a positive root and we have
$$\sg\la_{\De^{+}_{\si},\De^{-}_{\si}},\check{\gamma}\sd=x_{n}=\frac{1}{2}$$ as required.\end{itemize}

The proof of Claim 3' (hence of Statement (1)) is complete. 
\end{proof}

We now prove (2). Let $w'<a_{w}\hw_{\la_{\si}}b_{\si}$ be such that $w'\in\bc^{L}_{\min}$ and write $$w'=x_{w'}a_{w'}\hw_{\la_{\si'}}b_{\si'}.$$
Assume that $W_{\la_{\si'}}=W_{\la_{\si}}$. Then 
since
$$a_{w'}\hw_{\la_{\si'}}b_{\si'}<z_{w'}a_{w'}\hw_{\la_{\si'}}b_{\si'}<a_{w}\hw_{\la_{\si}}b_{\si}$$
and $a_{w'}\hw_{\la_{\si'}},a_{w}\hw_{\la_{\si'}}\in W_{\la_{\si}}$, we get (see \cite[Proof of Lemma 9. 10]{bible}) that either $b_{\si'}<b_{\si}$ or $b_{\si'}=b_{\si}$ and $a_{w'}\hw_{\la'}<a_{w}\hw_{\la'}$ which implies that $a_{w'}<a_{w}$ as required. 

\medskip
\noindent
Assume that $W_{\la_{\si'}}\neq W_{\la_{\si}}$. Write $w'=w_{\si}z'$ where $w_{\si}\in W_{\la_{\si}}$ and $z'$ has minimal length in the coset $W_{\la_{\si}}w'$. 

\medskip
\noindent
Note that since $w'=w_{\si}z'<a_{w}\hw_{\la_{\si}}b_{\si}$ we get that $z'\leq b_{\si}$. But we must have $z'< b_{\si}$ otherwise we would have  $w_{\si}<a_{w}\hw_{\la_{\si}}$, which together with the condition $w'\in\cb^{L}_{\min}$ would imply that $w_{\si}=u'\hw_{\la_{\si}}$ for some $u'\in W^{\circ}_{\la_{\si}}$ and $\la_{\si}= \la_{\si'}$.

\medskip
\noindent
If we show that 
$$D'=H(b_{\si'}A_{0},z'A_{0})\cap H(A_{0},b_{\si'}A_{0})=\emptyset$$
then the result will follow. Indeed if $D'=\emptyset$ then there exists an $x\in W$ such that $z'=xb_{\si'}$ and $x\add b_{\si'}$ and we get $b_{\si'}<b_{\si}$ since $b_{\si'}\leq z<b_{\si}$. 

\medskip
\noindent
Let $\la'$ be the unique $L$-special point which contains $z'A_{0}$ and $w^{\circ}_{\la_{\si}}z'A_{0}$. A hyperplane $H$ which lies in $D'$ cannot contain $\la_{\si'}$ (otherwise $H\notin H(A_{0},b_{\si'}A_{0}$)) nor $\la'$  (otherwise  $H\notin H(b_{\si'}A_{0},z'A_{0})$) but it has to separate these two points. Hence it also separates any alcoves which contains $\la_{\si'}$ and any alcoves which contains $\la'$. In particular it separates $a_{w'}w^{\circ}_{\la_{\si'}}b_{\si'}A_{0}$ and $w'A_{0}=x_{w'}a_{w'}\hw_{\la_{\si'}}b_{\si'}A_{0}$ but there are no such hyperplanes as we have shown in the proof of Statement~(1).  
\end{proof}


\begin{Exa}[{{\bf positive weight functions}}]\label{lowest positive}
In this example, and only in this example, we assume that $L$ is {\bfit positive}. Let $\PC_L(S)$ be the set of proper subsets $I$ of $S$ such that 
$L(w_I)=\nu_L$. If $I \in \PC_L(S)$, then $W_{I}\simeq W_{0}$ (because $L$ is positive). 
Note also that, since $L$ is positive, 
$$\PC_L(S)=\{S_\l~|~\l \in \spe_L(W)\}$$
$$\{w \in \WC~|~L(w)=\nu_L\}=\{w_I~|~I \in \PC_L(S)\}=\{w_\l~|~\l \in \spe_L(W)\}.
\leqno{\text{and}}$$
Recall also that $\LC_L(W)=\LC(W)$. 
Therefore 
\eqna
\cb^L_{\min}&=&\{xw_Iy \in W~|~x, y \in W,~x \add w_I \add y\text{ and }I \in \PC_L(S)\}\\
&=&\{xw_\l y~|~x, y \in W,~x \add w_{\la} \add y\text{ and } \la\in \spe_L(W)\}.
\endeqna
It we set $M_{\lambda}^L=\{z\in W~|~w_\l \add z \text{ and } 
sw_{\lambda}z\notin \cb_{\min}^L\text{ for all $s\in S_{\lambda}$}\}$ as in \cite[Proof of Theorem 5.4]{Bremke}, then we obtain the following decomposition of $\cb_{\min}^{L}$
$$\cb_{\min}^L=
\dot{\bigcup_{\substack{\lambda\in \spe_L^0(W)\\ z\in M_{\lambda}^L}}}
N_{\lambda,z}\quad \text{(disjoint union)}$$
 where $\spe_L^0(W)=\spe_L(W) \cap \overline{A}_0$
 is a set of representatives for the 
$\Omega$-orbits on $\spe_L(W)$ and 
$$N_{\la,z}=\{xw_{\la}z~|~ x\in X_{\la}\}.$$
It is easily seen that the set $M^{L}_{\la}$ consists of our elements $b_{\si}$ and that the sets $N_{\la,z}$ correspond to $N_{\si}^L$.\finl
\end{Exa}


\section{Semidirect product decomposition}
\label{semidirect}
\medskip

We fix a non-negative weight function $L : W \to \G$ where $\Ga$ is a totally abelian group. 
The aim of this section is to express the lowest two-sided cell $\cb^{L}_{\min}$ in relation to the decomposition of $W$ as  semidirect product 
of two Coxeter groups as in \cite{Bonnafe-Dyer}.

\subsection{Coxeter groups} 
If $I$ is a subset of $S$, we set 
$$I^\circ=\{s \in I~|~L(s)=0\}\quad\text{and}\quad I^+=\{s \in I~|~L(s) > 0\},$$
so that $I=I^\circ \dotcup I^+$. We also set
$$\Iti=\{wsw^{-1}~|~w \in W_{I^\circ}\text{ and }s \in I^+\}$$
and we denote by $\Wti_{\Iti}$ the subgroup of $W$ generated by $\Iti$. For 
simplification, we set $W^\circ=W_{S^\circ}$ and $\Wti=\Wti_{\Sti}$. 

Note that, if $s \in I^\circ$ and $t \in I^+$, then $L(s) \neq L(t)$, 
so $s$ and $t$ are not conjugate in $W$. It then follows from \cite{Bonnafe-Dyer} that 
\refstepcounter{Th}
\begin{equation}\label{eq:wtilde}
W_I = W_{I^\circ} \ltimes \Wti_\Iti\quad\text{\it and}\quad \text{\it $(\Wti_\Iti,\Iti)$ 
is a Coxeter group.}
\end{equation}
If $I=S$, we get that 
$$W = W^\circ \ltimes \Wti\quad\text{and}\quad\text{ $(\Wti,\Sti)$ is a Coxeter group.}$$

\medskip
%

We will assume that $W^{\circ}$ is finite. Note however, that this assumption is not very restrictive when dealing with an affine Weyl group. 
Indeed, by direct products, we can assume that $W$ is irreducible. 
In this case, either $L=0$ (and then $\cb^L_{\min}=W$ and the problem 
is uninteresting) or $S^\circ$ is a proper subset of $S$ (and then $W^\circ$ 
is finite because $W$ is irreducible).


\bigskip

\subsection{The group $\tOm$} We keep the notation of Section \ref{geometric}. Let 
$$S_{\Om}=\{\si_{H}\mid H \text{ is a wall of $A_{0}$}\}.$$
Then $(\Om,S_{\Om})$ is a Coxeter system. Let 
$$S^{\circ}_{\Om}=\{\si_{H}\mid H \text{ is a wall of $A_{0}$ and $L_{H}=0$} \}$$
and
$$S^{+}_{\Om}=\{\si_{H}\mid H \text{ is a wall of $A_{0}$ and $L_{H}>0$} \}.$$
Then we have $\Om=\Om^{\circ}\ltimes \tOm$ where $\Om^{\circ}$ is generated by $S_{\Om}^{\circ}$ and $\tOm$ is generated by 
$$\tS_{\Om}:=\{\rho \si_{H} \rho^{-1}\mid \rho\in \Om^{\circ} \text{ and } \si_{H}\in S^{+}_{\Om}\}=\{\si_{H\rho}\mid \rho\in \Om^{\circ} \text{ and } \si_{H}\in S^{+}_{\Om}\}.$$
We set
$$\FCt=\{H\in\cF\mid \si_{H}\in \tOm\}.$$
It is clear by definition that $\tOm$ is generated by  $\{\si_{H}\mid H\in \FCt\}$. Further, the following conditions are satisfied
\begin{enumerate}
\item[(D1)] $\tOm$ stabilizes $\FCt$.
\item[(D2)] The group $\tOm$, endowed with the discrete topology, acts properly on $V$. 
\end{enumerate}
We prove (D1). Let $\tsi\in\tOm$ and $H\in\FCt$, that is $\si_{H}\in\tOm$. We have $\tsi \si_{H}\tsi^{-1}=\si_{H\tsi}\in\tOm$ and, therefore, $H\tsi\in\FCt$. Condition (D2) follows easily form the fact that $\Om$, endowed with the discrete topology, acts properly on $V$. \\
We denote by $\alc(\FCt)$ the set of alcoves with respect to $\FCt$, that is the connected components of 
$$V-\bigcup_{H\in \FCt} H.$$
Let $\tA_{0}$ be the unique alcove (with respect to $\FCt$) which contains $A_{0}$. Then we have (see \cite[Chapter 5, \S 3]{Bourbaki} and \cite[\S 4]{Bonnafe-Dyer}) 
\begin{enumerate}
\item The group $\tOm$ is generated by the orthogonal reflections with respect to the wall of $\tA_{0}$.
\item $\ov{\tA_{0}}$ is a fundamental domain for the action of $\tOm$.
\item $\ov{\tA_{0}}=\underset{\rho\in\Om^{\circ}}{\bigcup} \ov{A_{0}}\rho$.
\item Any element $\si_{H}\in\tOm$ is conjugate in $\tOm$ to an orthogonal reflection with respect to a wall of $\tA_{0}$.
\end{enumerate}
It follows that $\tOm$ is an affine Weyl group (see \cite[\S 4]{Bonnafe-Dyer}). Note that $\tOm$ is not necessarily irreducible.  In fact, as we expect, the group $\tOm$ is nothing else that the group generated by the reflections with respect to the hyperlanes in $\cF^{L}=\{H\in\cF\mid L_{H}>0\}$.
\begin{Lem}
We have  $\FCt=\cF^{L}$.
\end{Lem}
\begin{proof}
Let $H\in \cF^{L}$. There exists $\si\in \Omega$ and a wall $H'$ of $A_{0}$ of positive weight such that $H'\si=H$. Write $\si=\rho\tilde{\si}$ where $\rho\in \Om^{\circ}$ and $\tilde{\si}\in \tOm$. Then $\si_{H'\rho}\in \tOm$ and we have
$$\si_{H}=\tilde{\si}\si_{H'\rho} \tilde{\si}^{-1}$$
therefore $\si_{H}\in \tOm$. 

Conversely, let $H\in\FCt$ that is $\si_{H}\in\tOm$. By (4), $\si_{H}$ is conjugate (in $\tOm$) to $\si_{H'}$ where $H'$ is a wall of $\tA_{0}$. By (3), we know that the walls of $\tA_{0}$ are of the form $H\rho$ where $H$ is a wall of $A_{0}$ of positive weight. In particular, $H'$ has positive weight. It follows that $H$ has positive weight and $H\in\cF^{L}$ as required. 
\end{proof}
Finally we want to define a root system associated to $\tOm$. Let
$$\tPhi:=\{b_{\al}\al\mid \al\in \Phi^{L}\}$$
where $b_{\al}$ is defined to be the smallest integer such that $H_{\al,b_{\al}}$ has positive weight. We also fix a set of positive roots \
$$\tPhi^{+}=\{b_{\al}\al\mid \al\in\Phi^{L}\cap \Phi^{+}\}.$$
\begin{Rem}
If $\Om$ is not of type $\tilde{\Crm}$ then we simply have $\tPhi=\Phi^{L}$. Indeed in this case, any two parallele hyperplane have same weights, hence $b_{\al}=1$ for all $\al\in\Phi^{L}$. If $\Om$ is of type $\tilde{\Crm}$, then we may have $b_{\al}=2$ for some choices of parameters, namely when $L(t)>L(t')=0$ (see Convention \ref{convention}).\finl
\end{Rem}
\begin{Rem}
We have $\Om^{L}_{0}=\tOm_{0}$ where $\tOm_{0}=\sg \si_{H_{\tilde{\al},0}}\mid \tilde{\al}\in\tPhi \sd$.\finl
\end{Rem}
\begin{Lem}
The group $\tOm$ is the affine Weyl group associated to $\tPhi$. Further the alcove $\tA_{0}$ is the fundamental alcove associated to $\tPhi$, that is
$$\tA_{0}=\{x\in V\mid 0<\sg x,\check{\al} \sd<1 \text{ for all $\al\in\tPhi^{+}$}\}$$
\end{Lem}
\begin{proof}
The first statement is clear since we have 
$$\cF^{L}=\{H_{b_{\al}\al,n}\mid \al\in\Phi^{L}\cap \Phi^{+}, n\in\nZ\}.$$  
The second statement follows easily from the above equality and the fact that $A_{0}\subset \tA_{0}$. 
\end{proof}
Doing as in Section \ref{geometric}, we obtain another geometric realization of $\tOm$, namely as a group generated by involutions on the set $\alc(\FCt)$. Indeed, $\tOm$ acts transitively on the set of faces of alcoves in $\alc(\FCt)$: we denote by $\{\tit_{1},\ldots,\tit_{m}\}$ the set of $\tOm$-orbits in the set of faces.  Note that the set of faces of $\tA_{0}$ is a set of representatives of the set of orbits. To each  $\tit_{i}$ we can associate an involution $\tA\mapsto \tit_{i}\tA$ of $\alc(\FCt)$ where $\tit_{i}\tA$ is the unique alcove of $\alc(\FCt)$ which shares with $\tA$ a face of type $\tit_{i}$.  The group generated by all the $\tilde{t}_{i}$ is an affine Weyl group isomorphic to $\tOm$. We would like to use the notation $\tW$ and $\tS$ for this group, and eventually we will, but before one needs to be careful since $\tW$ also denotes the group appearing in the semidirect product decomposition of $W$ (where $W$ is the group generated by involutions on $\alc(\cF)$). 

\subsection{Alcoves of $\tW$} 
Recall the definition of $(W,S)$ in Section~\ref{geometric} and that
$$S^{\circ}=\{s'\in S\mid L(s)=0\}\text{ and } S^{+}=\{s\in S\mid L(s)>0\}.$$
Then we have $W=W^{\circ}\ltimes \tW$ where $W^{\circ}$ is generated by $S^{\circ}$ and $\tW$ is generated by 
$$\tS=\{wtw^{-1}\mid t\in S^{+}\text{ and } w\in W^{\circ}\}.$$
\begin{Lem}
Let $\tilde{t}\in \tS=\{wtw^{-1}\mid t\in S^{+}\text{ and } w\in W^{\circ}\}$. Then there exists a unique wall $H$ of $\tA_{0}$ such that  
$$\tilde{t} A_{0}=A_{0}\si_{H}.$$
\end{Lem}
\begin{proof}
Let $w \in W^\circ$ and $t \in S^+$ be such that $\tilde{t}=wtw^{-1}$. 
Let $\rho\in\Om^{\circ}$ be such that $wA_{0}=A_{0}\rho$ and let $H'$ be the unique hyperplane which contains the face of type $t$ of $A_{0}$. Then we have
$$wtw^{-1}A_{0}=wtA_{0}\rho^{-1}=wA_{0}\si_{H'}\rho^{-1}=A_{0}\rho \si_{H'}\rho^{-1}=A_{0}\si_{H'\rho}$$
and the result follows. 
\end{proof}

Therefore there is a natural bijection between the set $\tS$ and the set of faces of $\tA_{0}$ and therefore between $\tS$ and the set of orbits $\{\tit_{1},\ldots,\tit_{m}\}$: we will freely identify those two sets.
Note that an element $\tilde{t}\in\tS$ can be viewed as acting on the set of alcoves $\alc(\cF)$ when it is considered as an element of $\tW\subset W$ but  it can also be viewed as acting on $\alc(\FCt)$ if $\tilde{t}$ is considered as acting on $\alc(\FCt)$ via the action defined at the end of the previous section. In the following lemma, we show that these two actions behaves well with one another. 
\begin{Lem}
\label{Atilde}
If $\tw \in \tW$, then 
$$\tw A_{0}\subset \tw \tA_{0}.$$
From where it follows that 
$$\ov{\bigcup_{w^{\circ}\in W^{\circ}} w^{\circ}\tw A_{0}}=\ov{\tw\tA_{0}}.$$
\end{Lem}
\begin{proof}
Let $\tit\in\tS$. Then  $\tit \tA_{0}$ is the unique alcove in $\alc(\FCt)$ which shares with $\tA_{0}$ a face of type $\tit$, hence we have 
$$\tit \tA_{0}= \tA_{0}\si_{H}$$
where $H$ is the hyperplane which supports the face of $\tA_{0}$ of type $\tit$.  By the previous lemma we see that
$$\tit A_{0}=A_{0}\si_{H}.$$
Hence since $A_{0}\subset \tA_{0}$, the first assertion follows. The second assertion follows from  $\ov{\tA_{0}}=\underset{\rho\in\Om^{\circ}}{\bigcup} \ov{A_{0}}\rho=\underset{w\in W^{\circ}}{\bigcup} \ov{wA_{0}}$.
\end{proof}

\subsection{The lowest two-sided cell of $\tW$}
\label{tildeL}
Let $\Lti$ denote the restriction of $L$ to $\Wti$.  By \cite[Corollary 1.4]{Bonnafe-Dyer}, it 
is a {\it positive} weight function. Note that we have $\tL(wtw^{-1})=L(t)$ for all $w\in W^{\circ}$. 
\begin{Th}\label{theo:cmin}
We have
$$\cb^L_{\min}(W)=W^\circ \cdot \cb^\Lti_{\min}(\tW)\quad\text{and}\quad N^{L}_{\si}(W)= W^\circ \cdot N^{\tL}_{\si}(\tW)$$
for all $\si\in \Om^{L}_{0}$.  
\end{Th}

\bigskip

\begin{proof}
First, since $\FCt=\cF^{L}$ and $A_{0}\subset \tA_{0}$ we see that
$$\UC^{\tL}(\tA_0)=\UC^L(A_0).$$
Then applying the results of the previous section we get
$$\cb^\Lti_{\min}(\tW)= \{\tw \in \tW~|~\tw(\tA_0) \not\subset \UC^{\tL}(\tA_0)\}$$
$$\cb^L_{\min}(W)= \{w \in W~|~w(A_0) \not\subset \UC^L(A_0)\}.$$
Let $w\in \cb^L_{\min}$ and write $w=w^{\circ}\tw$ where $w^{\circ}\in W^{\circ}$ and $\tw\in \tW$. We have $w^{\circ}\tw A_{0}\notin \UC^{L}(A_0)\text{ that is } w^{\circ}\tw A_{0}\notin \UC^{\tL}(\tA_0).$
Since the only hyperplane separating $\tw A_{0}$ and $w^{\circ}\tw A_{0}$ are hyperplanes of weight $0$, this implies that $\tw A_{0}\notin\UC^{\tL}(\tA_0)$. Hence, by Lemma \ref{Atilde}, we get that $\tw \tA_{0}\notin\UC^{\tL}(\tA_0)$ and $\tw\in \cb^{\tL}_{\min}(\tW)$ as required. 

\medskip

Conversely let $w^{\circ}\tw\in W^\circ \cdot \cbt^\Lti_{\min}$. Since $\tw\in  \cbt^\Lti_{\min}$ we have $\tw \tA_{0}\notin  \UC^{\tL}(\tA_0)$. By Lemma \ref{Atilde}, it follows that $\tw A_{0}\notin \UC^L(A_0)$ and $w^{\circ}\tw A_{0}\notin  \UC^L(A_0)$ as required.  

\medskip

The second equality in the theorem follows easily from the fact that  $\tOm_{0}=\Om^{L}_{0}$, Lemma \ref{Atilde} and 
$$N^{L}_{\si}(W)=\{w\in W\mid wA_{0}\subset \cC'_{\si}\}\quad\text{and}\quad N^{\tL}_{\si}(\tW)=\{\tw\in \tW\mid \tw\tA_{0}\subset \cC'_{\si}\}. $$

\end{proof}

\begin{Rem}
Since $\cb^{L}_{\min}(W)$, $\cb^{\tL}_{\min}(\tW)$ and $W^{\circ}$ are stable by taking the inverse, we get that
$$\cb^L_{\min}(W)=W^\circ \cdot \cb^\Lti_{\min}(\tW)= \cb^\Lti_{\min}(\tW)\cdot W^{\circ}=W^{\circ} \cdot  \cb^\Lti_{\min}(\tW)\cdot W^{\circ}.~\SS{\blacksquare}$$
\end{Rem}


\section{Kazhdan-Lusztig cells }

\subsection{Iwahori-Hecke algebras}
\label{Iwahori-Hecke}
Recall that $\Gamma$ is a totally ordered abelian group, whose law is denoted by $+$ and whose 
order relation is denoted by $\le$. Let $\cA$ be the group algebra of $\Gamma$ over $\nZ$. We shall use the following notation for $\cA$
$$\cA=\mathop{\oplus}_{\ga\in\Gamma} \nZ v^{\ga} \text{ where } v^{\ga}.v^{\ga'}=v^{\ga+\ga'}.$$
Let $L:W\rightarrow \Gamma$ be a weight function. For $s\in S$ we set $v_{s}=v^{L(s)}.$

\bigskip
 
 We denote by $\cH=\cH(W,S,L)$ the corresponding generic Iwahori-Hecke algebra, that is the free associative $\cA$-algebra with $\cA$-basis $\{T_{w}\mid w\in W\}$ and multiplication given by
$$\begin{array}{ccll}
(\text{a}) & \qquad & T_w T_{w'} =T_{ww'} & 
\text{if } \ell(ww')=\ell(w)+\ell(w')\qquad \\
(\text{b}) & \qquad & T_s^2 = (v_{s}-v_{s}^{-1}) T_s +1\quad& 
\text{if } s \in S. \\
\end{array}$$

\bigskip

Let $\bar{\ }$ be the involution of $\cA$ which takes $v^{\ga}$ to $v^{-\ga}$. It is well known that this map can be extended to a ring involution on $\cH$ (we will also denote it by $\bar{\ }$) via the formula:
$$\ov{\sum_{w\in W} a_{w}T_{w}}=\sum_{w\in W} \ov{a_{w}}T^{-1}_{w^{-1}}.$$
For all  $w\in W$, by \cite[Theorem 5.2]{bible}, there exists a unique element $C_{w}\in \cH$ such~that
\begin{itemize}
\item $\ov{C_{w}}=C_{w}$,
\item $C_{w}\in T_{w}+(\bigoplus_{y<w}\cA_{<0} T_{y})$ where $\cA_{<0}=\bigoplus_{\ga< 0} \nZ v^{\ga}$.
\end{itemize}
From the second condition, it is clear that the set $\{C_{w},w\in W\}$ forms an $\cA$-basis of $\cH$, known as the Kazhdan-Lusztig basis. 

\bigskip

We write 
$$C_{w}=\sum_{y\in W} P_{y,w}T_{w} \text{ where $P_{y,w}\in \cA$.}$$
The elements $P_{y,w}$ are called the Kazhdan-Lusztig polynomials and they satisfy the following properties (\cite[\S 5.3]{bible})
\begin{enumerate}
\item $P_{y,y}=1$
\item $P_{y,w}=0$ if $y\notle w$,
\item $P_{y,w}\in \cA_{<0}$ if $y<w$,
\item $P_{y,w}=v_{s}^{-1}P_{sy,w}$ if $sy>y$ and $sw<w$.
\end{enumerate}
Following \cite[\S 6]{bible}, we now describe the multiplication rule for the $C_{w}$'s. For each $y,w\in W$ and $s\in S$ such that $sy<y<w<sw$ we define $M^{s}_{y,w}\in\cA$ by the inductive condition 
$$M_{y,w}^{s}-\underset{sz<z}{\sum_{y<z<w}}P_{y,z}M_{z,w}^{s}-v_{s}P_{y,w}\in \cA_{<0}$$
and the symmetry condition
$$\ov{M_{y,w}^{s}}=M_{y,w}^{s}.$$
For $w\in W$ and $s\in S$, we obtain the following multiplication formula for the Kazhdan-Lusztig basis 
$$C_{s}C_{w}=
\begin{cases}
C_{sw}+\underset{z;sz<z<w}{\sum}M_{z,w}^{s}C_{z}, & \mbox{if }w<sw,\\ 
(v_{s}+v_{s}^{-1})C_{w}, & \mbox{if }sw<w.
\end{cases}
$$
Since $C_{s}=T_{s}+v_{s}^{-1}T_{1}$ for all $s\in S$, one can see that
\begin{equation*}
T_{s}C_{w}=
\begin{cases}
C_{sw}-v_{s}^{-1}C_{w}+\underset{z;sz<z<w}{\sum}M_{z,w}^{s}C_{z}, & \mbox{if }w<sw,\\ 
v_{s}C_{w}, & \mbox{if }sw<w.
\end{cases}
\end{equation*}
We will also need the following relation for Kazhdan-Lusztig polynomials. Let $y<w\in W$ and $s\in S$ such that $sw<w$. We have
\begin{align*}
P_{y,w}&=v_{s}P_{y,sw}+P_{sy,sw}-\underset{sz<z}{\sum_{y\leq z<sw}}P_{y,z}M_{z,sw}&\text{if $sy<y$}\tag{1}\\
P_{y,w}&=v^{-1}_{s}P_{sy,w}&\text{if $sy>y$}\tag{2}
\end{align*}


Finally we define the preorders  $\leq_{\cL},\leq_{\cR},\leq_{\cLR}$ as in \cite{bible}. For instance $\leq_{\cL}$ is the transitive closure of the relation:
$$y\leftarrow_{\cL} w\Longleftrightarrow \text{there exists $s\in S$ such that $M^{s}_{y,w}\neq 0$}.$$
Each of these preorders give rise to an equivalence relation $\sim_{\cL}, \sim_{\cR}$ and $\sim_{\cLR}$. The equivalence classes associated to $\sim_{\cL}$, $\sim_{\cR}$ and $\sim_{\cLR}$ are called left, right and two-sided cells, respectively. The partition of $W$ in cells {\it depends on the choice of the weight function}.  The
preorders $\leq_{\cL},\leq_{\cR},\leq_{\cLR}$ induce partial orders on the left, right and two-sided cells, respectively. 

\begin{Rem}
\label{left-right}
We have $x\sim_{\cL} y$ if and only if $x^{-1}\sim_{\cR} y^{-1}$  \cite[\S 8]{bible}. It follows easily that a union of left cells which is stable by taking the inverse is a also a union of two-sided cells.\finl 
\end{Rem}

\begin{Rem}
\label{pos-neg}
All the above can also be defined for weight functions which take negative values. It is shown in \cite{semicontinuity} that the partition into cells with respect to a weight function $L^{-}$ is the same as the partition into cells with respect to $L$ where $L$ is defined by 
$$L(s)=\begin{cases}
L(s)&\mbox{if $L^{-}(s)\geq 0$,}\\
-L(s)&\mbox{if $L^{-}(s)<0$.}
\end{cases}$$
Note that $L$ is a non-negative weight function. Hence the computation of Kazhdan-Lusztig cells can be reduced to the non-negative case.\finl 
\end{Rem}


\subsection{Kazhdan-Lusztig lowest two-sided cell}
\label{KL-lowest}
In the case where $L$ is a positive weight function, it is a well known fact that there is a lowest (Kazhdan-Lusztig) two sided cell with respect to the partial order $\leq_{\cLR}$. This two-sided cell has been thoroughly studied \cite{Shi-1987,Shi-1988,Xi-book,Bremke,jeju2} and it is equal to 
$$\cb^{L}_{\min}=\{xwy~|~w \in \WC,~x \add w \add y \text{ and } L(w)=\nu_L\}$$
(Hence the name for the set $\cb^{L}_{\min}$.)  The aim of this section is to show that this presentation also holds for non-negative weight function. 

\bigskip
Let $L$ be a non-negative weight function. Then, following Section \ref{semidirect}, we have $W=W^{\circ}\ltimes \tW$. 
Let $\tL$ be the restriction of $L$ to $\tW$. Then $\tL$ is a positive weight function on $\tW$ and  $\tL(wtw^{-1})=L(t)$ for all $w\in W^{\circ}$ and $t\in S^{+}$. We denote by $\tcH=\cH(\tW,\tS,\tL)$ the corresponding Hecke algebra. The group $W^{\circ}$ acts on $\tW$ and stabilizes $\tS$ and $\tL$, therefore it naturally acts on $\tcH$ and we can define the the semidirect product of algebras 
$$W^{\circ}\ltimes \tcH.$$
It has an $\cA$-basis $(x\cdot T_{\tw})_{x\in W^{\circ}, \tw\in \tW}$ and the map 
$$x\cdot T_{\tw}\longmapsto T_{x\tw}$$
 defines an isomorphism of $\cA$-algebras from $\tcH$ to $\cH(W,S,L)$. The cells of $(W,S,L)$ can then be described in the following way. 

\begin{Th}(\cite[Corollary 2.13]{semicontinuity})
\label{cell-til} The left cells  (respectively the two-sided cells) of $(W,S,L)$ are of the form $W^{\circ}.C$ (respectively $W^{\circ}.C.W^{\circ}$) where $C$ is a left cell (respectively a two-sided cell) of $(\tW,\tS,\tL)$.
\end{Th}

Finally we are ready to prove one of the main result of this paper which gives a general presentation of the lowest two-sided cell, including the case when the weight function $L$ vanishes on some generators. Note that this theorem is already known when the parameters are positive: see \cite{Shi-1987,Shi-1988} for the equal parameter case and 
\cite[\S 5]{Bremke}, \cite[Chapter 3]{Xi-book} and \cite{jeju2} for the unequal parameters.

\begin{Th}
\label{main1}
Let $(W,S,L)$ be an irreducible affine Weyl group and let $L$ be a non-negative weight function on $W$. Set 
$$\nu_{L}=\max_{I\subset S} w_{I}\quad\text{and}\quad \WC=\bigcup_{I\subset S} W_I$$
 where $I$ runs over the subset of $S$ such that $W_{I}$ is finite. Then the lowest two-sided cell of $W$ is
 $$\cb^{L}_{\min}=\{xwy~|~w \in \WC,~x \add w \add y \text{ and } L(w)=\nu_L\}.$$
Further, the decomposition of $\cb^{L}_{\min}$ into left cells is
$$\cb^{L}_{\min}=\bigcup_{\si\in\Om_{0}^{L}} N^{L}_{\si}.$$
\end{Th}
\begin{proof}
As mentioned previously this result is already known when $L$ is a positive weight function. On the one hand, by Theorem \ref{cell-til}, the lowest  two-sided Kazhdan-Lusztig cell of $W$ with respect to $\leq_{\cLR}$ and the weight function $L$ is
$$W^{\circ}\cdot\bc\cdot W^{\circ}$$
where $\bc$ is the lowest Kazhdan-Lusztig cell of $(\tW,\tL)$. But in this case we know that $$\bc=\bc^{\tL}_{\min}(\tW)$$
since $\tL$ is a positive weight function. 
Then the result follows from Theorem \ref{theo:cmin}, where we proved that
$$W^{\circ}\cdot\cb^{\tL}_{\min}(\tW)\cdot W^{\circ}=\bc_{\min}^{L}(W).$$
The left cells lying in $\cb^{L}_{\min}$ are of the form $W^{\circ}\cdot N^{\tL}_{\si}(\tW)$ where $N^{\tL}_{\si}(\tW)$ is a left cell of $\bc^{\tL}_{\min}(\tW)$. Once again, by Theorem \ref{theo:cmin}, we know that $W^{\circ}\cdot N^{\tL}_{\si}(\tW)=N^{L}_{\si}$ as required. 
\end{proof}


\section{On the asymptotic semicontinuity of the lowest two-sided cell}
In this section, we fix a totally ordered abelian group $\Ga$.
\subsection{Semicontinuity conjecture}
Let $\bar{S}=\{\o_{1},...,\o_{m}\}$ be the set of conjugacy classes in $S$. Let $\nZ[\bar{S}]$ be the free $\nZ$-module with basis $\bar{S}$ and let $V'=\nR\otimes_{\nZ} \nZ[\bar{S}]$. We shall view the elements of $\nZ[\bar{S}]$ as embedded in $V'$.  We denote by $\o^{\ast}_{1},\ldots,\o^{\ast}_{m}$ the dual basis of $\o_{1},\ldots, \o_{m}$. For $(n_{1},...,n_{m})\in \nQ^{r}-\{0\}$ we set 
$$H_{n_{1}\o_{1}+\ldots+n_{m}\o_{m}}:=\ker(\sum n_{i}\o^{\ast}_{i}).$$
Such an hyperplane is called a rational hyperplane. \\

Since $\Ga$ is torsion-free, the natural map $\Ga\rightarrow \nQ\otimes_{\nZ}\Ga$ is injective, so we shall 
view $\Ga$ as embedded in the $\nQ$-vector space $\nQ\otimes_{\nZ}\Ga$: in particular, if $r\in\nQ$ and $\ga\in \Ga$
then $r\ga$ is well-defined. Moreover, the order on $\Ga$ extends uniquely to a total order 
on $\nQ\otimes_{\nZ}\Ga$ that we still denote by $\leq$. \\

Following \cite{Bourbaki} we now introduce the notion of facets and chambers associated to a finite set of rational hyperplanes. Let $H=H_{n_{1}\o_{1}+\ldots+n_{m}\o_{m}}$ where $n_{i}\in\nQ$. We say that a weight function $L\in\text{Weight}(W,\Ga)$ lies on $H$ if we have 
$$\sum^{m}_{i=1}n_{i}L(\o_{i})=0.$$
We say that two weight functions $L,L'$ lie on the same side of $H_{n_{1}\o_{1}+\ldots+n_{m}\o_{m}}$ if we have
$$\sum^{m}_{i=1}n_{i}L(\o_{i})>0\text{ and }\sum^{m}_{i=1}n_{i}L'(\o_{i})>0$$
or
$$\sum^{m}_{i=1}n_{i}L(\o_{i})<0\text{ and }\sum^{m}_{i=1}n_{i}L'(\o_{i})<0.$$
Let $\fH$ be a finite set of rational hyperplanes. We define an equivalence relation on $\text{Weight}(W,\Ga)$: we write $L\sim_{\fH} L'$
if for all $H\in\fH$ we have either
\begin{enumerate}
\item $L,L'\in H$;
\item $L,L'$ lie on the same side of $H$.
\end{enumerate}
The equivalence classes associated to this relation will be called $\fH$-facets. A $\fH$-chamber is a $\fH$-facet $\cF$ such that no weight function in  $\cF$ lies on a hyperplane $H\in\fH$. 

\begin{Rem}
In \cite{Bourbaki} the equivalence relation $\sim_{\fH}$ is defined on $V'$ in the following way: $\la\sim_{\fH}\mu\in V'$ if for all for all $H\in\fH$ we have either
\begin{enumerate}
\item $\la,\mu\in H$;
\item $\la,\mu$ lie on the same side of $H$.
\end{enumerate}
There is a one to one correspondance between the equivalence classes of this relation in $V'$ and the sets of facets in $\text{Weight}(W,\Ga)$. We will freely identify those two sets.\finl
\end{Rem}


\medskip

For an $\fH$-facet $\cF$ we denote by $W_{\cF}$ the parabolic subgroup generated by 
$$\{s\in S| L(s)=0 \text{ for all $L\in \cF$}\}.$$ 
We say that a subset $X$ of $W$ is stable by translation by $W_{I}$ ($I\subset S$) on the left  (respectively on both sides) if for all $w\in X$ we have $zw\in X$ (respectively $zwz'\in X$) for all $z\in W_{I}$ (respectively for all $z,z'\in W_{I}$). Finally we denote by $\cC_{\cL}(L)$ (respectively $\cC_{\cLR}(L)$) the partition of $W$ into left (respectively two-sided) cells with respect to the weight function $L$. \\

We can now state the first author's conjecture for the partition of $W$ into cells. It is enough to state it for left and two-sided cells (see Remark \ref{left-right}).
\begin{conjecture}
\label{SC}
There exists a finite set of rational hyperplanes $\fH$ of $V'$ satisfying the following properties
\begin{enumerate}
\item If $L_{1},L_{2}$ are two weight functions belonging to the same $\fH$-facet $\cF$ then $\cC_{\cL}(L_{1})$ (respectively $\cC_{\cLR}(L_{1})$) and $\cC_{\cL}(L_{2})$ (respectively $\cC_{\cLR}(L_{2})$) coincide (we denote these partitions by $\cC_{\cL}(\cF)$ and $\cC_{\cLR}(\cF)$).
\item Let $\cF$ be an $\fH$-facet. Then the cells of $\cC_{\cL}(\cF)$ (respectively $\cC_{\cLR}(\cF)$) are the smallest subsets of $W$ which are at the same time unions of cells of $\cC_{\cL}(C)$ (respectively $\cC_{\cLR}(C)$) for all chamber $C$ such that $\cF\subset \bar{C}$ and stable by translation on the left (respectively on both sides) by $W_{\cF}$.
\end{enumerate}
\end{conjecture}

\begin{Rem}
\label{semi-pos-neg}
There are no restriction, such as non-negativity, on the weight functions in this conjecture. However, changing the sign of some values of the weight function $L$ has no effect on the partition of $W$ into cells (see Remark \ref{pos-neg}). Therefore, to prove the conjecture, it is enough to find a finite set of rational hyperplanes $\fH$ such that Statements (1) and (2) hold for all non-negative weight functions. Indeed, the conjecture will then hold for the minimal finite set of hyperplane  which contain $\fH$ and which is stable under the action of the linear maps $\tau_{i}:V'\rightarrow V'$ defined by $\tau_{i}(\o_{i})=-\o_{i}$ and $\tau_{i}(\o_{k})=\o_{k}$ if $k\neq i$.
\end{Rem}

When only looking at the lowest two-sided cell, Statement (1) in the above conjecture is a direct consequence of Theorem \ref{main1} (see below). We denote by $\rm{Left}\text{$(\bc^{L}_{\min})$}$ the set of left cells of $(W,S,L)$ lying in $\bc^{L}_{\min}$. 

\begin{Cor}[of Theorem 5.4]
\label{cor54}
Let $W$ be an irreducible affine Weyl group. There exists a finite set of rational hyperplanes $\fH$ such that 
\begin{enumerate}
\item If $L_{1},L_{2}$ are two weight functions belonging to the same $\fH$-facet $\cF$ then $\cb^{L_{1}}_{\min}=\cb^{L_{2}}_{\min}$ (we denote this set $\cb^{\cF}_{\min}$) and $\rm{Left}\text{$(\bc^{L_{1}}_{\min})$}$  and $\rm{Left}\text{$(\bc^{L_{2}}_{\min})$}$  coincide (we denote this partition by $\rm{Left}(\bc^{\cF}_{\min})$).
\end{enumerate}
\end{Cor}
\begin{proof}
By Theorem 5.4,  $\cb^{L}_{\min}$ only depends on the values of $L$ on the elements of the set $\cW$ (see Section \ref{def-lowest}). But $\cW$ is finite, hence it  is easy to find a finite set of rational hyperplanes such that (1) holds. 
\end{proof}

In the remaining of this paper, we will prove the following theorem which is concerned with the asymptotic behaviour of the lowest two-sided cell, hence providing new evidences for the semicontinuity conjecture. 

\begin{Th}
\label{main2}
Let $W$ be an irreducible affine Weyl group. There exists a finite set of rational hyperplanes 
$\fH$ satisfying property (1) in Corollary \ref{cor54} and satisfying the following property: 
if $\cF$ is an $\fH$-facet which is contained in $H_{\o_{i}}$ for some $i$,  then $\cb^{\cF}_{\min}$ 
is a union of two-sided cells  of $\cC_{\cLR}(C)$ and the left cells in 
$\rm{Left}(\bc^{\cF}_{\min})$ are union of left cells of $\cC_{\cL}(C)$ for all $\fH$-facets $C$ such that $\cF\subset \bar{C}$.
\end{Th}

\begin{Rem} 
(a) At the end of this Theorem, we really mean {\it for all $\fH$-facets $C$ such that $\cF\subset \bar{C}$} and not {\it for all $\fH$-chambers $C$ such that $\cF\subset \bar{C}$} . It is clear that if it is true for all $\fH$-facets such that $\cF\subset \bar{C}$ then it will also be true for all $\fH$-chambers $C'$ such that $\cF\subset \bar{C'}$. But the converse is true only if the semicontinuity conjecture holds! \\

\noindent
(b) Arguying as in Remark \ref{semi-pos-neg}, to prove the theorem,  it is enough to find a finite set of rational hyperplanes such that (1) and (2) holds for non-negative weight functions and then take its closure under the action of the $\tau_{i}$'s.\\

\noindent
(c) To prove the theorem, it is ``enough'' to show that the left cells in $\rm{Left}(\bc^{\cF}_{\min})$ are union of left cells of $\cC_{\cL}(C)$ for all $\fH$-facets $C$ such that $\cF\subset \bar{C}$. Indeed $\cb^{\cF}_{\min}$ is stable by taking the inverse, hence, by Remark \ref{left-right}, if it is a union of left cells of $\cC_{\cL}(C)$, it is also a union of two-sided cells of  $\cC_{\cLR}(C)$.\finl 
\end{Rem}

\subsection{Irreducible affine Weyl groups of type $\tilde{\Brm}_n$ , 
$\tilde{\Frm}_4$ or $\tilde{\Grm}_2$.} 
\label{typeBFG}

Let $(W,S)$ be an irreducible affine Weyl group of one of the following types
$$
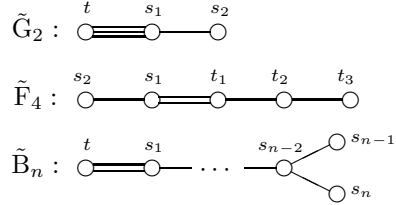
\renewcommand{\arraystretch}{1.6}
\begin{array}{ccccccccccccc}
\tilde{\Grm}_{2}:&
\begin{picture}(100,20)
\put(0,2){\circle{6}}\put(-1,9){$\SS{t}$}
\put(25,2){\circle{6}}\put(22,9){$\SS{s_1}$}
\put(50,2){\circle{6}}\put(47,9){$\SS{s_2}$}
\put(28,2){\line(1,0){19}}
\put(2.4,3.8){\line(1,0){20.2}}
\put(2.4,0.2){\line(1,0){20.2}}
\put(3,2){\line(1,0){19}}
\end{picture}\\
\tilde{\Frm}_{4}:&
\begin{picture}(100,20)
\put(0,2){\circle{6}}\put(-5,9){$\SS{s_2}$}
\put(25,2){\circle{6}}\put(22,9){$\SS{s_1}$}
\put(50,2){\circle{6}}\put(47,9){$\SS{t_1}$}
\put(75,2){\circle{6}}\put(70,9){$\SS{t_2}$}
\put(100,2){\circle{6}}\put(95,9){$\SS{t_3}$}
\put(3,2){\line(1,0){19}}
\put(27.7,3.2){\line(1,0){19.4}}
\put(27.7,0.8){\line(1,0){19.4}}
\put(53,2){\line(1,0){19}}
\put(78,2){\line(1,0){19}}
\end{picture}\\
\tilde{\Brm}_{n}:&
\begin{picture}(100,20)
\put(0,2){\circle{6}}\put(-1,9){$\SS{t}$}
\put(25,2){\circle{6}}\put(22,9){$\SS{s_1}$}
\put(75,2){\circle{6}}\put(65,9){$\SS{s_{n-2}}$}
\put(95,12){\circle{6}}\put(100,12){$\SS{s_{n-1}}$}
\put(95,-8){\circle{6}}\put(100,-8){$\SS{s_{n}}$}
\put(77.6,3.5){\line(2,1){14.7}}
\put(77.6,0.5){\line(2,-1){14.7}}
\put(2.7,3.2){\line(1,0){19.4}}
\put(2.7,0.8){\line(1,0){19.4}}
\put(28,2){\line(1,0){12}}
\put(72,2){\line(-1,0){12}}
\put(43,-1){$\cdots$}
\end{picture}
\end{array}
$$

\bigskip
\noindent
Then $|\bar{S}|=2$. We set $\bar{S}=\{\bs,\bt\}$ where $\bs$ (respectively $\bt$) is the subset of $S$ which consists of all the generators named with the letter $s$ (respectively~$t$). In this case, we will identify $\nZ[\bar{S}]$ with $\nZ^{2}$ through $(i,j)\longrightarrow i\bs+j\bt$. 

\medskip

Let $m_{1},m_{2}\in\nQ_{>0}$. We define the following  finite set of rational hyperplanes of~$V'$
$$\fH(m_{1},m_{2}):=\{H_{\bs+ m_{1}\bt}, H_{\bs- m_{1}\bt}, H_{\bs+ m_{2}\bt}, H_{\bs-m_{2}\bt},H_{\bs},H_{\bt}\}.$$
Note that $\fH(m,M)$ is stable under the actions of the $\tau_{i}$ (see Remark \ref{semi-pos-neg}).
In Figure~\ref{fH-mM}, we draw the finite set of hyperplanes $\fH(m,M)$ for some choice of constants $M,m\in \nQ_{>0}$. 
\psset{xunit=.55cm}
\psset{yunit=.55cm}
\begin{figure}[h!]
\caption{Set of hyperplanes $\bar{\fH}(m,M)$}
\label{fH-mM}
\begin{center}
\begin{pspicture}(-5,-5.5)(5,5.5)
\psline(-5,0)(5,0)
\psline(0,-5)(0,5)
\psline(-2,-5)(2,5)
\psline(-2,5)(2,-5)
\psline(-5,2)(5,-2)
\psline(-5,-2)(5,2)

\rput(1,0){$\bullet$}
\rput(1.2,.25){${\small \bs}$}

\rput(0,1){$\bullet$}
\rput(.25,1.25){${\small \bt}$}

\rput(.5,5){$H_{\bs}$}
\rput(5.5,.2){$H_{\bt}$}
\rput(4.7,1.15){$H_{\bs-m_{1}\bt}$}
\rput(2.9,4.5){$H_{\bs-m_{2}\bt}$}
\rput(3,.3){$\cC_{1}$}
\rput(.7,3){$\cC_{2}$}
\rput(2,2){$\cC$}
\end{pspicture}
\end{center} 
\end{figure}

\noindent
The set of weight functions corresponding to the $\fH$-facet $\cC_{1}$  of $V'$ is
$$\{L\in \text{Weight}(W,\Ga)\mid L(\bs)>m_{1}\cdot L(\bt)\text{ and } L(\bs),L(\bt)>0\}.$$
\begin{Th}
\label{main2-BFG}
There exists $m_{1},m_{2}\in\nQ_{>0}$ such that Theorem \ref{main2} holds for $\fH(m_{1},m_{2})$.  
\end{Th}
The proof of this theorem will be given in Section \ref{proof-BFG}.

\begin{Rem}
Note that this theorem is equivalent to Theorem 1. Let $L$ be a non-negative weight function on $W$ which vanishes on a proper non-empty subset $S^{\circ}$ of $S$. Then we have either $L\in H_{\bs}$ or $H_{\bt}$. Assume that $L\in H_{\bt}$, that is  $L(t)=0$ for all $t\in \bt$. Let $\bc$ be an $L$-cell contained in $\cb^{L}_{\min}$ (note that we have either $\bc=\cb^{L}_{\min}$ or $\bc=N^{L}_{\si}$ for some $\si\in\Om^{L}_{0}$).  Then Theorem 1 implies that there exists an integer $m$ such that for all weight functions $L'$ such that $L'(\bs)>m\cdot L'(\bt)$, the set $\bc$ is a union of $L'$-cells. In other words, $\bc$ is a union of $L'$-cell  for all weight function $L'$ in $\cC_{1}$ (with $m_{1}=m$). Conversely, if Theorem 6.9 holds, then Theorem 1 holds for all for any integer $m$ greater than $m_{1}$. The case $L\in H_{\bs}$ is similar using $m=1/m_{2}$.\end{Rem}

\subsection{Irreducible affine Weyl group of type $\tilde{\Crm}$.} 
\label{typeC}
Let $W$ be an irreducible affine Weyl group of type $\tilde{\Crm}$ with diagram as follows
$$\begin{picture}(100,20)
\put(5,2){\circle{6}}\put(3,9){$\SS{t}$}
\put(30,2){\circle{6}}\put(25,9){$\SS{s_1}$}
\put(80,2){\circle{6}}\put(70,9){$\SS{s_{n-1}}$}
\put(105,2){\circle{6}}\put(103,9){$\SS{t'}$}
\put(7.7,3.2){\line(1,0){19.4}}
\put(7.7,0.8){\line(1,0){19.4}}
\put(82.7,3.2){\line(1,0){19.4}}
\put(82.7,0.8){\line(1,0){19.4}}
\put(33,2){\line(1,0){12}}
\put(77,2){\line(-1,0){12}}
\put(48,-1){$\cdots$}
\end{picture}
$$

\medskip
\noindent
Then $|\bar{S}|=3$. We set $\bar{S}=\{\bt,\bs,\bt'\}$ where $\bt=\{t\}$, $\bs=\{s_{1},\ldots,s_{n-1}\}$ and $\bt'=\{t'\}$.  In this case, we will identify $\nZ[\bar{S}]$ with $\nZ^{3}$ through $(i,j,k)\longrightarrow i\bt+j\bs+k\bt'$. \\

Let $\bm=(m_{1},\ldots,m_{6})\in\nQ^{6}_{>0}$. We define the following  finite set of rational hyperplanes of~$V'$
$$\fH(\bm):=\{H_{\bs},H_{\bt},H_{\bt'},H_{\bt-\bt'}, H_{\bt - m_{1}\bs},H_{\bt'- m_{2}\bs},H_{\bt- m_{3}(\bs+\bt')}, $$
$$H_{\bt'- m_{4}(\bs+\bt)}, H_{(\bt-\bt')\pm m_{5}\bs},H_{(\bt+\bt')- m_{6}\bs}\}.$$
We then set $\bar{\fH}(\bm)$ to be the closure of $\fH(\bm)$ under the actions of the $\tau_{i}$ (see Remark~\ref{semi-pos-neg}).
In Figure~\ref{fm}, we draw the finite set of hyperplanes $\fH(\bm)$ for some choice of constants $\bm\in \nQ^{6}_{>0}$. We intersect on the affine hyperplane with equation  $\bs^{\ast}(\mu)=1$. We put an arrow in a chamber $\cC$ pointing towards $\cF\subset V'$ to indicate that $\ov{\cC}\cap \cF\neq \emptyset$.  

\begin{figure}[h!]
\caption{Set of hyperplanes $\fH(\bm)$}
\label{fm}
\begin{center}
\psset{xunit=.5cm}
\psset{yunit=.5cm}
\begin{pspicture}(0,-2)(21,22)
\psset{linewidth=.1mm}
\SpecialCoor


%


\psline{->}(0,0)(0,20)
\psline{->}(0,0)(20,0)

\psline[linewidth=.5mm](4,0)(20,9)

\psline[linewidth=.5mm](0,4)(9,20)

\psline[linewidth=.5mm](2,0)(0,2)

\psline(0,0)(17,17)
\psline[linewidth=.5mm](1,1)(17,17)

\psline[linewidth=.5mm](6,1.2)(6,20)
\psline(6,0)(6,1.2)

\psline[linewidth=.5mm](1.2,6)(20,6)
\psline(0,6)(1.2,6)

\psline(4,0)(20,16)
\psline[linewidth=.5mm](10,6)(20,16)

\psline(0,4)(16,20)
\psline[linewidth=.5mm](6,10)(16,20)

\rput(15.5,2){{\psline{->}(0,0)(2,0)}}
\rput(19.5,2){$H_{\bf{s}}\cap H_{\bf{t'}}$}

\rput(17.5,6.5){{\psline{->}(0,0)(3,0)}}
\rput(22,6.5){$H_{\bf{s}}\cap H_{\bf{t'}}$}

\rput(17.5,6.5){{\psline{->}(0,0)(3,1.5)}}
\rput(21,8){$H_{\bf{s}}$}

\rput(16,9){{\psline{->}(0,0)(3,1.5)}}
\rput(19.7,11){$H_{\bf{s}}$}
\rput(16,9){{\psline{->}(0,0)(3,3)}}
\rput(20,12.5){$H_{\bf{s}}\cap H_{\bf{t}-\bf{t'}}$}

\rput(13,11){{\psline{->}(0,0)(3,3)}}
\rput(16.8,14.7){$H_{\bf{s}}\cap H_{\bf{t}-\bf{t'}}$}

\rput(9,11){$\cC'_{1}$}

\rput(11,9){$\cC_{1}$}
\rput(15,7.6){$\cC_{2}$}
\rput(16.6,6.5){$\cC_{3}$}
\rput(16.6,3.5){$\cC_{4}$}
\rput(3.5,1.5){$\cC_{5}$}
\rput(1,.4){{\footnotesize $\cC_{6}$}}
\rput(.4,1){{\footnotesize $\cC'_{6}$}}


\rput(0,0){$\bullet$}

\rput(-1.2,-1.4){$H_{\bt}\cap H_{\bt'}$}
\rput(-1,-1){\psline{->}(0,0)(.9,.9)}
\rput(13,-1){$H_{\bf{t'}}$}
\rput(14,-1){\psline{->}(0,0)(0,.9)}

\rput(20,-.6){$e_{\bf{t}}$}
\rput(-.6,20){$e_{\bf{t'}}$}

\rput{45}(10,10.5){$H_{\bt-\bt'}$}
\rput{45}(14,9.3){$H_{\bt-\bt'-m_{5}\bs}$}
\rput{45}(10,14.6){$H_{\bt-\bt'+m_{5}\bs}$}
\rput{28}(12,3.8){$H_{\bt-m_{3}(\bt'+\bs)}$}
\rput{63}(4.5,13){$H_{\bt'-m_{4}(\bt+\bs)}$}

\rput{-45}(1.3,1.3){$H_{\bt+\bt'-m_{6}\bs)}$}

\end{pspicture}
\end{center}
\end{figure}

\noindent
The chamber $\cC_{1}$ corresponds to the weight functions 
$$\{L\mid L(\bt)>L(\bt'), L(\bt')>m_{2}\cdot L(\bs), L(\bt)-L(\bt')<m_{5}\cdot L(\bs) \}.$$
\begin{Th}
\label{main2-C}
There exist $\bm\in\nQ_{>0}^{6}$ such that Theorem~\ref{main2} holds for $\bar{\fH}(\bm)$.
\end{Th}

\begin{Rem}  
Theorem~\ref{main2-C} is stronger than Theorem 1. Indeed, let $L$ be a non-negative weight function on $W$ which vanishes on a proper non-empty subset $S^{\circ}$ of $S$. Then, Theorem 1 only gives us informations on weight functions $L'$ satisfying $L'(s^{+})=L(s^{+})$ for all $s^{+}\in S^{+}$. However, in Theorem~\ref{main2-C}, we may have $L'(s^{+})\neq L(s^{+})$ for some $s^{+}\in S^{+}$. For instance, if $L(\bt)=L(\bt')>0$ and $L(\bs)=0$, then Theorem~\ref{main2-C} implies that  for all weight functions $L'$ such that 
$$\{L'\mid L'(\bt)\geq L'(\bt'), L'(\bt')>m_{2}\cdot L'(\bs), L'(\bt)-L'(\bt')<m_{5}\cdot L'(\bs) \}$$
any $L$-cell contained in $\cb^{L}_{\min}$ is a union of $L'$-cells. But Theorem~1 does not tell us anything in this case as we do not have $L'(\bt)=L(\bt)=L(\bt')=L'(\bt')$.  We now show in more details that Theorem~\ref{main2-C} implies Theorem~1.
\begin{enumerate}
\item[(i)]  Assume that $S^{\circ}=\bs$ and $L(\bt)=L(\bt')$.  Let $m\geq \max\{m_{1},m_{2}\}$. Then if $L'$ satisfies $L'(\bt)=L'(\bt')>mL'(\bs)$ we must have $L'\in\bar{\cC_{1}}\cap\bar{\cC_{1}'}$. But $\bar{\cC_{1}}\cap\bar{\cC_{1}'}$ contains $L$ in its closure, therefore Theorem ~\ref{main2-C} tells us that any $L$-cells included in $\cb^{L}_{\min}$ is a union of $L'$ cells (see also Claim~\ref{87}). 
\item[(ii)]  Assume that $S^{\circ}=\bs$  and $L(\bt)>L(\bt')$.  Let $m$ be such that 
$$\frac{L(\bt)-L(\bt')}{m_{5}}m>L(\bt)\text{ and } m>m_{2}.$$
Then if $L'$ satisfies $L'(\bt)=L(\bt)>mL'(\bs)$ and $L'(\bt')=L(\bt')>mL'(\bs)$ we must have
$$L'(\bt')-L'(\bt)>m_{5}L'(\bs) \text{ and } L'(\bt')>m_{2}L'(\bs)$$
that is $L'\in \cC_{2}\cup\cC_{3}\cup (\ov{\cC_{2}}\cap \ov{\cC_{3}})$. Then Theorem 1 then follows from Theorem~\ref{main2-C} (see also Claim~\ref{86}).
\item[(iii)] Assume that $S^{\circ}=\bs\cup\bt'$.  Let $m$ be  such that  $m>2m_{3}$. Then if $L'$ satisfies $L'(\bt)=L(\bt)>mL'(\bs)$ and $L'(\bt)=L(\bt)>mL'(\bt')$ we must have
$$2L'(\bt)>2m_{3}(L'(\bt')+L'(\bs))$$
that is $L'\in \cC_{3}\cup\cC_{4}\cup (\ov{\cC_{3}}\cap \ov{\cC_{4}})$. Then Theorem~1 then follows from Theorem~\ref{main2-C} (see also Claim \ref{84}).
\item[(iv)] Assume that $S^{\circ}=\bt'\cup\bt$ that is $L\in H_{\bt'}\cap H_{\bt}$.  Let $m$ be such that  $m>\frac{2}{m_{6}}$. Then if $L'$ satisfies
$L'(\bs)=L(\bs)>mL'(\bt)$ and $L'(\bs)=L(\bs)>mL'(\bt')$ we must have $m_{6}L'(\bs)>L'(\bt)+L'(\bt')$. In other words $L'\in \cC_{6}$ and Theorem~1 follows from Theorem~\ref{main2-C} (see also Claim~\ref{85}).
\item[(v)] Assume that $S^{\circ}=\bt'$. Then the result is trivial since $\cb^{L}_{\min}=\cb^{L'}_{\min}$ for all $L,L'$ such that 
\begin{itemize}
\item  $L(\bt),L(\bt')>L(\bt')=0$
\item $L'(\bs)=L(\bs),L'(\bt)=L(\bt)>L'(\bt')$ and $L'(\bt')>0$.
\end{itemize}
\end{enumerate}
\end{Rem}

\begin{Rem}
\label{problem}
In this remark, we explain why we do need an hyperplane of the form $H_{(\bt-\bt')-m_{5}\bs}$ in our finite set of hyperplanes in Theorem 
\ref{main2-C}, eventhough the lowest two-sided cell is the same whether the weight function lies in $\cC_{1}$ or $\cC_{2}$.
Assume that $W$ is of type $\tilde{\Crm}_{2}$. It is shown in \cite{jeju4,comp} that Conjecture \ref{SC} holds for the following set of hyperplanes
$$\fH:=\{\cH_{(1,0,0)},\cH_{(0,1,0)},\cH_{(0,0,1)},\cH_{(\e,\e,0)},\cH_{(0,\e,\e)},\cH_{(\e,0,\e)},\cH_{(\e,\e,\e)},\cH_{(\e,2\e,\e)}\}.$$
We describe this set of hyperplanes in Figure \ref{essential-C2}, projecting on the affine hyperplane with equation $\bs^{\ast}(\mu)=1$.

\begin{figure}[h!]
\caption{ Hyperplanes in $\fH$.}
\label{essential-C2}
\begin{center}
\psset{xunit=1.5cm}
\psset{yunit=1.5cm}
\begin{pspicture}(0,-0.3)(5.5,5.3)
\psgrid[subgriddiv=1,griddots=10,gridlabels=8 pt](0,0)(5,5)

\psline(0,0)(5,5)
\psline(1,0)(1,1)
\psline(1,1)(5,1)
\psline(1,0)(0.5,.5)
\psline(1,0)(5,4)
\psline(2,0)(1,1)
\psline(2,0)(5,3)


\psline(0,1)(1,1)

\psline(.5,.5)(0,1)

\psline(1,1)(1,5)

\psline(1,1)(0,2)

\psline(0,2)(3,5)
\psline(0,1)(4,5)


\rput(3.5,.5){$A_{1}$}
\rput(2,.5){$A_{2}$}
\rput(1.5,.2){$A_{3}$}
\rput(1.2,.5){$A_{4}$}
\rput(1.5,.8){$A_{5}$}


\rput(4.5,1.5){$C_{3}$}
\rput(3,1.5){$C_{2}$}
\rput(2,1.5){$C_{1}$}

\rput(.5,.25){$B_{2}$}
\rput(.8,.5){$B_{1}$}

\psline{->}(0,0)(0,5)
\psline{->}(0,0)(5,0)
\end{pspicture}
\end{center}
\end{figure}
 
In the figure below, we show the partition of $W$ into cells for a weight function in $C_{1}$ and for a weight function $L'$ such that $L'(\bs)=0$ and $L'(\bt)>L'(\bt')>0$. The set $\cb^{L'}_{\min}$ consists of the yellow alcoves. We see that  $\cb^{L'}_{\min}$ is NOT a union of cells of $(W,S,L)$. Hence, we need the hyperplane $H_{(\bt-\bt')-m_{5}\bs}$ so that there are no weight function $L'$ such that  $L'(\bs)=0$ and $L'(\bt)>L'(\bt')>0$ and which lies  in the closure of $\cC_{1}$ (see Figure \ref{fm}).

\psset{linewidth=.13mm}
\begin{textblock}{10}(2,6)
\psset{unit=.5cm}

\begin{center}
\begin{pspicture}(-6,-6)(6,6)

\pspolygon[fillstyle=solid,fillcolor=lightgray](0,0)(0,1)(.5,.5)

\pspolygon[fillstyle=solid,fillcolor=red!5!yellow!42!](0,0)(0,-6)(-6,-6)
\pspolygon[fillstyle=solid,fillcolor=red!5!yellow!42!](1,-1)(1,-6)(6,-6)
\pspolygon[fillstyle=solid,fillcolor=red!5!yellow!42!](2,0)(6,0)(6,-4)
\pspolygon[fillstyle=solid,fillcolor=red!5!yellow!42!](1,1)(6,6)(6,1)
\pspolygon[fillstyle=solid,fillcolor=red!5!yellow!42!](1,3)(1,6)(4,6)
\pspolygon[fillstyle=solid,fillcolor=red!5!yellow!42!](0,2)(0,6)(-4,6)
\pspolygon[fillstyle=solid,fillcolor=red!5!yellow!42!](-1,1)(-6,1)(-6,6)
\pspolygon[fillstyle=solid,fillcolor=red!5!yellow!42!](-2,0)(-6,0)(-6,-4)
\pspolygon[fillstyle=solid,fillcolor=green!80!black!70!](-1,0)(-2,0)(-6,-4)(-6,-5)
\pspolygon[fillstyle=solid,fillcolor=green!80!black!70!](1,0)(2,0)(6,-4)(6,-5)
\pspolygon[fillstyle=solid,fillcolor=green!80!black!70!](1,2)(1,3)(4,6)(5,6)
\pspolygon[fillstyle=solid,fillcolor=green!80!black!70!](-0.5,1.5)(-5,6)(-4,6)(0,2)(0,1)


\pspolygon[fillstyle=solid,fillcolor=PineGreen](0,0)(.5,.5)(6,-5)(6,-6)
\pspolygon[fillstyle=solid,fillcolor=PineGreen](0,0)(-.5,.5)(-6,-5)(-6,-6)
\pspolygon[fillstyle=solid,fillcolor=PineGreen](1,1)(6,6)(5,6)(.5,1.5)
\pspolygon[fillstyle=solid,fillcolor=PineGreen](-1,1)(-6,6)(-5,6)(-.5,1.5)%


\pspolygon[fillstyle=solid,fillcolor=orange](0,1)(-.5,1.5)(-1,1)(-.5,.5)
\pspolygon[fillstyle=solid,fillcolor=orange](0,1)(.5,1.5)(1,1)(.5,.5)

\pspolygon[fillstyle=solid,fillcolor=BurntOrange](0,0)(0,1)(-.5,.5)

\pspolygon[fillstyle=solid,fillcolor=RedOrange](0,1)(.5,1.5)(0,2)

\pspolygon[fillstyle=solid,fillcolor=Aquamarine](0,0)(0,-1)(0.5,-.5)

\pspolygon[fillstyle=solid,fillcolor=NavyBlue](0.5,1.5)(0,2)(0,6)(1,6)(1,2)
\pspolygon[fillstyle=solid,fillcolor=NavyBlue](.5,.5)(1,1)(6,1)(6,0)(1,0)
\pspolygon[fillstyle=solid,fillcolor=NavyBlue](.5,-.5)(0,-1)(0,-6)(1,-6)(1,-1)
\pspolygon[fillstyle=solid,fillcolor=NavyBlue](-.5,.5)(-1,1)(-6,1)(-6,0)(-1,0)

\psline(-6,-6)(-6,6)
\psline(-5,-6)(-5,6)
\psline(-4,-6)(-4,6)
\psline(-3,-6)(-3,6)
\psline(-2,-6)(-2,6)
\psline(-1,-6)(-1,6)
\psline(0,-6)(0,6)
\psline(1,-6)(1,6)
\psline(2,-6)(2,6)
\psline(3,-6)(3,6)
\psline(4,-6)(4,6)
\psline(5,-6)(5,6)
\psline(6,-6)(6,6)

\psline(-6,-6)(6,-6)
\psline(-6,-5)(6,-5)
\psline(-6,-4)(6,-4)
\psline(-6,-3)(6,-3)
\psline(-6,-2)(6,-2)
\psline(-6,-1)(6,-1)
\psline(-6,0)(6,0)
\psline(-6,1)(6,1)
\psline(-6,2)(6,2)
\psline(-6,3)(6,3)
\psline(-6,4)(6,4)
\psline(-6,5)(6,5)
\psline(-6,6)(6,6)

\psline(0,0)(1,1)
\psline(0,1)(1,0)

\psline(1,0)(2,1)
\psline(1,1)(2,0)

\psline(2,0)(3,1)
\psline(2,1)(3,0)

\psline(3,0)(4,1)
\psline(3,1)(4,0)

\psline(4,0)(5,1)
\psline(4,1)(5,0)

\psline(5,0)(6,1)
\psline(5,1)(6,0)

\psline(-1,0)(0,1)
\psline(-1,1)(0,0)

\psline(-2,0)(-1,1)
\psline(-2,1)(-1,0)

\psline(-3,0)(-2,1)
\psline(-3,1)(-2,0)

\psline(-4,0)(-3,1)
\psline(-4,1)(-3,0)

\psline(-5,0)(-4,1)
\psline(-5,1)(-4,0)

\psline(-6,0)(-5,1)
\psline(-6,1)(-5,0)

\psline(0,1)(1,2)
\psline(0,2)(1,1)

\psline(1,1)(2,2)
\psline(1,2)(2,1)

\psline(2,1)(3,2)
\psline(2,2)(3,1)

\psline(3,1)(4,2)
\psline(3,2)(4,1)

\psline(4,1)(5,2)
\psline(4,2)(5,1)

\psline(5,1)(6,2)
\psline(5,2)(6,1)

\psline(-1,1)(0,2)
\psline(-1,2)(0,1)

\psline(-2,1)(-1,2)
\psline(-2,2)(-1,1)

\psline(-3,1)(-2,2)
\psline(-3,2)(-2,1)

\psline(-4,1)(-3,2)
\psline(-4,2)(-3,1)

\psline(-5,1)(-4,2)
\psline(-5,2)(-4,1)

\psline(-6,1)(-5,2)
\psline(-6,2)(-5,1)

\psline(0,2)(1,3) 
\psline(0,3)(1,2) 

\psline(1,2)(2,3) 
\psline(1,3)(2,2) 

\psline(2,2)(3,3) 
\psline(2,3)(3,2) 
 
\psline(3,2)(4,3) 
\psline(3,3)(4,2) 

\psline(4,2)(5,3) 
\psline(4,3)(5,2) 

\psline(5,2)(6,3) 
\psline(5,3)(6,2) 

\psline(-1,2)(0,3) 
\psline(-1,3)(0,2) 

\psline(-2,2)(-1,3) 
\psline(-2,3)(-1,2) 

\psline(-3,2)(-2,3) 
\psline(-3,3)(-2,2) 

\psline(-4,2)(-3,3) 
\psline(-4,3)(-3,2) 

\psline(-5,2)(-4,3) 
\psline(-5,3)(-4,2) 

\psline(-6,2)(-5,3) 
\psline(-6,3)(-5,2)

\psline(0,3)( 1,4) 
\psline(0,4)( 1,3) 

\psline(1,3)( 2,4) 
\psline(1,4)( 2,3) 

\psline(2,3)( 3,4) 
\psline(2,4)( 3,3) 
 
\psline(3,3)( 4,4) 
\psline(3,4)( 4,3) 

\psline(4,3)( 5,4) 
\psline(4,4)( 5,3) 

\psline(5,3)( 6,4) 
\psline(5,4)( 6,3) 

\psline(-1,3)( 0,4) 
\psline(-1,4)( 0,3) 

\psline(-2,3)( -1,4) 
\psline(-2,4)( -1,3) 

\psline(-3,3)( -2,4) 
\psline(-3,4)( -2,3) 

\psline(-4,3)( -3,4) 
\psline(-4,4)( -3,3) 

\psline(-5,3)( -4,4) 
\psline(-5,4)( -4,3) 

\psline(-6,3)( -5,4) 
\psline(-6,4)( -5,3)

\psline(0,4)(1,5)
\psline(0,5)(1,4)

\psline(1,4)(2,5)
\psline(1,5)(2,4)

\psline(2,4)(3,5)
\psline(2,5)(3,4)

\psline(3,4)(4,5)
\psline(3,5)(4,4)

\psline(4,4)(5,5)
\psline(4,5)(5,4)

\psline(5,4)(6,5)
\psline(5,5)(6,4)

\psline(-1,4)(0,5)
\psline(-1,5)(0,4)

\psline(-2,4)(-1,5)
\psline(-2,5)(-1,4)

\psline(-3,4)(-2,5)
\psline(-3,5)(-2,4)

\psline(-4,4)(-3,5)
\psline(-4,5)(-3,4)

\psline(-5,4)(-4,5)
\psline(-5,5)(-4,4)

\psline(-6,4)(-5,5)
\psline(-6,5)(-5,4)

\psline(0,5)(1,6)
\psline(0,6)(1,5)

\psline(1,5)(2,6)
\psline(1,6)(2,5)

\psline(2,5)(3,6)
\psline(2,6)(3,5)

\psline(3,5)(4,6)
\psline(3,6)(4,5)

\psline(4,5)(5,6)
\psline(4,6)(5,5)

\psline(5,5)(6,6)
\psline(5,6)(6,5)

\psline(-1,5)(0,6)
\psline(-1,6)(0,5)

\psline(-2,5)(-1,6)
\psline(-2,6)(-1,5)

\psline(-3,5)(-2,6)
\psline(-3,6)(-2,5)

\psline(-4,5)(-3,6)
\psline(-4,6)(-3,5)

\psline(-5,5)(-4,6)
\psline(-5,6)(-4,5)

\psline(-6,5)(-5,6)
\psline(-6,6)(-5,5)

\psline(0,0)(1,-1)
\psline(0,-1)(1,0)

\psline(1,0)(2,-1)
\psline(1,-1)(2,0)

\psline(2,0)(3,-1)
\psline(2,-1)(3,0)

\psline(3,0)(4,-1)
\psline(3,-1)(4,0)

\psline(4,0)(5,-1)
\psline(4,-1)(5,0)

\psline(5,0)(6,-1)
\psline(5,-1)(6,0)

\psline(-1,0)(0,-1)
\psline(-1,-1)(0,0)

\psline(-2,0)(-1,-1)
\psline(-2,-1)(-1,0)

\psline(-3,0)(-2,-1)
\psline(-3,-1)(-2,0)

\psline(-4,0)(-3,-1)
\psline(-4,-1)(-3,0)

\psline(-5,0)(-4,-1)
\psline(-5,-1)(-4,0)

\psline(-6,0)(-5,-1)
\psline(-6,-1)(-5,0)

\psline(0,-1)(1,-2)
\psline(0,-2)(1,-1)

\psline(1,-1)(2,-2)
\psline(1,-2)(2,-1)

\psline(2,-1)(3,-2)
\psline(2,-2)(3,-1)

\psline(3,-1)(4,-2)
\psline(3,-2)(4,-1)

\psline(4,-1)(5,-2)
\psline(4,-2)(5,-1)

\psline(5,-1)(6,-2)
\psline(5,-2)(6,-1)

\psline(-1,-1)(0,-2)
\psline(-1,-2)(0,-1)

\psline(-2,-1)(-1,-2)
\psline(-2,-2)(-1,-1)

\psline(-3,-1)(-2,-2)
\psline(-3,-2)(-2,-1)

\psline(-4,-1)(-3,-2)
\psline(-4,-2)(-3,-1)

\psline(-5,-1)(-4,-2)
\psline(-5,-2)(-4,-1)

\psline(-6,-1)(-5,-2)
\psline(-6,-2)(-5,-1)

\psline(0,-2)(1,-3) 
\psline(0,-3)(1,-2) 

\psline(1,-2)(2,-3) 
\psline(1,-3)(2,-2) 

\psline(2,-2)(3,-3) 
\psline(2,-3)(3,-2) 
 
\psline(3,-2)(4,-3) 
\psline(3,-3)(4,-2) 

\psline(4,-2)(5,-3) 
\psline(4,-3)(5,-2) 

\psline(5,-2)(6,-3) 
\psline(5,-3)(6,-2) 

\psline(-1,-2)(0,-3) 
\psline(-1,-3)(0,-2) 

\psline(-2,-2)(-1,-3) 
\psline(-2,-3)(-1,-2) 

\psline(-3,-2)(-2,-3) 
\psline(-3,-3)(-2,-2) 

\psline(-4,-2)(-3,-3) 
\psline(-4,-3)(-3,-2) 

\psline(-5,-2)(-4,-3) 
\psline(-5,-3)(-4,-2) 

\psline(-6,-2)(-5,-3) 
\psline(-6,-3)(-5,-2)

\psline(0,-3)( 1,-4) 
\psline(0,-4)( 1,-3) 

\psline(1,-3)( 2,-4) 
\psline(1,-4)( 2,-3) 

\psline(2,-3)( 3,-4) 
\psline(2,-4)( 3,-3) 
 
\psline(3,-3)( 4,-4) 
\psline(3,-4)( 4,-3) 

\psline(4,-3)( 5,-4) 
\psline(4,-4)( 5,-3) 

\psline(5,-3)( 6,-4) 
\psline(5,-4)( 6,-3) 

\psline(-1,-3)( 0,-4) 
\psline(-1,-4)( 0,-3) 

\psline(-2,-3)( -1,-4) 
\psline(-2,-4)( -1,-3) 

\psline(-3,-3)( -2,-4) 
\psline(-3,-4)( -2,-3) 

\psline(-4,-3)( -3,-4) 
\psline(-4,-4)( -3,-3) 

\psline(-5,-3)( -4,-4) 
\psline(-5,-4)( -4,-3) 

\psline(-6,-3)( -5,-4) 
\psline(-6,-4)( -5,-3)

\psline(0,-4)(1,-5)
\psline(0,-5)(1,-4)

\psline(1,-4)(2,-5)
\psline(1,-5)(2,-4)

\psline(2,-4)(3,-5)
\psline(2,-5)(3,-4)

\psline(3,-4)(4,-5)
\psline(3,-5)(4,-4)

\psline(4,-4)(5,-5)
\psline(4,-5)(5,-4)

\psline(5,-4)(6,-5)
\psline(5,-5)(6,-4)

\psline(-1,-4)(0,-5)
\psline(-1,-5)(0,-4)

\psline(-2,-4)(-1,-5)
\psline(-2,-5)(-1,-4)

\psline(-3,-4)(-2,-5)
\psline(-3,-5)(-2,-4)

\psline(-4,-4)(-3,-5)
\psline(-4,-5)(-3,-4)

\psline(-5,-4)(-4,-5)
\psline(-5,-5)(-4,-4)

\psline(-6,-4)(-5,-5)
\psline(-6,-5)(-5,-4)

\psline(0,-5)(1,-6)
\psline(0,-6)(1,-5)

\psline(1,-5)(2,-6)
\psline(1,-6)(2,-5)

\psline(2,-5)(3,-6)
\psline(2,-6)(3,-5)

\psline(3,-5)(4,-6)
\psline(3,-6)(4,-5)

\psline(4,-5)(5,-6)
\psline(4,-6)(5,-5)

\psline(5,-5)(6,-6)
\psline(5,-6)(6,-5)

\psline(-1,-5)(0,-6)
\psline(-1,-6)(0,-5)

\psline(-2,-5)(-1,-6)
\psline(-2,-6)(-1,-5)

\psline(-3,-5)(-2,-6)
\psline(-3,-6)(-2,-5)

\psline(-4,-5)(-3,-6)
\psline(-4,-6)(-3,-5)

\psline(-5,-5)(-4,-6)
\psline(-5,-6)(-4,-5)

\psline(-6,-5)(-5,-6)
\psline(-6,-6)(-5,-5)

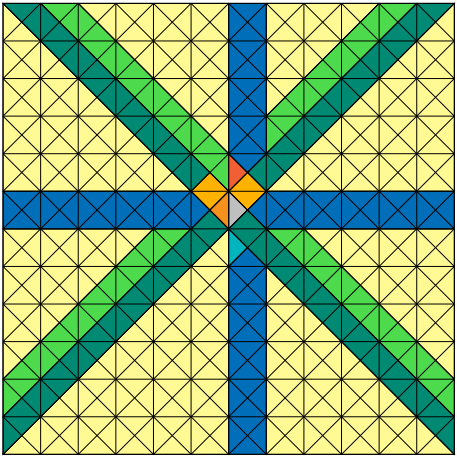
\rput(0,-6.7){Partition of $W$ into cells for $L\in C_{1}$.}

\end{pspicture}
\end{center}
\end{textblock}

\begin{textblock}{10}(9.55,6)
\psset{unit=.5cm}
\begin{center}
\begin{pspicture}(-6,-6)(6,6)

\pspolygon[fillstyle=solid,fillcolor=lightgray](0,0)(0,1)(.5,.5)

\pspolygon[fillstyle=solid,fillcolor=red!5!yellow!42!](0,0)(0,-6)(-6,-6)
\pspolygon[fillstyle=solid,fillcolor=red!5!yellow!42!](1,-1)(1,-6)(6,-6)
\pspolygon[fillstyle=solid,fillcolor=red!5!yellow!42!](2,0)(6,0)(6,-4)
\pspolygon[fillstyle=solid,fillcolor=red!5!yellow!42!](1,1)(6,6)(6,1)
\pspolygon[fillstyle=solid,fillcolor=red!5!yellow!42!](1,3)(1,6)(4,6)
\pspolygon[fillstyle=solid,fillcolor=red!5!yellow!42!](0,2)(0,6)(-4,6)
\pspolygon[fillstyle=solid,fillcolor=red!5!yellow!42!](-1,1)(-6,1)(-6,6)
\pspolygon[fillstyle=solid,fillcolor=red!5!yellow!42!](-2,0)(-6,0)(-6,-4)

\pspolygon[fillstyle=solid,fillcolor=green!80!black!70!](0.5,0.5)(1,1)(6,-4)(6,-5)
\pspolygon[fillstyle=solid,fillcolor=green!80!black!70!](-0.5,0.5)(-1,1)(-6,-4)(-6,-5)
\pspolygon[fillstyle=solid,fillcolor=green!80!black!70!](0.5,1.5)(5,6)(4,6)(0,2)
\pspolygon[fillstyle=solid,fillcolor=green!80!black!70!](-0.5,1.5)(-5,6)(-4,6)(-0,2)

\pspolygon[fillstyle=solid,fillcolor=PineGreen](0,0)(.5,.5)(6,-5)(6,-6)
\pspolygon[fillstyle=solid,fillcolor=PineGreen](0,0)(-.5,.5)(-6,-5)(-6,-6)
\pspolygon[fillstyle=solid,fillcolor=PineGreen](1,1)(6,6)(5,6)(.5,1.5)
\pspolygon[fillstyle=solid,fillcolor=PineGreen](-1,1)(-6,6)(-5,6)(-.5,1.5)%


\pspolygon[fillstyle=solid,fillcolor=orange](0,1)(-.5,1.5)(-1,1)(-.5,.5)
\pspolygon[fillstyle=solid,fillcolor=orange](0,1)(.5,1.5)(1,1)(.5,.5)


\pspolygon[fillstyle=solid,fillcolor=lightgray](0,0)(0,1)(-.5,.5)


\pspolygon[fillstyle=solid,fillcolor=RedOrange](0,1)(.5,1.5)(0,2)

\pspolygon[fillstyle=solid,fillcolor=RedOrange](0,1)(-.5,1.5)(0,2)

\pspolygon[fillstyle=solid,fillcolor=red!5!yellow!42!](0,0)(0,-6)(1,-6)(1,-1)
\pspolygon[fillstyle=solid,fillcolor=red!5!yellow!42!](1,1)(6,1)(6,0)(2,0)
\pspolygon[fillstyle=solid,fillcolor=red!5!yellow!42!](-1,1)(-6,1)(-6,0)(-2,0)
\pspolygon[fillstyle=solid,fillcolor=red!5!yellow!42!](0,2)(0,6)(1,6)(1,3)

\psline(-6,-6)(-6,6)
\psline(-5,-6)(-5,6)
\psline(-4,-6)(-4,6)
\psline(-3,-6)(-3,6)
\psline(-2,-6)(-2,6)
\psline(-1,-6)(-1,6)
\psline(0,-6)(0,6)
\psline(1,-6)(1,6)
\psline(2,-6)(2,6)
\psline(3,-6)(3,6)
\psline(4,-6)(4,6)
\psline(5,-6)(5,6)
\psline(6,-6)(6,6)

\psline(-6,-6)(6,-6)
\psline(-6,-5)(6,-5)
\psline(-6,-4)(6,-4)
\psline(-6,-3)(6,-3)
\psline(-6,-2)(6,-2)
\psline(-6,-1)(6,-1)
\psline(-6,0)(6,0)
\psline(-6,1)(6,1)
\psline(-6,2)(6,2)
\psline(-6,3)(6,3)
\psline(-6,4)(6,4)
\psline(-6,5)(6,5)
\psline(-6,6)(6,6)

\psline(0,0)(1,1)
\psline(0,1)(1,0)

\psline(1,0)(2,1)
\psline(1,1)(2,0)

\psline(2,0)(3,1)
\psline(2,1)(3,0)

\psline(3,0)(4,1)
\psline(3,1)(4,0)

\psline(4,0)(5,1)
\psline(4,1)(5,0)

\psline(5,0)(6,1)
\psline(5,1)(6,0)

\psline(-1,0)(0,1)
\psline(-1,1)(0,0)

\psline(-2,0)(-1,1)
\psline(-2,1)(-1,0)

\psline(-3,0)(-2,1)
\psline(-3,1)(-2,0)

\psline(-4,0)(-3,1)
\psline(-4,1)(-3,0)

\psline(-5,0)(-4,1)
\psline(-5,1)(-4,0)

\psline(-6,0)(-5,1)
\psline(-6,1)(-5,0)

\psline(0,1)(1,2)
\psline(0,2)(1,1)

\psline(1,1)(2,2)
\psline(1,2)(2,1)

\psline(2,1)(3,2)
\psline(2,2)(3,1)

\psline(3,1)(4,2)
\psline(3,2)(4,1)

\psline(4,1)(5,2)
\psline(4,2)(5,1)

\psline(5,1)(6,2)
\psline(5,2)(6,1)

\psline(-1,1)(0,2)
\psline(-1,2)(0,1)

\psline(-2,1)(-1,2)
\psline(-2,2)(-1,1)

\psline(-3,1)(-2,2)
\psline(-3,2)(-2,1)

\psline(-4,1)(-3,2)
\psline(-4,2)(-3,1)

\psline(-5,1)(-4,2)
\psline(-5,2)(-4,1)

\psline(-6,1)(-5,2)
\psline(-6,2)(-5,1)

\psline(0,2)(1,3) 
\psline(0,3)(1,2) 

\psline(1,2)(2,3) 
\psline(1,3)(2,2) 

\psline(2,2)(3,3) 
\psline(2,3)(3,2) 
 
\psline(3,2)(4,3) 
\psline(3,3)(4,2) 

\psline(4,2)(5,3) 
\psline(4,3)(5,2) 

\psline(5,2)(6,3) 
\psline(5,3)(6,2) 

\psline(-1,2)(0,3) 
\psline(-1,3)(0,2) 

\psline(-2,2)(-1,3) 
\psline(-2,3)(-1,2) 

\psline(-3,2)(-2,3) 
\psline(-3,3)(-2,2) 

\psline(-4,2)(-3,3) 
\psline(-4,3)(-3,2) 

\psline(-5,2)(-4,3) 
\psline(-5,3)(-4,2) 

\psline(-6,2)(-5,3) 
\psline(-6,3)(-5,2)

\psline(0,3)( 1,4) 
\psline(0,4)( 1,3) 

\psline(1,3)( 2,4) 
\psline(1,4)( 2,3) 

\psline(2,3)( 3,4) 
\psline(2,4)( 3,3) 
 
\psline(3,3)( 4,4) 
\psline(3,4)( 4,3) 

\psline(4,3)( 5,4) 
\psline(4,4)( 5,3) 

\psline(5,3)( 6,4) 
\psline(5,4)( 6,3) 

\psline(-1,3)( 0,4) 
\psline(-1,4)( 0,3) 

\psline(-2,3)( -1,4) 
\psline(-2,4)( -1,3) 

\psline(-3,3)( -2,4) 
\psline(-3,4)( -2,3) 

\psline(-4,3)( -3,4) 
\psline(-4,4)( -3,3) 

\psline(-5,3)( -4,4) 
\psline(-5,4)( -4,3) 

\psline(-6,3)( -5,4) 
\psline(-6,4)( -5,3)

\psline(0,4)(1,5)
\psline(0,5)(1,4)

\psline(1,4)(2,5)
\psline(1,5)(2,4)

\psline(2,4)(3,5)
\psline(2,5)(3,4)

\psline(3,4)(4,5)
\psline(3,5)(4,4)

\psline(4,4)(5,5)
\psline(4,5)(5,4)

\psline(5,4)(6,5)
\psline(5,5)(6,4)

\psline(-1,4)(0,5)
\psline(-1,5)(0,4)

\psline(-2,4)(-1,5)
\psline(-2,5)(-1,4)

\psline(-3,4)(-2,5)
\psline(-3,5)(-2,4)

\psline(-4,4)(-3,5)
\psline(-4,5)(-3,4)

\psline(-5,4)(-4,5)
\psline(-5,5)(-4,4)

\psline(-6,4)(-5,5)
\psline(-6,5)(-5,4)

\psline(0,5)(1,6)
\psline(0,6)(1,5)

\psline(1,5)(2,6)
\psline(1,6)(2,5)

\psline(2,5)(3,6)
\psline(2,6)(3,5)

\psline(3,5)(4,6)
\psline(3,6)(4,5)

\psline(4,5)(5,6)
\psline(4,6)(5,5)

\psline(5,5)(6,6)
\psline(5,6)(6,5)

\psline(-1,5)(0,6)
\psline(-1,6)(0,5)

\psline(-2,5)(-1,6)
\psline(-2,6)(-1,5)

\psline(-3,5)(-2,6)
\psline(-3,6)(-2,5)

\psline(-4,5)(-3,6)
\psline(-4,6)(-3,5)

\psline(-5,5)(-4,6)
\psline(-5,6)(-4,5)

\psline(-6,5)(-5,6)
\psline(-6,6)(-5,5)

\psline(0,0)(1,-1)
\psline(0,-1)(1,0)

\psline(1,0)(2,-1)
\psline(1,-1)(2,0)

\psline(2,0)(3,-1)
\psline(2,-1)(3,0)

\psline(3,0)(4,-1)
\psline(3,-1)(4,0)

\psline(4,0)(5,-1)
\psline(4,-1)(5,0)

\psline(5,0)(6,-1)
\psline(5,-1)(6,0)

\psline(-1,0)(0,-1)
\psline(-1,-1)(0,0)

\psline(-2,0)(-1,-1)
\psline(-2,-1)(-1,0)

\psline(-3,0)(-2,-1)
\psline(-3,-1)(-2,0)

\psline(-4,0)(-3,-1)
\psline(-4,-1)(-3,0)

\psline(-5,0)(-4,-1)
\psline(-5,-1)(-4,0)

\psline(-6,0)(-5,-1)
\psline(-6,-1)(-5,0)

\psline(0,-1)(1,-2)
\psline(0,-2)(1,-1)

\psline(1,-1)(2,-2)
\psline(1,-2)(2,-1)

\psline(2,-1)(3,-2)
\psline(2,-2)(3,-1)

\psline(3,-1)(4,-2)
\psline(3,-2)(4,-1)

\psline(4,-1)(5,-2)
\psline(4,-2)(5,-1)

\psline(5,-1)(6,-2)
\psline(5,-2)(6,-1)

\psline(-1,-1)(0,-2)
\psline(-1,-2)(0,-1)

\psline(-2,-1)(-1,-2)
\psline(-2,-2)(-1,-1)

\psline(-3,-1)(-2,-2)
\psline(-3,-2)(-2,-1)

\psline(-4,-1)(-3,-2)
\psline(-4,-2)(-3,-1)

\psline(-5,-1)(-4,-2)
\psline(-5,-2)(-4,-1)

\psline(-6,-1)(-5,-2)
\psline(-6,-2)(-5,-1)

\psline(0,-2)(1,-3) 
\psline(0,-3)(1,-2) 

\psline(1,-2)(2,-3) 
\psline(1,-3)(2,-2) 

\psline(2,-2)(3,-3) 
\psline(2,-3)(3,-2) 
 
\psline(3,-2)(4,-3) 
\psline(3,-3)(4,-2) 

\psline(4,-2)(5,-3) 
\psline(4,-3)(5,-2) 

\psline(5,-2)(6,-3) 
\psline(5,-3)(6,-2) 

\psline(-1,-2)(0,-3) 
\psline(-1,-3)(0,-2) 

\psline(-2,-2)(-1,-3) 
\psline(-2,-3)(-1,-2) 

\psline(-3,-2)(-2,-3) 
\psline(-3,-3)(-2,-2) 

\psline(-4,-2)(-3,-3) 
\psline(-4,-3)(-3,-2) 

\psline(-5,-2)(-4,-3) 
\psline(-5,-3)(-4,-2) 

\psline(-6,-2)(-5,-3) 
\psline(-6,-3)(-5,-2)

\psline(0,-3)( 1,-4) 
\psline(0,-4)( 1,-3) 

\psline(1,-3)( 2,-4) 
\psline(1,-4)( 2,-3) 

\psline(2,-3)( 3,-4) 
\psline(2,-4)( 3,-3) 
 
\psline(3,-3)( 4,-4) 
\psline(3,-4)( 4,-3) 

\psline(4,-3)( 5,-4) 
\psline(4,-4)( 5,-3) 

\psline(5,-3)( 6,-4) 
\psline(5,-4)( 6,-3) 

\psline(-1,-3)( 0,-4) 
\psline(-1,-4)( 0,-3) 

\psline(-2,-3)( -1,-4) 
\psline(-2,-4)( -1,-3) 

\psline(-3,-3)( -2,-4) 
\psline(-3,-4)( -2,-3) 

\psline(-4,-3)( -3,-4) 
\psline(-4,-4)( -3,-3) 

\psline(-5,-3)( -4,-4) 
\psline(-5,-4)( -4,-3) 

\psline(-6,-3)( -5,-4) 
\psline(-6,-4)( -5,-3)

\psline(0,-4)(1,-5)
\psline(0,-5)(1,-4)

\psline(1,-4)(2,-5)
\psline(1,-5)(2,-4)

\psline(2,-4)(3,-5)
\psline(2,-5)(3,-4)

\psline(3,-4)(4,-5)
\psline(3,-5)(4,-4)

\psline(4,-4)(5,-5)
\psline(4,-5)(5,-4)

\psline(5,-4)(6,-5)
\psline(5,-5)(6,-4)

\psline(-1,-4)(0,-5)
\psline(-1,-5)(0,-4)

\psline(-2,-4)(-1,-5)
\psline(-2,-5)(-1,-4)

\psline(-3,-4)(-2,-5)
\psline(-3,-5)(-2,-4)

\psline(-4,-4)(-3,-5)
\psline(-4,-5)(-3,-4)

\psline(-5,-4)(-4,-5)
\psline(-5,-5)(-4,-4)

\psline(-6,-4)(-5,-5)
\psline(-6,-5)(-5,-4)

\psline(0,-5)(1,-6)
\psline(0,-6)(1,-5)

\psline(1,-5)(2,-6)
\psline(1,-6)(2,-5)

\psline(2,-5)(3,-6)
\psline(2,-6)(3,-5)

\psline(3,-5)(4,-6)
\psline(3,-6)(4,-5)

\psline(4,-5)(5,-6)
\psline(4,-6)(5,-5)

\psline(5,-5)(6,-6)
\psline(5,-6)(6,-5)

\psline(-1,-5)(0,-6)
\psline(-1,-6)(0,-5)

\psline(-2,-5)(-1,-6)
\psline(-2,-6)(-1,-5)

\psline(-3,-5)(-2,-6)
\psline(-3,-6)(-2,-5)

\psline(-4,-5)(-3,-6)
\psline(-4,-6)(-3,-5)

\psline(-5,-5)(-4,-6)
\psline(-5,-6)(-4,-5)

\psline(-6,-5)(-5,-6)
\psline(-6,-6)(-5,-5)

\rput(0,-6.7){Partition of $W$ into cells for $L'$. }

\end{pspicture}
\end{center}
\end{textblock}
\end{Rem}


$\ $\\
\vspace{6cm}

\section{Proof of Theorem \ref{main2} in the generic setting}
\subsection{Hypothesis and notation}
\label{generic}
Let $(W,S)$ be an irreducible affine Weyl group generated by $S$.  Let $S=S^{\circ}\cup S^{+}$ be a partition of $S$ such that no element of $S^{\circ}$ is conjugate to an element of $S^{+}$ and $S^{\circ},S^{+}\neq S$. For a subset $I$ of $S$ we set $I^{\circ}=I\cap S^{\circ}$ and $I^{+}=I\cap S^{+}$. 
We denote by $\bar{S}$ the set of conjugacy classes in $S$ in $W$ and we set
$$\bar{S}^{+}=\{\o\in \bar{S}|\o\subset S^{+}\}\text{ and }\bar{S}^{\circ}=\{\o\in \bar{S}|\o\subset S^{\circ}\}$$
As in the previous section, $\nZ[\bar{S}]$ denotes the free $\nZ$-module with basis $\bar{S}$ and $V'=\nR\otimes_{\nZ} \nZ[\bar{S}]$. 
We identify $\nZ[\bar{S}]$ with $\nZ^{|S|}$.\\

A subset $X$ of $\nZ[\bar{S}]$ is called positive if the following three conditions hold
\begin{enumerate}
\item $\nZ[\bar{S}]=X\cup (-X)$;
\item $X+X\subset X$;
\item $X\cap (-X)$ is a subgroup of $\nZ[\bar{S}]$.
\end{enumerate}
Any positive subset $X$ defines a total order $\leq_{X}$ on $\bGa:=\nZ[\bar{S}]/(X\cap (-X))$ simply by setting
$$\ga\geq_{X} 0\Longleftrightarrow  \text{all the representatives of $\g$ belong to $X$}.$$

\bigskip

We briefly explain how to classify all the positive subsets of $\nZ[\bar{S}]$. Let $\cP(\nZ[\bar{S}])$ be the set of all sequences $(\varphi_{1},\ldots, \varphi_{d})$ such that $\varphi_{i}$ is a non-zero linear form defined on  $\ker(\varphi_{i-1})\subset V'$, with the convention that $\varphi_{0}=0$.
Then we can associate to $\Phi=(\varphi_{1},\ldots, \varphi_{d})\in \cP(\nZ[\bar{S}])$ a positive subset $\text{Pos}(\Phi)$ of $\nZ[\bar{S}]$ by setting 
$$\text{Pos}(\Phi)=\{\ga\in \nZ[\bar{S}]\mid \exists\ 0\leq k\leq d-1; \ga\in\ker\va_{k}\text{ and } \va_{k+1}(\ga)>0\}\cup \ker\va_{d}.$$
It can be shown that all positive subsets can be obtained this way. We denote by $\cP_{+}(\nZ[\bar{S}])$ the subset of $\cP(\nZ[\bar{S}])$ which consists of all sequences $\Phi=(\va_{1},\ldots,\va_{d})$ such that 
\begin{equation*}
\va_{1}=\sum_{\o\in\bar{S}^{+}} \o^{\ast} \quad\text{ and }\quad \va_{k}(\o)>0 \text{ for all $\o\in \bar{S}$}.
\end{equation*}
\bigskip 

\begin{quotation}{\bf Hypothesis. }{\it
From now on and until the end of this section, we fix a positive subset $X=\text{Pos}(\Phi)$ such that $\Phi\in  \cP_{+}(\nZ[\bar{S}])$. In type $\tilde{\Crm}$ we  assume that $\bt\geq \bt'$. }
\end{quotation}

\bigskip
\noindent
To simplify the notation, we will denote by $\leq$ the total order on $\bGa$ instead of~$\leq_{X}$. 
\begin{Exa}
\label{lex}
Assume that $W$ is of type $\tilde{\Brm}_r$, $\tilde{\Frm}_4$ or $\tilde{\Grm}_2$ and let $\bar{S}^{+}=\{\bs\}$ and  $\bar{S}^{\circ}=\{\bt\}$. Let $X\in  \cP_{+}(\nZ[\bar{S}])$. Then we have $\va_{1}=\bs^{\ast}$ and $\ker(\va_{1})=\nR \bt$. Since we assumed that $\va_{2}(\bt)>0$ we must have $\va_{2}=\kappa\bt^{\ast}$ where $\kappa>0$. It follows that $\bGa=\nZ[\bar{S}]$ and that the order on $\bGa$ is simply the lexicographic order:
$$(i,j)<(i',j')\Longleftrightarrow i<i' \text{ or }(i=i'\text{ and } j<j').$$
Assume that $W$ is of type $\tilde{\Crm}_{r}$ and that $S^{+}=\bt$. Let $\va_{1}=\bt^{\ast}$ and $\va_{2}$ be defined by $\va_{2}(0,j,k)=bj+ck$ for $b,c\in\nN$. Then we have $\ker(\va_{2})=\sg (0,-c,b)\sd$. Finally we define $\va_{3}$ by $\va_{3}(0,-c,b)=1$ and we extend it by linearity. 
Then the order associated to $(\va_{1},\va_{2},\va_{3})$ can be describe as follows:
{\small $$\{(i,j,k)\mid (i,j,k)>0\}=\{(i,j,k)\mid i>0\}\cup \{(0,j,k)\mid bj+ck>0\}\cup \{(0,-kc,kb)\mid k>0\}~\SS{\blacksquare}$$}
\end{Exa}

\bigskip
\noindent
Let $\bL: W\longrightarrow \bGa$ be the weight function defined by $\bL(s)=\o_{i}$ if $s\in \o_{i}$. Let $\bA$ be the group algebra of $\bGa$ over $\nZ$. Recall that we use the exponential notation for  $\bA$
$$\bA=\bigoplus_{\ga\in\bGa} \nZ v^{\ga} \text{ where } v^{\ga}.v^{\ga'}=v^{\ga+\ga'}.$$
Let $\bH=\cH(W,S,\bL)$ be the associated Hecke algebra. We will denote by  $\bT_{x}$ the element of the standard bases of $\bH$, by $\bC_{x}$ the elements of the Kazhdan-Lusztig basis of $\bH$ and by $\bP_{x,y},\bM_{x,y}$ the polynomials in $\bA$ defined in Section \ref{Iwahori-Hecke}. 
We set
$$
\renewcommand{\arraystretch}{1.5}
\begin{array}{cllcllcccccc}
\bA_{< 0}&=\underset{\g< 0}{\bigoplus} \nZ v^{\g} &,&\bA_{\geq 0}&=\underset{\g\geq  0}{\bigoplus} \nZ v^{\g}
\end{array}
$$
and 
$$\bH_{< 0}=\underset{w\in W}{\bigoplus} \bA_{<0}\bT_{w}.$$

We denote by $^{+}:\bGa\rightarrow \bGa$ (respectively $^{\circ}$) the map induced by the projection of $\nZ[\bar{S}]$ onto $\oplus_{\o\in\bar{S}^{+}} \nZ \o$ (respectively $\oplus_{\o\in\bar{S}^{\circ}}\nZ \o$).
For $a=\sum_{\g\in\bGa} a_{\g}v^{\g}\in \bA$ we define 
$$\deg(a)=\max\{\g\in\bGa\mid a_{\g}\neq 0\}.$$
We will write $\deg^{+}(a)$ instead of $\deg(a)^{+}$. 
\begin{Rem}
\label{neg}
Note that if $a=\sum a_{\ga}v^{\ga}$ satisfies $\va_{1}(\deg^{+}(a))<0$, then $a\in \bA_{<0}$. Indeed, if $a_{\ga}\neq 0$, then $\ga=\ga^{+}+\ga^{\circ}$ where $\ga^{+}\leq \deg^{+}(a)$. Applying $\va_{1}$ yields $\va_{1}(\ga)=\va_{1}(\ga^{+})+0\leq \va_{1}(\deg^{+}(a))<0$ that is $\ga<0$ as required.\finl 
\end{Rem}

\medskip

In this section we will have to distinguish the following cases. 
(We keep the notation of Section \ref{typeBFG} and \ref{typeC}.) \\

{\bf Case 1.} $W$ is of type $\tilde{\Brm}_r$, $\tilde{\Frm}_4$ or $\tilde{\Grm}_2$. \\

{\bf Case 2.} $W$ is of type $\tilde{\Crm}_r$, $\bar{S}^{+}\neq \{\bt,\bt'\}$. \\

{\bf Case 3.} $W$ is of type $\tilde{\Crm}_r$, 
$\bar{S}^{+}=\{\bt,\bt'\}$ and $(-1,k,1)<0$ for all $k>0$. \\

{\bf Case 4.}  $W$ is of type $\tilde{\Crm}_r$, 
$\bar{S}^{+}=\{\bt,\bt'\}$ and $(-1,k,1)>0$ for some $k>0$. \\

\noindent
If we are in Case (1)--(3), we define the weight function $\bL^{+}$ by 
$$\text{$\bL^{+}(\o)=\bL(\o)$ if $\o\in \bar{S}^{+}$ and $\bL^{+}(\o)=0$ if $\o\in \bar{S}^{\circ}.$}$$ 
Hence we have $\bL^{+}(w)=(\bL(w))^{+}$.\\

\noindent
If we are in Case (4), we define the weight function $\bL^{+}$ by 
$$\text{$\bL^{+}(t)=\bL^{+}(t')=\bt$  and $\bL^{+}(s)=0\text{ if $s\in \bs$}$}.$$
Note that in this case, we have  $\bL^{+}(w)\neq (\bL(w))^{+}$.\\

\noindent
Recall that, for a weight function $L$,  we say that a hyperplane $H$ of direction $\al$ is of maximal $L$-weight if $L_{H}=L_{\al}$ where $L_{\al}=\max_{n\in\nZ} H_{\al,n}$. 
Let $H,H'\in \cF$. Then we have either \cite[Lemma 2.2]{Bremke} (a) $L_{H}=L_{H'}$ or (b) $W$ is of type $\tilde{\Crm}_r$, $H$ contains a face of type $t_{1}$ and $H'$ a face of type $t_{2}$ and $\{t_{1},t_{2}\}=\{t,t'\}$ and  $L(t)\neq L(t')$.

\bigskip
\noindent
\begin{Lem}
\label{Lplus-spe}
Let $\la$ be a $\bL^{+}$-special point and let $H$ be an hyperplane orthogonal to $\a$ which contains $\la$ and such that $\bL^{+}_{H}>0$. Then in Case 1--3, $\bL_{H}$ is of maximal $\bL$-weight. In Case 4, we may have $\bL_{H}=\bt'<\bt=\bL_{\al}$.
\end{Lem}
\begin{proof}
Let $\la$ be a $\bL^{+}$-special point and let $H$ be an hyperplane which contains $\la$ and such that $\bL^{+}_{H}>0$. In Case 1, the result is clear since any hyperplane is of maximal weight. In Case 2, since $\bt>\bt'$ and $\bar{S}^{+}\neq \{\bt,\bt'\}$, we must have 
$\bt'\subset \bar{S}^{\circ}$. If $\bt\in\bar{S}^{\circ}$, then since $\bL^{+}_{H}>0$ we have $L_{H}=\bs$ and $H$ is of maximal weight.  If $\bt'\in\bar{S}^{+}$, then since $\la$ is a $\bL^{+}$-special we have $S_{\la}=\{t,s_{1},\ldots,s_{n-1}\}$ and we have $L_{H}=\bt$ or $L_{H}=\bs$ and the result follows. In Case 3, since $\bt>\bt'$, we must have $S_{\la}=\{t,s_{1},\ldots,s_{n-1}\}$ and the result follows as above. Finally in Case 4, the result is clear. 
\end{proof}
\begin{Rem}
\label{diff-case}
In Case 3, if $\deg^{+}(a)\leq (-1,0,1)$ then $a\in \bA_{<0}$. Indeed, we have $\deg(a)=\deg^{+}(a)+\ga$ where $\ga\in \ker(\va_{1})$. But then $\ga=(0,k,0)$ for some $k\in\nZ$. Therefore $\deg(a)=(-1,k,1)<0$.\finl
\end{Rem}


\subsection{Generalized induction of Kazhdan-Lusztig cells}
Recall that every element $w\in\cb^{\bL^{+}}_{\min}$ can be uniquely written under the form $x_{w}a_{w}\hw_{\la_{\si}}b_{\si}$ where $\si\in\Om^{\bL^{+}}_{0}$, $a_{w}\in W_{S_{\la_{\si}}^{\circ}}$ and $x_{w}\in X_{\la_{\si}}$. For $\si\in\Om_{0}^{\bL^{+}}$, we set 
\begin{align*}
N^{\bL^{+}}_{\si}&:=\{w\in W\mid w=xaw_{\la_{\si}}b_{\si}, x\in X_{\la_{\si}}, a\in W_{S^{\circ}_{\la_{\si}}}\},\\
U_{\si}&:=\{w\in W\mid w=aw_{\la_{\si}}b_{\si}, a\in W_{S^{\circ}_{\la_{\si}}}\}\\
\end{align*}
and
\begin{align*}
N^{\leq}_{\si}&:=\bigcup_{\si'\in\Om^{L}_{0},b_{\si'}\leq b_{\si}} N^{\bL^{+}}_{\si'},\\
U^{\leq}_{\si}&:=\bigcup_{\si'\in\Om^{L}_{0},b_{\si'}\leq b_{\si}} U_{\si'}.
\end{align*}
For $u=a\hw_{\la_{\si}}y_{\si}\in U_{\si}$, we set $X_{u}:= X_{\la_{\si}}$. 

\medskip

\begin{Th}
\label{I1-I5}
For all $\si\in\Om^{\bL^{+}}_{0}$, the set $U^{\leq}_{\si}$ together with the collection of subsets $\{X_{u}\ |\ u\in U^{\leq}_{\si}\}$ satisfy the following condition
\begin{enumerate}
\item[{\bf I1}.] for all $u\in U^{\leq}_{\si}$, we have $e\in X_{u}$,
\item[{\bf I2}.] for all $u\in U^{\leq}_{\si}$ and $x\in X_{u}$ we have $\ell(xu)=\ell(x)+\ell(u)$,
\item[{\bf I3}.] for all $u,v\in U^{\leq}_{\si}$ such that $u\neq v$ we have $X_{u}u\cap X_{v}v=\emptyset$,
\item[{\bf I4}.] the submodule $\cM:=\sg \bT_{x}\bC_{u}|\ u\in U^{\leq}_{\si},\ x\in X_{u}\sd_{\cA}\subseteq \cH$ is a left ideal.
\item[{\bf I5}.] for all $v\in U^{\leq}_{\si}$ and $y\in X_{v}$ we have
$$\bT_{y}\bC_{v}\equiv \bT_{yv}+\underset{u\in U^{\leq}_{\si}}{\sum_{\ell(xu)< \ell(yv)}} a_{xu,yv}\bT_{xu} \mod \bH_{<0} $$
\end{enumerate}
\end{Th}
Assuming that this theorem holds, we would get, using the Generalised Induction Theorem \cite[Theorem 6.3]{jeju4}, that the set
$$N_{\si}^{\leq}=\{xu\mid u\in U^{\leq}_{\si}, x\in X_{u}\}$$
is a left ideal of $(W,\bL)$ (i.e. $y\leq_{\cL} w\in N_{\si}^{\leq}$ implies $y\in N_{\si}^{\leq}$) for all $\si\in\Om^{\bL^{+}}_{0}$. In particular it would be a union of left cells. Then by an easy induction on the length of $b_{\si}\in W$, we would get that each $N_{\si}$ is a union of left cells. In turn, since $\cb^{\bL^{+}}_{\min}$ is stable by taking the inverse, this would implie that $\cb^{\bL^{+}}_{\min}$ is indeed a union of two-sided cells of $(W,\bL)$. 

\begin{Rem}
Condition {\bf I5} is stated slightly differently in \cite[\S 6]{jeju4}:
{\it for all $v\in U$, $y\in X_{v}$ we have}
$$T_{y}C_{v}\equiv T_{yv}+\sum_{xu\sqs yv} a_{xu,yv}T_{x}C_{u} \mod \cH_{<0}$$
where $\sqs$ denotes a preorder such that $xu\sqs yv$ implies $\ell(xu)<\ell(yv)$. It is a straightforward induction on $\ell(xu)$ to show that those two conditions are equivalent.\finl 
\end{Rem}


\subsection{Kazhdan-Lusztig Polynomials and $M$-polynomials}

Let $x\in W$ and let $I\subset S$ be such that $W_{I}$ is finite. There exist unique $x'\in W_{I}$ and $d_{x}\in X^{-1}_{I}$ such that $x=x'd_{x}$. Next $x'\in W_{I}$ can be uniquely written as $x'=au$ where $a\in W_{I^{\circ}}$ and $u$ has minimal length in the coset $W_{I^{\circ}}x'$ of $W_{I}$. We will write $x=a^{I}_{x}u^{I}_{x}d^{I}_{x}$ for this decomposition or simply $x=a_{x}u_{x}d_{x}$ if it is clear from the context what the subset $I$ should be. We denote by $w^{\circ}_{I}$ (not to confuse with $w_{I^{\circ}}$) the element of minimal length in the coset $W_{I^{\circ}}w_{I}$. 

\medskip

\begin{Rem}
Note that, for all $a\in W_{I^{\circ}}$ and all $t\in I^{+}$, we must have $taw^{\circ}_{I}<aw^{\circ}_{I}$ since the number of elements of $I^{+}$ appearing in any reduced expression of  $aw^{\circ}_{I}$ is maximal.\finl 
\end{Rem}

\bigskip

\begin{Lem}
\label{KLasym}
Let $I\subset S$ be such that $W_{I}$ is finite and let $y\in W$ be such that $y=a\hw_{I}z$ where $a\in W_{I^{\circ}}$ and $z\in X^{-1}_{I}$. Let $x=a_{x}u_{x}d_{x}<y$.  Then
$$\deg^{+}(\bP_{x,y})\leq \bL(u_{x})^{+}-\bL(w^{\circ}_{I})^{+}$$
Furthermore, for all $s\in S^{\circ}$ such that $sx<x<y<sy$ we have
$$\bM^{s}_{x,y}\neq 0 \Longrightarrow u_{x}=w^{\circ}_{I}.$$
\end{Lem}

\begin{proof}
We prove the result by induction. To this end, to any element $x,y\in W$ satisfying the hypothesis of the lemma, we associate a pair 
$$\cP(x,y):=(\ell(y)-\ell(x), \ell_{0}-\ell(a_{x}))$$
where $\ell_{0}=\ell(w_{I^{\circ}})$ and $x=a_{x}u_{x}d_{x}$.
We order such pairs by the usual lexicographic order. Let $x<y$. If $\bL(u_{x})^{+}=\bL(w^{\circ}_{I})^{+}$ then the result is clear. Thus we may assume that $\bL(u_{x})^{+}<\bL(w^{\circ}_{I})^{+}$.

\medskip
\noindent
First assume that there exists $t\in I^{+}$ such that $tx>x$. Then, since $ty<y$, we have 
$$\bP_{x,y}=v^{-\bL(t)}\bP_{tx,y}$$
and the result follows by induction. 

\medskip
\noindent
Next assume that $tx<x$ for all $t\in I^{+}$.  Since we supposed that $\bL(u_{x})^{+}<\bL(w^{\circ}_{I})^{+}$, there exists  $s\in I_{0}$ such that $sx>x$.
If $sy<y$ then
$$\bP_{x,y}=v^{-\bL(s)}\bP_{sx,y}$$
and the result follows by induction since $\ell_{0}-\ell(sa_{x})=\ell_{0}-\ell(a_{x})-1$. \\
If $sy>y$ then we have
\begin{equation*}
\tag{$\dag$}\bP_{sx,sy}=v^{\bL(s)}\bP_{sx,y}+\bP_{x,y}-\underset{sz<z}{\sum_{sx\leq z<y}}\bP_{sx,z}\bM^{s}_{z,y}.
\end{equation*}
By induction we know that 
$$\deg^{+}(v^{\bL(s)}\bP_{sx,y})\leq \bL(u_{x})^{+}-\bL(w^{\circ}_{I})^{+}.$$
Further if $\bM^{s}_{z,y}\neq 0$ then $\bL(u_{z})^{+}=\bL(w^{\circ}_{I})^{+}$ where  $z=a_{z}u_{z}d_{z}$. We know that $\deg(\bM_{x,y}^{s})< \bL(s)$ (see \cite[\S 6.3]{bible}). Thus if $\bM^{s}_{z,y}\neq 0$, we get using the induction hypothesis 
$$\deg_{+}(\bP_{sx,z}\bM^{s}_{z,y})\leq  \bL(u_{x})^{+}-\bL(w^{\circ}_{I})^{+}.$$ 
Now we have
$$\cP(sx,sy)=(\ell(sy)-\ell(sx), \ell_{0}-\ell(sx_{0}))=(\ell(y)-\ell(x), \ell_{0}-\ell(a_{x})-1)<\cP_{x,y}.$$
Hence by induction 
$$\deg^{+}(\bP_{sx,sy})\leq \bL(u_{sx})^{+}-\bL(w^{\circ}_{I})^{+}=\bL(u_{x})^{+}-\bL(w^{\circ}_{I})^{+}.$$
The result follows using $(\dag)$.
\end{proof}

\begin{Rem}
The same proof can easily be generalised to any Coxeter group $(W,S)$. Assume that $W$ is finite and let $w_{S}$ be the longest element of $W$. Using the previous lemma, we can show that
$$W_{S^{\circ}}\quad\text{and}\quad W_{S^{\circ}}w_{0}$$
are union of cells of $(W,S,\bL)$. Indeed, let $y\in W_{S^{\circ}}w_{S}$ and let $w\in W_{S^{\circ}}$ be such that $y=w\hw_{S}$ and $w\add \hw_{S}$. Let $x\in W$ be such that $x\leq_{\cL} y$. We may assume that there exists $s\in S$ such that $sx<x<y<sy$ and $\bM^{s}_{x,y}\neq 0$. Note that $y<sy$ we implies that $s\in S^{\circ}$. Write $x=a^{S}_{x}u^{S}_{x}d^{S}_{x}$. Since $\bM^{s}_{x,y}\neq 0$ the previous lemma implies that $\bL(u^{S}_{x})=\bL(\hw_{S})$ where  which in turn implies that $x\in W_{S^{\circ}}w_{S}$. Thus we have shown that $W_{S^{\circ}}w_{S}$ is a left ideal of $W$ and thus a union of (left, right and two-sided) cells of $(W,\bL)$. Multiplying by the longest element sends (left, right and two-sided) cells to (left, right and two-sided) cells thus we get that $W_{S^{\circ}}$ is also a union of (left, right, two-sided) cells of $(W,\bL)$. This argument provides an alternative proof of Theorem 1.1 in \cite{jeju3} when $W$ is finite.\finl   
\end{Rem}


\subsection{Multiplication of the standard basis}
\label{standard}
We set
$$\bT_{x}\bT_{y}=\sum_{w} \bff_{x,y,w}\bT_{w}\text{ where $\bff_{x,y,w}\in\bA$.}$$
Following \cite[\S 2.3]{jeju2}, we want to study the degree of the polynomials $\bff_{x,y,w}$. We will need more precise result than in \cite{jeju2}, but the method of the proof is similar. \\

We introduce some notation. For $\al\in\Phi^{+}$, we set $\cF_{\al}=\{H_{\al,n}\mid n\in\nZ\}$. For $x,y\in W$ we set
\begin{align*}
H_{x,y}&=\{H\in \mathcal{F}\mid H\in H(A_{0},yA_{0})\cap H(yA_{0},xyA_{0})\},\\
I_{x,y}&=\{\al\in\Phi^{+}\mid H_{x,y}\cap \cF_{\al}\neq \emptyset\}.
\end{align*}
For $\al\in I_{x,y}$ we set
$$c^{\bL}_{x,y}(\al)=\underset{H\in H_{x,y}\cap \cF_{\al}}{\text{max }} \bL_{H}.$$
Let 
$$c^{\bL}_{x,y}=\underset{\al\in I_{x,y}}{\sum}c^{\bL}_{x,y}(\al).$$
The following two lemmas can be found in \cite{jeju2}. 
\begin{Lem}
\label{Lem1}
Let $x,y\in W$ and $s\in S$ be such that $xs>x$. Then
$$I_{x,sy}\subseteq I_{xs,y}.$$
\end{Lem}

\begin{Lem}
\label{Lem2}
Let $x,y\in W$ and $s\in S$ be such that $xs>x$ and $sy<y$. Let $\al\in\Phi^{+}$ and $n\in\nZ$ be such that $H_{\al,n}$ is the unique hyperplane separating $yA_{0}$ and $syA_{0}$. There is an injective map $\varphi$ from $I_{x,y}$ to $I_{xs,y}-\{ \al \}$. Furthermore if $\beta\in I_{x,y}$ we have either $\va(\beta)=\beta$ or $\va(\beta)=\pm\sigma_{H_{\al,0}}(\beta)$.
\end{Lem}
Using these two lemmas, one can obtain the following bound on the degree of $f_{x,y,z}$ in terms of $x$ and $y$. 
\begin{Th}
\label{bound first}
We have $\deg(\bff_{x,y,z})\leq c^{\bL}_{x,y}$ for all $z\in W$. 
\end{Th}
Note that this implies that $\deg^{+}(\bff_{x,y,z})\leq (c^{\bL}_{x,y})^{+}$. 

Let $x=s_{N}\ldots s_{1}$ be a reduced expression of $x$. We denote $\fI_{x,y}$ the collection of all subsets $I=\{i_{1},\ldots,i_{p}\}$ such that $1\leq i_{1}<\ldots<i_{p}\leq k$ and
$$s_{i_{t}}\ldots \hs_{i_{t-1}}\ldots \hs_{i_{1}}\ldots s_{1}y<\hs_{i_{t}}\ldots \hs_{i_{t-1}}\ldots \hs_{i_{1}}\ldots s_{1}y.$$
For all $I=\{i_{1},\ldots,i_{p}\}\in\fI_{x,y}$ and all $1\leq k\leq p$, we set 
$$x_{k}=s_{N}\ldots s_{i_{p}}\ldots \hat{s}_{i_{k}} \text{ and }  y_{k}= \hs_{i_{k}}\ldots \hs_{i_{k-1}}\ldots \hs_{i_{1}}\ldots s_{1}y$$
and
$$z_{I}=s_{N}\ldots \hs_{i_{p}}\ldots \hs_{i_{1}}\ldots s_{1}y.$$
Then we have \cite[Proof of Proposition 5.1]{Bremke}
$$\bT_{x}\bT_{y}=\sum_{I\in\fI_{x,y}} \left(\prod^{p}_{k=1} (v^{\bL(s_{i_{k}})}-v^{-\bL(s_{i_{k}})})\right) T_{z_{I}}.$$
Hence
$$\deg(\bff_{x,y,z_{I}})=\sum_{k=1}^{p}\bL(s_{i_{k}}).$$
Fix $k$ such that $2\leq k\leq p$. Using the previous lemmas, there exists an injective map $\varphi_{k}$ such that 
$$
\begin{array}{ccccccccccc}
I_{x_{k},y_{k}}=I_{s_{N}\ldots s_{i_{p}}\ldots\hat{s}_{i_{k}},\hs_{i_{k}}\ldots \hs_{i_{k-1}}\ldots \hs_{i_{1}}\ldots s_{1}y}\overset{\va_{k}}{\longrightarrow} 
I_{s_{N}\ldots s_{i_{p}}\ldots s_{i_{k}},\hs_{i_{k}}\ldots \hs_{i_{k-1}}\ldots \hs_{i_{1}}\ldots s_{1}y}\\
\subseteq I_{s_{N}\ldots s_{i_{p}}\ldots s_{i_{k}}(s_{i_{k}-1}\ldots s_{i_{k-1}+1}),\hs_{i_{k-1}}\ldots \hs_{i_{1}}\ldots s_{1}y}=I_{x_{k-1},y_{k-1}}
\end{array}$$
Thus we have a sequence 
$$I_{x_{p},y_{p}}\overset{\va_{p}}{\longrightarrow} I_{x_{p}s_{i_{p}},y_{p}}\subseteq I_{x_{p-1},y_{p-1}}\overset{\va_{p-1}}{\longrightarrow} I_{x_{p-1}s_{i_{p-1}},y_{p-1}}\ldots I_{x_{1},y_{1}}\overset{\va_{1}}{\longrightarrow} I_{x_{1}s_{i_{1}},y_{1}}\subseteq I_{x,y}.$$
If we denote by $\al_{i_{k}}$ the positive root such that the only hyperplane separating $y_{k}A_{k}$ and $s_{i_{k}}y_{k}A_{0}$ lies in $\cF_{\al_{i_{k}}}$ then we have
$$\{\al_{i_{1}},\varphi_{1}(\al_{i_{2}}), \ldots,(\varphi_{1}\ldots \varphi_{p-1})(\al_{i_{p}})\}\subseteq I_{x,y}.$$

\begin{Th}
\label{bound}
Let $w=a_{w}w^{\circ}_{\la_{\si}}b_{\si}\in\cb^{\bL^{+}}_{\min}$. Let $x\in X_{\la_{\si}}$. We have
$$T_{x}C_{w}\equiv T_{xw}+\sum_{z\in N_{\si'}, b_{\si'}<b_{\si}} a_{z} T_{z}\ \mod \bH_{<0}.$$
\end{Th} 
\begin{proof}
We have 
\begin{align*}
T_{x}C_{w}&=T_{x}T_{w}+T_{x}\left(\sum_{y< w} P_{y,w}T_{y} \right)\\
&=T_{xw}+\sum_{y\in\cb^{\bL^{+}}_{\min}}P_{y,w}T_{x}T_{y}+\sum_{y\notin\cb^{\bL^{+}}_{\min}} P_{y,w}T_{x}T_{y}\\
&=T_{xw}+\sum_{a'_{w}<a_{w}} P_{y,w}T_{x}T_{a'_{w}w^{\circ}_{\la_{\si}}d_{\si}}+\underset{b_{\si'}<b_{\si}}{\sum_{y\in N_{\si'}}}P_{y,w}T_{x}T_{y}+\sum_{y\notin\cb^{\bL^{+}}_{\min}} P_{y,w}T_{x}T_{y}\\
&=T_{xw}+\sum_{a'_{w}<a_{w}} P_{y,w}T_{xa'_{w}w^{\circ}_{\la_{\si}}d_{\si}}+\underset{b_{\si'}<b_{\si}}{\sum_{y\in N_{\si'}}}P_{y,w}T_{x}T_{y}+\sum_{y\notin\cb^{\bL^{+}}_{\min}} P_{y,w}T_{x}T_{y}\\
&\equiv T_{xw}+\underset{b_{\si'}<b_{\si}}{\sum_{y\in N_{\si'}}}P_{y,w}T_{x}T_{y}+\sum_{y\notin\cb^{\bL^{+}}_{\min}} P_{y,w}T_{x}T_{y}\quad\mod \bH_{<0}
\end{align*}
Fix a $y$ in the on of the sum above. Let $\la$ be the unique $\bL^{+}$-special point which is contained in the closure of $yA_{0}$ and which lies in the same orbit as $\la_{\si}$ (i.e. $W_{\la}=W_{\la_{\si}}$). Write $y=a_{y}u_{\la}d_{y}$ where $d_{y}\in X^{-1}_{\la}$, $u_{\la}$ has minimal length in the coset $W_{S^{\circ}_{\la}}yd^{-1}_{y}$ of $W_{\la}$ and $a_{y}\in W_{S^{\circ}_{\la}}$. 

\medskip


Let $a_{y}u_{\l}=s_{k}\ldots s_{1}$ be a reduced expression and let $v=s_{n}\ldots s_{k+1}$ be such that $\bL^{+}(s_{n}\ldots s_{1})=\nu_{\bL^{+}}$ and $\ell(s_{n}\ldots s_{1})=n$. 
Let $H_{i}$ be the unique hyperplane which separates $s_{i}\ldots s_{1}d_{y}A_{0}$  and $s_{i+1}\ldots s_{1}d_{y}A_{0}$ and let $\al_{i}\in\Phi^{+}$ be such that $H_{i}\in\cF_{\al_{i}}$.  Let $1\leq i\leq k$ and assume that $H_{i}=H_{\al_{i},n}$ where $n>0$ (the case $n\leq 0$ is similar). We have $H_{i}\in H(A_{0},yA_{0})$ and since $\la\in H_{i}\cap \ov{yA_{0}}$ we see that
$$n<\sg \mu,\check{\al_{i}}\sd<n+1\text{ for all $\mu\in yA_{0}$}.$$
It follows that $H_{\al_{i},m}\notin H(A_{0},yA_{0})$ for all $m\geq n+1$. Next, since $x\in X_{\la_{\si}}=X_{\la}$, we see that $H_{i}\notin H(yA_{0},xyA_{0})$ and it follows that $H_{\al_{i},m}\notin H(yA_{0},xyA_{0})$ for all $m\leq n$. Finally, since $\la$ is a $\bL^{+}$-special point, we must have 
$$\{\al_{i}\mid  1\leq i\leq n\}=\Phi^{+}\cap \Phi^{\bL^{+}}.$$
 It follows that 
$$I_{x,y}\cap \Phi^{\bL^{+}}\subseteq \{\al_{k+1},\ldots,\al_{n}\}.$$
By Lemma \ref{KLasym}, we know that 
\begin{equation*}
 \deg^{+}(\bP_{y,w})\leq \bL(u_{\la})^{+}-\bL(w^{\circ}_{\la})^{+}=-\bL(v)^{+}.
\end{equation*}
Therefore, we have
\begin{equation*}
\tag{$\ast$}\deg^{+}(\bP_{x,y}\bff_{x,y,z})\leq \deg^{+}(\bff_{x,y,z}) -\bL(v)^{+}\leq   (c^{\bL}_{x,y})^{+} -\bL(v)^{+}.
\end{equation*}
If we are in Case 1--3, all the hyperplane which contains $\la$ must be of maximal weight, hence we have $\bL(s_{i})=\bL_{\al_{i}}$ for all $i$. Hence 
\begin{equation*}
\tag{1}(c^{\bL}_{x,y})^{+}\leq \sum^{n}_{i=k+1}\bL(s_{i})=\bL(v)^{+}.
\end{equation*}
In Case 4, we may   have $\bL^{+}(s_{i})=\bt'<\bt$. Let $i,k,i',k'\in\nN$ be such that 
 \begin{equation*}
\tag{2}\bL(v)^{+}=i\bt+k\bt'\text{ and }(c^{\bL}_{x,y})^{+}=i'\bt+k'\bt'.
\end{equation*}  
 Then by the work above we know that $i'+k'\leq i+k$.
\bigskip

\begin{Claim}
If $y\in N^{\bL^{+}}_{\si'}$ then $\bP_{x,y}\bff_{x,y,z}\in\bA_{<0}$ whenever $z\notin N^{\bL^{+}}_{\si'}$.
\end{Claim}
\begin{proof}
Let $x=s_{N}\ldots s_{1}$ be a reduced expression of $x$. There exists $I=\{i_{1},\ldots,i_{p}\}\in\fI_{x,y}$ such that $z_{I}=z$. 
Assume that $z_{I}\notin N^{\bL^{+}}_{\si}$. Then there exist $\al\in\Phi^{\bL^{+}}$ and  $i_{k+1}-1>M>i_{k}$ such that 
$$s_{M}\ldots \hs_{i_{k}}\ldots \hs_{i_{1}}\ldots s_{1}yA_{0}\notin U^{\bL^{+}}_{\al}(A_{0})\text{ and } s_{M+1}\ldots \hs_{i_{k}}\ldots \hs_{i_{1}}\ldots s_{1}yA_{0}\in U^{\bL^{+}}_{\al}(A_{0})$$
and the unique hyperplane which separates these two alcoves is a wall of $N^{\bL^{+}}_{\si'}$. 
That means that in the sequence
$$I_{x_{p},y_{p}}\overset{\va_{p}}{\longrightarrow} I_{x_{p}s_{i_{k}},y_{p}}\subseteq I_{x_{p-1},y_{p-1}}\overset{\va_{p-1}}{\longrightarrow} I_{x_{p-1}s_{i_{p-1}},y_{p-1}}\ldots I_{x_{1},y_{1}}\overset{\va_{1}}{\longrightarrow} I_{x_{1}s_{i_{1}},y_{1}}\subseteq I_{x,y}$$
we have $H\in I_{x_{k+1}s_{i_{k+1}},y_{k+1}}$ but $H\notin  I_{x_{k},y_{k}}$. Further, if there is an hyperplane of direction $\al$ in  $I_{x_{k},y_{k}}$ then it can't be of maximal weight. Hence we have
$$c^{\bL}_{x_{k},y_{k}}\geq c^{\bL}_{x_{k+1}s_{i_{k+1}},y_{k+1}}+\bC_{\al}$$
where
$$\bC_{\al}=\begin{cases}
\bL_{\al} &\mbox{ if we are in Case 1.}\\
\bL_{\al}\text{ or }\bt-\bt' &\mbox{ if we are in Case 2.}\\
\bt\text{ or }\bt-\bt' &\mbox{ if we are in Case 3.}\\
\bt \text{ or }\bt' &\mbox{ if we are in Case 4.}\\
\end{cases}$$
Next, by repetitive use of Lemmas \ref{Lem1} and \ref{Lem2} we get the following inequalities 
\begin{align*}
c^{\bL}_{x,y}&\geq c^{\bL}_{x_{1}s_{i_{1}},y_{1}}\geq  c^{\bL}_{x_{1},y_{1}}+\bL(s_{i_{1}})\\
&\geq c^{\bL}_{x_{2}s_{i_{2}},y_{2}}+\bL(s_{i_{1}})\geq c^{\bL}_{x_{2},y_{2}}+\sum^{2}_{\ell=1} \bL(s_{i_{\ell}})\\
&\ldots \\
&\geq c^{\bL}_{x_{k}s_{i_{k}},y_{k}}+\sum^{k-1}_{\ell=1} \bL(s_{i_{\ell}})\geq c^{\bL}_{x_{k},y_{k}}+\sum^{k}_{\ell=1} \bL(s_{i_{\ell}})\\
&\geq c^{\bL}_{x_{k+1}s_{i_{k+1}},y_{k+1}}+\sum^{k}_{\ell=1} \bL(s_{i_{\ell}})+\bC_{\al}\\
&\geq c^{\bL}_{x_{k+1},y_{k+1}}+\sum^{k+1}_{\ell=1} \bL(s_{i_{\ell}})+\bC_{\al}\\
&\geq \ldots \geq c^{\bL}_{x_{p},y_{p}}+\sum_{\ell=1}^{p}\bL(s_{i_{p}})+\bC_{\al}.
\end{align*}
Hence, we get $(\bc^{\bL}_{x,y})^{+}\geq \deg^{+}(\bff_{x,y,z})+ \bC^{+}_{\al}$. \\

\noindent
If we are in Case 1 or in Case 2 with $\bC_{\al}=\bL_{\al}$ then we have, using (1)
$$\deg^{+}(\bff_{x,y,z})-\bL(v)^{+}\leq (c^{\bL}_{x,y})^{+}-\bL(v)^{+}-\bC_{\al}^{+}\leq- \bL_{\al}.$$
But we know that $\va_{1}(\bL_{\al})<0$ hence
$\va_{1}(\deg^{+}(\bP_{x,y}\bff_{x,y,z}))<0$ and $\bP_{x,y}\bff_{x,y,z}$ lies in $\bA_{<0}$; see Remark \ref{neg}.\\

\noindent
If we are in Case 2 and $C_{\al}=\bt-\bt'$, then, we must have $\bt\in \bar{S}^{+}$ and $\bt'\notin \bar{S}^{\circ}$. Therefore we have $C_{\al}^{+}=\bt$ and we can conclude as above.  \\

\noindent
If we are in Case 3 then $C_{\al}\geq \bt-\bt'$ and we get
$$\deg^{+}(\bff_{x,y,z})-\bL(v)^{+}\leq (c^{\bL}_{x,y})^{+}-\bL(v)^{+}-\bC_{\al}^{+}\leq (-1,0,1).$$
But since we are in Case 3, this implies that $\bP_{x,y}\bff_{x,y,z}\in\bA_{<0}$; see Remark \ref{diff-case}.\\
 
 \noindent
Finally, if we are in Case 4 then $\bC_{\al}\geq \bt'$. Using (2), we get
\begin{align*}
\deg^{+}(\bff_{x,y,z})-\bL(v)^{+}&\leq (c^{\bL}_{x,y})^{+}-\bL(v)^{+}-\bt'\\
&\leq  (i'+k') \bt-(i+k) \bt'-\bt'\\
&=(i'+k',0,-(i+k)-1)
\end{align*}
and since $(i+k)\geq (i'+k')\geq 0$ we have $\va_{1}((i'+k',0,-(i+k)-1))<0$. The result follows; see Remark \ref{neg}.
\end{proof}

\begin{Claim}
If $y\notin \cb^{\bL^{+}}_{\min}$ then  $\bP_{x,y}\bff_{x,y,z}\in\bA_{<0}$. 
\end{Claim}
\begin{proof}
Since $y\notin \cb^{\bL^{+}}_{\min}$ we have  $y\in U^{\bL^{+}}_{\al_{m}}(A_{0})$ for some $k+1\leq m\leq n$. Then
$$c^{\bL}_{x,y}(\al_{m})=
\begin{cases}
0&\mbox{ if we are in Case 1.}\\
0 \text{ or }\bt' &\mbox{ if we are in Case 2.}\\
0 \text{ or }\bt'  &\mbox{ if we are in Case 3.}\\
0 &\mbox{ if we are in Case 4.}\\
\end{cases}$$
We have 
$$(c^{\bL}_{x,y})^{+}\leq \sum^{n}_{i=m+1} \bL_{\al_{i}}+c_{x,y}^{\bL}(\al_{m})^{+}.$$

\noindent
If we are in Case (1) or in Case 2 with $c^{\bL}_{x,y}(\al_{m})=0$,  then $\bL(s_{i})=\bL_{\al_{i}}$ for all $1\leq i\leq n$ (see Lemma \ref{Lplus-spe}). Hence
$$(c^{\bL}_{x,y})^{+}-\bL(v)^{+}\leq -\bL_{\al_{m}}$$
and the result follows using $(\ast)$.\\

\noindent
If we are in Case (2) with $c^{\bL}_{x,y}(\al_{m})=\bt'$. Then 
$$(c^{\bL}_{x,y})^{+}-\bL(v)^{+}\leq -\bL_{\al_{m}}+(c^{\bL}_{x,y})^{+}= -\bt.$$
and the result follows using $(\ast)$ and Remark \ref{neg}.\\

\noindent
If we are in Case (3), then $c^{\bL}_{x,y}(\al_{m})\leq \bt'$. Then $\bL(s_{i})=\bL_{\al_{i}}$ for all $1\leq i\leq n$ and 
$$(c^{\bL}_{x,y})^{+}-\bL(v)^{+}\leq -\bL_{\al_{m}}+(c^{\bL}_{x,y})^{+}\leq  -\bt+\bt'.$$
and the result follows; ; see Remark \ref{diff-case}.\\

\noindent
Finally, if we are in Case 4, then using (2) we must have $i+k>i'+k'$ since $y\in U^{\bL^{+}}_{\al_{m}}(A_{0})$. The result follows, arguying in a similar fashion as at the end of the previous Claim.  
\end{proof}

\noindent
The theorem follows easily from the two claims and the expression: 
$$T_{x}C_{w}= T_{xw}+\underset{b_{\si'}<b_{\si}}{\sum_{y\in N_{\si'}}}P_{y,w}T_{x}T_{y}+\sum_{y\notin\cb^{\bL^{+}}_{\min}} P_{y,w}T_{x}T_{y}\quad\mod \bH_{<0}$$
\end{proof}

\subsection{Proof of Conditions {\bf I1}--{\bf I5}}

Condition {\bf I1} it is clear. Condition {\bf I2} is a direct consequence of Lemma \ref{dot-property}. Condition {\bf I3} follows from the fact that $\cb^{\bL^{+}}_{\min}$ is a disjoint union of the sets $N^{\bL^{+}}_{\si}$. Condition {\bf I5} is Theorem \ref{bound}. Hence we only to prove {\bf I4}, that is, we need to show that the $\bA$-module 
 $$\cM:=\sg \bT_{x}\bC_{u}|\ u\in U^{\leq}_{\si},\ x\in X_{u}\sd_{\cA}\subseteq \cH$$is a left ideal of $\cH$.
Let $u=a_{u}\hw_{\la_{\si}}b_{\si}\in U_{\la,z}$ and $x\in X_{\la_{\si}}$. It is enough to show that $T_{s}T_{x}C_{u}\in \cM$ for all $s\in S$. There is 3 cases to consider:
\begin{enumerate}
\item $sx>x$ and $sx\in X_{\la_{\si}}$;
\item $sx<x$ and $sx\in X_{\la_{\si}}$;
\item $sx>x$ and $sx\notin X_{\la_{\si}}$.
\end{enumerate} 
The result is clear in the first two cases since we have respectively
$$\bT_{s}\bT_{x}\bC_{u}=\bT_{sx}\bC_{u}\in  \cM$$
and 
$$\bT_{s}\bT_{x}C_{u}=(\bT_{sx}+(v^{\bL(s)}-v^{-\bL(s)})\bT_{x})\bC_{u} \in \cM$$
Assume that we are in the third case. Then by Deodhar's lemma (see \cite[Lemma 2.1.2]{Geck-Pfeiffer}), there exists $s'\in S_{\la_{\si}}$ such that $sx=xs'$. If $s'u<u$ we get 
$$\bT_{s}\bT_{x}\bC_{u}=\bT_{sx}\bC_{u}=\bT_{x}\bT_{s'}\bC_{u}=v^{\bL(s')}\bT_{x}\bC_{u}\in  \cM$$
 as required. Assume $s'u>u$. Note that this implies that $s'\in S^{\circ}$ since for all $s'\in S^{+}$ we have $s'u<u$. Then
 $$\bT_{s}\bT_{x}\bC_{u}=\bT_{x}\bT_{s'}\bC_{u}=\bT_{x}\big(\bC_{s'u}+\sum_{z<w} \bM^{s'}_{z,u}\bC_{z}\big).$$
In the above sum, we know that the term $\bT_{x}\bC_{s'u}\in \cM$. If $\bM_{z,u}^{s'}\neq 0$ then by Lemma \ref{KLasym}, we must have $z\in\cb^{\bL^{+}}_{\min}$ which in turn implies that by Lemma \ref{dot-property}, that either $z=a_{z}\hw_{\la_{\si}}b_{\si}$ 
with $a_{z}<a_{u}$ (in which case $\bT_{z}\bC_{u}\in \cM$) or $z=a_{z}\hw_{\la_{\si'}}b_{\si'}$  with $b_{\si'}<b_{\si}$. From there, the result follows by an easy induction on the length of~$b_{\si}$.


\section{Proof of Theorem \ref{main2}}
Let $\Ga$ be a totally ordered group abelian group. In this section we study the relation between $\bH=(W,S,\bL)$ as define in the previous section and $\cH=\cH(W,S,L)$ where $L\in\text{Weight}(W,\Ga)$. The element of $\bH$ and $\bA$ will be written with a bold symbols.

\subsection{Specialisation}
Let $L:W\rightarrow \Ga$ be weight function. Then the map $\theta^{\bL}_{\Ga}:\bGa\rightarrow \Ga$ which sends $\bL(s)$ to $L(s)$ is a group homomorphism. Further, this homomorphism induces a morphism of $\nZ$-algebras $\theta_{\bA}^{L}: \bA\rightarrow \cA$ which sends $v^{\bL(s)}$ to $v^{L(s)}$. If $\cH$ is viewed as a $\bA$-algebra through $\theta^{L}_{\bA}$, then there is a unique morphism of $\bA$-algebras $\theta^{L}_{\bH}:\bH\rightarrow \cH$ such that $\theta(\bT_{x})=T_{x}$ for all $x\in W$.

\bigskip

We recall the some result of \cite{jeju1,Geck-F4}. 
Let $N\in \nN$ and let $X_{N}=\{z\in W\mid \ell(z)\leq N\}$. We now define three subsets $\bGa_{+}^{1}(N)$, $\bGa_{+}^{2}(N)$, $\bGa_{+}^{3}(N)\subset \bGa$. First, let $\bGa_{+}^{1}(N)$ be the set of all elements $\gamma >0\in\bGa$ such that $v^{-\ga}$ occurs with a non-zero coefficient in a polynomial $\mathbf{P}_{z_{1},z_{2}}$ for some $z_{1}<z_{2}\in X_{N}$. Next for any $z_{1},z_{2}$ in $X_{N}$ such that $\mathbf{M}_{z_{1},z_{2}}^{s}\neq 0$ for some $s$, we write $\mathbf{M}_{z_{1},z_{2}}^{s}=n_{1}v^{\g_{1}}+\ldots+n_{\ell}v^{\g_{\ell}}$ where $0\neq n_{i}\in \mathbb{Z}$, $\gamma_{i}\in \bGa$ and $\g_{i}-\g_{i-1}>0$ for $2\leq i\leq \ell$. Let $\bGa_{+}^{2}(N)$ be the set of all elements  $\g_{i}-\g_{i-1}>0$ arising in this way, for any $z_{1},z_{2}\in X_{N}$ and $s\in S$. Finally let $\bGa_{+}^{3}(N)$ be the set of all elements $\gamma>0 \in \bGa$ such that $v^{-\g}$ occurs with a non-zero coefficient in a polynomial of the form 
$$\underset{z;z_{1}\leq z<z_{2};sz<z}{\sum}\mathbf{ P}_{z_{1},z}\mathbf{ M}_{z,z_{2}}^{s}-v_{s}\mathbf{ P}_{z_{1},z_{2}}$$
for some $z_{1},z_{2}\in X_{N}$ and $s\in S$.  We set $\bGa_{+}(N)=\bGa_{+}^{1}(N)\cup\bGa_{+}^{2}(N)\cup \bGa_{+}^{3}(N)$. 

\begin{Prop}(see \cite[Proposition 3.3]{jeju1})
\label{specialisation}
Let $L:W\rightarrow \Ga$ be a weight function such that the ring homomorphism $\theta_{L}$
satisfies the condition 
\begin{equation*}
\theta^{L}_{\bGa}(\bGa_{+}(N))\subseteq \{\g\mid \g>0\}. \tag{$\ast$}
\end{equation*} 
Then, for all $x,y\in X_{N}$, we have $\theta^{L}_{\bA}(\bP_{x,y})=P_{x,y}$ and $\theta^{L}_{\bA}(\bM^{s}_{x,y})=M^{s}_{x,y}$. In particular $\theta^{L}_{\bH}(\bC_{x})=C_{x}$.  
\end{Prop}

The following Lemma is a straightforward generalisation of \cite[Lemma 3.4]{Geck-F4}.

\begin{Lem}
\label{gen-bound}
We have
$$\bGa_{+}(N)\subset\{(\g_{1},\ldots,\g_{|S|})\in\nZ^{|S|}\mid -N\leq\g_{i}\leq N\}$$
\end{Lem}

\bigskip

We now give an outline of the proof of Theorem \ref{main2}. First, note that the proof of Theorem \ref{I1-I5} in  the previous section only involved elements of bounded length, say by $N_{0}\in \nN$.  Let $(W,S)$ be an irreducible affine Weyl group and assume that  $S=S^{+}\cup S^{\circ}$ is such that no element of $S^{+}$ is conjugate to an element of $S^{\circ}$. Let $X\in \cP_{+}(\nZ[\bar{S}])$ as in Section \ref{generic} and $\bGa=\nZ[\bar{S}]/(X\cap (-X))$. Then, using $N_{0}$ and the previous lemma, we will determine a set of weight functions $\cC$ (which will correspond at the end to a union of chambers) such that for all $L\in \cC$ we have $$\theta^{L}_{\bGa}(\bGa_{+}(N_{0}))\subseteq \{\g\mid \g>0\}.$$ 
For all such weight functions, we can apply the same proof as before to $(W,S,L)$ by Proposition \ref{specialisation} and we get that  $N^{\bL^{+}}_{\si}$, for all $\si\in\Om^{\bL^{+}}_{0}$ is a union of cells of $(W,S,L)$. But then, it will be easy to see that $N^{\bL^{+}}_{\si}=N^{L'}_{\si}$ for all weight function $L'$ which takes some zero values and which lies in $\ov{\cC}$.  
Finally by changing the sets $S^{+}$ and $S^{\circ}$ (and also the subset $X$ in type $\tilde{\Crm}$), we will eventually cover all the cases and determine the constants such that Theorem \ref{main2-BFG} and Theorem \ref{main2-C} holds. The constants we are determining are nowhere near the best we can find as one can see by looking at the essential hyperplanes in type $\tilde{\Crm}_{2}$ in Remark \ref{problem}.

\subsection{Affine Weyl group of type $\tilde{\Brm}_r$, $\tilde{\Frm}_4$ or $\tilde{\Grm}_2$.} 
\label{proof-BFG}
We keep the notation of Section \ref{typeBFG}. Let $W$ be an irreducible affine Weyl group of type $\tilde{\Brm}_r$, $\tilde{\Frm}_4$ or $\tilde{\Grm}_2$. \\

First let $\bar{S}^{+}=\{\bs\}$ and $\bar{S}^{\circ}=\{\bt\}$. Let $\Phi\in \cP_{+}(\nZ[\bar{S}])$. Then $\bGa=\nZ[\bar{S}]$ and the order is the lexicographic order (see Example \ref{lex}).
\begin{Claim}
Let $L\in\text{Weight}(W,\Ga)$ be such that $L(\bs)>N_{0}\cdot L(\bt)$. Then 
$$\theta^{L}_{\bGa}(\bGa_{+}(N_{0}))\subset  \{v^{\g}\mid \g>0\}$$
\end{Claim}
\begin{proof}
Since the order on $\bGa$ is the lexicographic order, we must have
$$\bGa_{+}(N_{0})\subset \{(i,j)\mid i> 0, -N_{0}\leq i,j,k\leq N_{0}\}\cup \{(0,j)\mid j>0\}.$$
Let $i>0$ and $-N_{0}\leq j\leq N_{0}$. Then
$$i\cdot \bL(s)+j\cdot \bL(t)>iN_{0}\cdot L(\bt) -N_{0}\cdot L(\bt)\geq 0.$$
The result follows.
\end{proof}
Thus Condition $(\ast)$ in Proposition \ref{specialisation} holds for all weight functions satisfying the hypothesis of the lemma. Therefore we get that $N^{\bL^{+}}_{\si}$ is a union of cells of $(W,S,L)$.  But one can easily see that for all weight functions $L^{+}:W\rightarrow \Ga$ such that $L^{+}(\bs)>0$ and $L^{+}(\bt)=0$ we have  $\Om^{\bL^{+}}_{0}=\Om^{L^{+}}_{0}$ and $N^{L^{+}}_{\si}=N^{\bL^{+}}_{\si}$.

\bigskip

Next let $\bar{S}^{+}=\{\bt\}$ and $\bar{S}^{\circ}=\{\bs\}$. Then we get that the order on $\bGa$ is as follows
$$(i,j)<(i',j')\Longleftrightarrow j<j' \text{ or }(j=j'\text{ and } i<i').$$
Arguying as above, we get that for all weight functions  $L\in \text{Weight}(W,\Ga)$ such that $L(\bt)>N_{0}\cdot L(\bs)$ and all weight functions $L^{+}:W\rightarrow \Ga$ such that $L^{+}(\bt)>0$ and $L^{+}(\bs)=0$, the sets $N^{L^{+}}_{\si}$ are union of left cells of $(W,S,L)$. 

\bigskip

Finally, putting all this together, we get Theorem \ref{main2} holds for the finite set of rational hyperplanes $\fH(N_{0},1/N_{0})$. 

\subsection{Affine Weyl group of type $\tilde{\Crm}$}
\label{proof-C} 
We keep the notation of Section \ref{typeC}. We have $\bar{S}=\{\bt,\bs,\bt'\}$. \\

\begin{Claim}
\label{84}
Let $\bar{S}^{+}=\{\bt\}$ and $\bar{S}^{\circ}=\{\bs,\bt'\}$ and let $L\in\text{Weight}(W,\Ga)$ be  such that $L(\bt)>N_{0}\cdot L(\bs)+N_{0}\cdot L(\bt')$. Then, there exists $\Phi=(\va_{1},\ldots,\va_{d})\in  \cP_{+}(\nZ[\bar{S}])$ such that 
$$\theta^{L}_{\bGa}(\bGa_{+}(N_{0}))\subset  \{v^{\g}\mid \g>0\}$$
where $\bGa$ is the totally ordered abelian group associated to $\Pos(\Phi)$. 
\end{Claim}
\begin{proof}
Since $\bar{S}^{+}=\{\bt\}$ we set $\va_{1}=\bt^{\ast}$. Hence
$$\ker(\va_{1})=\{(0,j,k)\mid j,k\in\nR\}.$$
Now we want to find a linear map $\va_{2}:\ker(\va_{1})\rightarrow \nR$ defined by $\va_{2}((0,j,k))=bj+ck$ (where $b,c\geq 0$) such that 
the following property holds for all $-N_{0}\leq i,j\leq N_{0}$:
\begin{equation*}
\tag{$\dag$}bj+ck>0\text{ then }L(\bs)j+L(\bt)k>0
\end{equation*}

To do so we proceed as in \cite[\S 3]{Geck-F4}. Set $\cE:=\{x\in\nQ\mid x=\pm\frac{k}{j}\text{ where } j,k\neq 0 \text{ and } -N_{0}\leq j,k\leq N_{0}\}$ and write $\cE=\{x_{1},\ldots,x_{n}\}$ where $x_{1}<x_{2}<\ldots<x_{n}$. We set $x_{0}=0$ and $x_{n+1}=+\infty$. Let  $b,c\geq 0$ be integers such that  
$$x_{k}=\frac{b}{c}=\max\{r\in\cE\mid L(\bs)\geq rL(\bt')\}.$$
Note that we must have $x_{k+1}L(\bt')>L(\bs)$. 
Then we claim that property $(\dag)$ holds. 
Let $-N_{0}\leq j,k\leq N_{0}$ be such that  $bj+ck>0$. If $j>0$ and $k<0$ then $b/c>-k/j$ and we have
$$L(\bs)>-\frac{k}{j}L(\bt')\text{ that is } jL(\bs)+kL(\bt')>0$$
as required. If $j<0$ and $k>0$ (in this case we have $x_{k+1}<\infty$) then $b/c<-k/j$ but this forces $x_{k+1}\leq -k/j$. Hence
$$-\frac{k}{j}L(\bt')>L(\bs)\text{ that is } jL(\bs)+kL(\bt')>0$$
as required. \\
Finally we set $\va_{2}:\ker(\va_{1})\rightarrow \nR$ by $\va_{2}((0,j,k))=bj+ck$ where $b,c\geq 0$ are chosen as above. Then  $\Pos(\va_{1},\va_{2})\in \cP_{+}(\nZ[\bar{S}])$ and we have 
$$\bGa_{+}(N_{0})\subset \{(i,j,k)\mid i> 0, -N_{0}\leq i,j,k\leq N_{0}\}\cup \{(0,j,k)\mid L(\bs)j+L(\bt')k>0\}.$$
The result follows easily from our assumptions on $L$. 
\end{proof}
Let $L\in\text{Weight}(W,\Ga)$ and $\Phi$ as in the previous claim. By Proposition \ref{specialisation}, we know that $N_{\si}^{\bL^{+}}$ (for all $\si\in\Om^{\bL^{+}}_{0}$) is a union of left cells in $(W,S,L)$. But one can see that  $N_{\si}^{\bL^{+}}=N_{\si}^{L^{+}}$ for all  $L^{+}\in \text{Weight}(W,\Ga)$ such that  $L^{+}(\bt)>0$ and $L^{+}(\bs)=L^{+}(\bt')=0$ and all $\si\in \Om^{\bL^{+}}_{0}=\Om_{0}^{L^{+}}$. 

\begin{Rem}
If we set $m_{3}=N_{0}$, the above claim implies that Theorem 6.9 holds for all positive weight functions lying in $\cC_{3}\cup\cC_{4}\cup (\bar{\cC_{3}}\cap\bar{\cC_{4}})$ and all non-negative weight functions lying  in $H_{\bs}\cap H_{\bt'}$.
\end{Rem}

\begin{Claim}
\label{85}
Let $\bar{S}^{+}=\{\bs\}$ and $\bar{S}^{\circ}=\{\bt,\bt'\}$ and let $L\in\text{Weight}(W,\Ga)$ be  such that $L(\bs)>N_{0}\cdot L(\bt)+N_{0}\cdot L(\bt')$. Then, there exists $\Phi=(\va_{1},\ldots,\va_{d})\in  \cP_{+}(\nZ[\bar{S}])$ such that 
$$\theta^{L}_{\bGa}(\bGa_{+}(N_{0}))\subset  \{v^{\g}\mid \g>0\}$$
where $\bGa$ is the totally ordered abelian group associated to $\Pos(\Phi)$.
\end{Claim}
\begin{proof}
Since $\bar{S}^{+}=\{\bs\}$ and we set $\va_{1}=\bt^{\ast}$. Hence
$$\ker(\va_{1})=\{(i,0,k)\mid i,k\in\nR\}.$$
Arguing as in the proof of the previous claim, there exist integers $a,c\geq 0$ such that if we define $\va_{2}:\ker(\va_{1})\rightarrow \nR$ by $\va_{2}((i,0,k))=ai+ck$ then $\va_{2}((i,0,k))>0$ implies $L(\bt)i+L(\bt')k>0$. For such $\va_{2}$ we have $\Pos(\va_{1},\va_{2})\in \cP_{+}(\nZ[\bar{S}])$ and 
$$\bGa_{+}(N_{0})\subset \{(i,j,k)\mid j> 0, -N_{0}\leq i,j,k\leq N_{0}\}\cup \{(i,0,k)\mid \bL(\bt)j+\bL(\bt')k>0\}.$$
The result follows easily from our assumptions on $L$. 
\end{proof}
Let $L\in\text{Weight}(W,\Ga)$ and $\Phi$ as in the previous claim. Arguying as above,we get that $N_{\si}^{L^{+}}$ is a union of left cells of $(W,S,L)$ for all $L^{+}\in \text{Weight}(W,\Ga)$ such that $L^{+}(\bs)>0$ and $L^{+}(\bt)=L^{+}(\bt')=0$.

\begin{Rem}
If we set $m_{6}=1/N_{0}$, the above claim implies that Theorem 6.9 holds for all positive weight functions lying in $\cC_{6}\cup (\bar{\cC_{6}}\cap\bar{\cC'_{6}})$
and all non-negative weight functions lying in $H_{\bt}\cap H_{\bt'}$.
\end{Rem}

\begin{Claim}
\label{86}
Let $\bar{S}^{+}=\{\bt,\bt'\}$ and $\bar{S}^{\circ}=\{\bs\}$ and let $L\in\text{Weight}(W,\Ga)$ be  such that $L(\bt)>N^{2}_{0}\cdot L(\bs)$, $L(\bt')>N^{2}_{0}\cdot L(\bs)$ and $L(\bt)-L(\bt')>N_{0}\cdot L(\bs)$. Then, there exists $\Phi=(\va_{1},\ldots,\va_{d})\in  \cP_{+}(\nZ[\bar{S}])$ such that 
$$\theta^{L}_{\bGa}(\bGa_{+}(N_{0}))\subset  \{v^{\g}\mid \g>0\}.$$
where $\bGa$ is the totally ordered abelian group associated to $\Pos(\Phi)$.
\end{Claim}
\begin{proof}
Since $\bar{S}^{+}=\{\bt,\bt'\}$ we set $\va_{1}=\bt^{\ast}+\bt'^{\ast}$. Hence
$$\ker(\va_{1})=\{(i,j,-i)\mid i,j\in\nR\}.$$
We define $\va_{2}:=\bt^{\ast}$ and $\va_{3}=\bs^{\ast}$. Then $\Pos(\va_{1},\va_{2},\va_{3})\in \cP_{+}(\nZ[\bar{S}])$ and we have 
\begin{align*}
\bGa_{+}(N_{0})\subset &\{(i,j,k)\mid i+k>0, -N_{0}\leq i,j,k\leq N_{0}\}\\
&\cup \{(i,j,-i)\mid i>0, -N_{0}\leq i,j\leq N_{0}\}\\
&\cup \{0,j,0\mid j>0\}.
\end{align*}
The result follows from our assumptions on $L$. Note that the $N^{2}_{0}$ in the hypothesis of the lemma comes from the fact that we need to have 
$$\theta_{\bGa}^{L}(-(N_{0}-1),-N_{0},N_{0})>0.$$
That is 
$$L(\bt')-(N_{0}-1)(L(\bt)-L(\bt'))-N_{0}L(\bs)>0.$$
The last inequality will hold whenever $L(\bt')>N_{0}^{2}L(\bs)$ and $L(\bt)-L(\bt')>N_{0}\cdot~L(\bs)$. 
\end{proof}
Let $L\in\text{Weight}(W,\Ga)$ and $\Phi$ as in the previous claim. Arguying as above,we get that $N_{\si}^{L^{+}}$ is a union of left cells of $(W,S,L)$ for all   $L^{+}\in \text{Weight}(W,\Ga)$ such that $L^{+}(\bt)>L^{+}(\bt')>0$ and $L^{+}(\bs)=0$. 
\begin{Rem}
If we set $m_{1}=m_{2}=N^{2}_{0}$ and $m_{5}=N_{0}$, the above claim implies that Theorem 6.9 holds for all positive weight functions lying  in $\cC_{2}\cup\cC_{3}\cup (\bar{\cC_{2}}\cap\bar{\cC_{3}})$
and all non-negative weight functions in $H_{\bs}\cap H_{\bt'}$.
\end{Rem}

\begin{Claim}
\label{87}
Let $\bar{S}^{+}=\{\bt,\bt'\}$ and $\bar{S}^{\circ}=\{\bs\}$ and let $L\in\text{Weight}^{+}(W,\Ga)$ be  such that $L(\bt)>N^{2}_{0}\cdot L(\bs)$, $L(\bt')>N^{2}_{0}\cdot L(\bs)$. Then, there exists $\Phi=(\va_{1},\ldots,\va_{d})\in  \cP_{+}(\nZ[\bar{S}])$ such that  $(1,j_{0},-1)>0$ for some $j_{0}\in\nN$ and 
$$\theta^{L}_{\bGa}(\bGa_{+}(N_{0}))\subset  \{v^{\g}\mid \g>0\}$$
where $\bGa$ is the totally ordered abelian group associated to $\Pos(\Phi)$.
\end{Claim}
\begin{proof}
Since $\bar{S}^{+}=\{\bt,\bt'\}$ we set $\va_{1}=\bt^{\ast}+\bt'^{\ast}$. Hence
$$\ker(\va_{1})=\{(i,j,-i)\mid i,j\in\nR\}.$$
Arguing as in the proof of Claim \ref{84}  there exist integers $d,b\geq 0$ (with $b>0$ so that $(1,j_{0},-1)>0$ for some $j_{0}$) such that if we define $\va_{2}:\ker(\va_{1})\rightarrow \nR$ by $\va_{2}((i,j,-i))=di+bj$ then $\va_{2}((i,j,-i))>0$ implies $(L(\bt)-L(\bt'))i+L(\bs)j>0$. 
For such $\va_{2}$ we have $\Pos(\va_{1},\va_{2})\in \cP_{+}(\nZ[\bar{S}])$ and 
\begin{align*}
\bGa_{+}(N_{0})\subset &\{(i,j,k)\mid i+k>0, -N_{0}\leq i,j,k\leq N_{0}\}\\
&\cup \{(i,j,-i)\mid (L(\bt)-L(\bt'))\cdot i+L(\bs)\cdot j>0\}.
\end{align*}

The result follows from our assumptions on $L$. 
\end{proof}
Let $L\in\text{Weight}(W,\Ga)$ and $\Phi$ as in the previous claim (so that we are in Case 4). Arguying as above,we get that $N_{\si}^{L^{+}}$ is a union of left cells of $(W,S,L)$ for all   $L^{+}\in \text{Weight}(W,\Ga)$ such that $L^{+}(\bt)=L^{+}(\bt')>0$ and $L^{+}(\bs)=0$. 

\begin{Rem}
If we set $m_{1}=m_{2}=N^{2}_{0}$, then above claim implies that Theorem 6.9 holds for all positive weight functions lying  in $\cC_{1}\cup (\bar{\cC_{1}}\cap\bar{\cC'_{1}})$
and all non-negative weight functions in $H_{\bs}\cap H_{\bt-\bt'}$.

\end{Rem}

Finally we still have to consider the case $\bar{S}^{+}=\{\bt,\bs\}$ and $\bar{S}^{\circ}=\{\bt'\}$. However in this case, there is nothing to prove since  for all $L,L^{+}\in\text{Weight}(W,\Ga)$ such that 
\begin{itemize}
\item $L(\bs),L(\bt)>L(\bt')$
\item  $L^{+}(\bt)=L(\bt)$, $L^{+}(\bs)=L(\bs)$ and $L^{+}(\bt')=0$
\end{itemize}
we have $\cb^{L}_{\min}=\cb^{L^{+}}_{\min}$. Indeed, in both cases, we have $\WC^{\max}=\{w_{0}\}$ where $w_{0}$ is the longest element of the group generated by $t,s_{1},\ldots,s_{n-1}$. 

\bigskip

The proof of Theorem \ref{main2-C} for type $C$ is now complete.


\begin{thebibliography}{10}

\bibitem{semicontinuity}
C.~Bonnaf\'e.
\newblock Semicontinuity properties of Kazhdan-Lusztig cells.
\newblock {\em New Zealand Journal of Mathematics}, Volume {\bf 39} (2009),  
Pages 171--192. 

\bibitem{Bonnafe-Dyer}
C.~Bonnaf\'e and M.~Dyer.
\newblock Semidirect product decomposition of {C}oxeter groups.
\newblock {\em Communications in Algebra, Volume {\bf 38}, Pages 1549--1574},
  2010.

\bibitem{Bourbaki}
N.~Bourbaki.
\newblock {\em Groupes et alg\`ebres de Lie, Chap 4--6}.
\newblock Hermann, Paris, 1968; Masson, Paris, 1981.

\bibitem{Bremke}
K.~Bremke.
\newblock On generalized cells in affine {W}eyl groups.
\newblock {\em J. of Algebra $\mathbf{191}$, 149--173}, 1997.

\bibitem{Geck-induction}
M.~Geck.
\newblock On the induction of {K}azhdan-{L}usztig cells.
\newblock {\em Bull. London Math. Soc. {\bf 35}, 608--614}, 2003.

\bibitem{Geck-F4}
M.~Geck.
\newblock Computing {K}azhdan-{L}usztig cells for unequal parameters.
\newblock {\em J. of Algebra $\mathbf{281}$, 342--365}, 2004.

\bibitem{Geck-Pfeiffer}
M.~Geck and G.~Pfeiffer.
\newblock {\em Characters of Finite {C}oxeter Groups and Iwahori--Hecke
  Algebras}.
\newblock London Mathematical Society Monographs. New Series, 21. The Clarendon
  Press, Oxford University Press, New York, 2000.

\bibitem{jeju1}
J.~Guilhot.
\newblock On the determination of {K}azhdan-{L}usztig cells for affine {W}eyl
  group with unequal parameter.
\newblock {\em J. of Algebra $\mathbf{318}$, 893--917}, 2007.

\bibitem{jeju2}
J.~Guilhot.
\newblock On the lowest two-sided cell of an affine {W}eyl groups.
\newblock {\em Represent. Theory {\bf 12}, 327--345.}, 2008.

\bibitem{jeju3}
J.~Guilhot.
\newblock Generalized induction of {K}azhdan-{L}usztig cells.
\newblock {\em Ann. Inst. Fourier, {\bf 59} p. 1385-1412}, 2009.

\bibitem{jeju4}
J.~Guilhot.
\newblock Kazhdan-{L}usztig cells in affine {W}eyl groups of rank 2.
\newblock {\em Int Math Res Notices Vol. 2010 3422-3462}, 2010.

\bibitem{comp}
J.~Guilhot.
\newblock Some computations about {K}azhdan-{L}usztig cells in affine {W}eyl
  groups of rank 2.
\newblock {\em available at http://arxiv.org/abs/0810.5165}, 2010.

\bibitem{Lusztig-1980}
G.~Lusztig.
\newblock Hecke algebras and {J}antzen's generic decomposition patterns.
\newblock {\em Adv. in Math. $\mathbf{37}$, 121--164}, 1980.

\bibitem{bible}
G.~Lusztig.
\newblock {H}ecke algebras with unequal parameters.
\newblock {\em CRM Monograph Series $\mathbf{18}$, Amer. Math. Soc.,
  Providence, RI}, 2003.

\bibitem{Shi-1987}
J.-Y. Shi.
\newblock A two-sided cell in an affine {Weyl} group.
\newblock {\em J. London Math. Soc. $\mathbf{36}$, 407--420}, 1987.

\bibitem{Shi-1988}
J.-Y. Shi.
\newblock A two-sided cell in an affine {Weyl} group {I}{I}.
\newblock {\em J. London Math. Soc. $\mathbf{37}$, 253--264}, 1988.

\bibitem{Xi-book}
N.~Xi.
\newblock Representations of affine {H}ecke algebras.
\newblock {\em Lecture Notes in Math., Springer-Verlag, Berlin, vol. 1587},
  1994.

\end{thebibliography}
\end{document}